\documentclass{amsart}
\usepackage{a4wide}
\usepackage{amsmath,amscd}%
\usepackage{amsfonts}%
\usepackage{amssymb}%
\usepackage{graphicx}
\usepackage[usenames,dvipsnames]{color}
\usepackage{subfigure}
\usepackage{enumerate}
\usepackage[noprefix]{nomencl}  
\makenomenclature 
\usepackage[heads=vee]{diagrams}
\usepackage[all,cmtip]{xy}
\usepackage[sort&compress]{natbib}

\newcommand{\kdifform}[2]{#1^{(#2)}} 
\newcommand{\difform}[1]{#1} 
\newcommand{\kdifformh}[2]{#1^{(#2)}_{h}} 
\newcommand{\kchain}[2]{\mathbf{#1}_{(#2)}} 
\newcommand{\kcochain}[2]{\mathbf{#1}^{(#2)}} 
\newcommand{\ccomplex}[1]{{#1}} 
\newcommand{\reconstruction}{\mathcal{I}} 
\newcommand{\reduction}{\mathcal{R}} 
\newcommand{\kformspace}[1]{\Lambda^{#1}} 
\newcommand{\kformspaceh}[1]{\Lambda_h^{#1}} 
\newcommand{\kformspacedomain}[2]{\Lambda^{#1}(#2)} 
\newcommand{\kformspacedomainh}[2]{\Lambda_h^{#1}(#2)} 
\newcommand{\kchainspace}[1]{C_{#1}} 
\newcommand{\kchainspacedomain}[2]{C_{#1}(#2)} 
\newcommand{\kcochainspace}[1]{C^{#1}} 
\newcommand{\kcochainspacedomain}[2]{C^{#1}(#2)} 

\newcommand{\incidenceboundary}[2]{\mathsf{E}_{(#1,#2)}} 
\newcommand{\incidencederivative}[2]{\mathsf{E}^{(#1,#2)}} 

\newcommand{\projection}{\pi_{h}} 
\newcommand{\dualprojection}{\tilde{\pi}_{h}} 
\newcommand{\coprojection}{\pi^\star_{h}} 
\newcommand{\dualcoprojection}{\tilde{\pi}^\star_{h}} 
\newcommand{\mapping}{\Phi} 
\newcommand{\pullback}{\Phi^{\star}} 
\newcommand{\pullbacks}[1]{#1^{\star}} 
\newcommand{\pushforward}{\Phi_{\star}} 

\newcommand{\manifold}[1]{\mathcal{#1}} 
\newcommand{\tangentspace}[1]{T_{p}\manifold{#1}} 
 
\newcommand{\boundary}{\partial} 
\newcommand{\ederiv}{\mathrm{d}} 
\newcommand{\dederiv}{\delta} 
\newcommand{\coderiv}{\mathrm{d}^{*}} 
\newcommand{\spacemap}{\rightarrow} 
 
\newcommand{\pderivative}[2]{\frac{\partial #1}{\partial #2}} 

\newcommand{\inner}[2]{\left( #1, #2\right)} 
\newcommand{\innerspace}[3]{\left(#1,#2\right)_{#3}} 
\newcommand{\duality}[2]{\langle #1, #2\rangle} 

\renewcommand{\eqref}[1]{(\ref{#1})} 
\newcommand{\figref}[1]{Figure~\ref{#1}} 
\newcommand{\secref}[1]{Section~\ref{#1}} 
\newcommand{\defref}[1]{Definition~\ref{#1}} 
\newcommand{\propref}[1]{Proposition~\ref{#1}} 
\newcommand{\theoremref}[1]{Theorem~\ref{#1}} 
\newcommand{\lemmaref}[1]{Lemma~\ref{#1}} 
\newcommand{\exampleref}[1]{Example~\ref{#1}} 
\newcommand{\remarkref}[1]{Remark~\ref{#1}} 

\newtheorem{theorem}{Theorem}
\theoremstyle{plain}

\newtheorem{corollary}{Corollary}

\newtheorem{definition}{Definition}
\newtheorem{example}{Example}

\newtheorem{lemma}{Lemma}

\newtheorem{proposition}{Proposition}
\newtheorem{remark}{Remark}

\numberwithin{equation}{section}

\newcommand{\define}{\mathrel{\mathop:}=}

\newcommand{\tr}{\mathrm{tr}\:}

\newcommand{\ve}{\varepsilon}

\begin{document}
\title[Mimetic framework]{Mimetic framework on curvilinear quadrilaterals of arbitrary order }
\author{Jasper Kreeft}
\address[]
	{Delft University of Technology, Faculty of Aerospace Engineering, \newline%
	\indent  Kluyverweg 2, 2629 HT Delft, The Netherlands.}%
\email[]{J.J.kreeft@TUDelft.nl, A.Palha@TUDelft.nl, M.I.Gerritsma@TUDelft.nl}%
\author{Artur Palha}
\author{Marc Gerritsma}
\thanks{Jasper Kreeft is funded by STW Grant 10113.}
\thanks{Artur Palha is funded by FCT Grant SFRH/ BD/36093 / 2007.}
\thanks{This paper is in final form and no version of it will be submitted for
publication elsewhere.}
\date{\textcolor{red}{\today}}

\begin{abstract}
In this paper higher order mimetic discretizations are introduced which are firmly rooted in the geometry in which the variables are defined. The paper shows how basic constructs in differential geometry have a discrete counterpart in algebraic topology. Generic maps which switch between the continuous differential forms and discrete cochains will be discussed and finally a realization of these ideas in terms of mimetic spectral elements is presented, based on projections for which operations at the finite dimensional level commute with operations at the continuous level.

The two types of orientation (inner- and outer-orientation) will be introduced at the continuous level, the discrete level and the preservation of orientation will be demonstrated for the new mimetic operators. The one-to-one correspondence between the continuous formulation and the discrete algebraic topological setting, provides a characterization of the oriented discrete boundary of the domain. The Hodge decomposition at the continuous, discrete and finite dimensional level will be presented. It appears to be a main ingredient of the structure in this framework.
\end{abstract}

\maketitle
\tableofcontents

\section{Introduction}
\subsection{Motivation}
The starting point of physics is the science of measurable, quantifiable objects and the relation among these objects. Any measurable quantity is associated with a spatial and a temporal geometric object. For instance, the measurement of velocity in air flow can be performed by particle image velocimetry (PIV), where tracer particles are released in the flow, two pictures are taken at consecutive time instants and the average velocity is calculated from the distance a particle has traveled divided by the time interval between the two snapshots. Velocity, measured in this way, is therefore associated with a curve in space (the trajectory of the particle between the two time instants) and a time interval (the time interval between the two snapshots). This association of velocity with a curve and a time interval does not depend on the particular way in which the measurement of velocity is performed. Note that
\[ \int_{t^1}^{t^2} \mathbf{v}\,dt = \int_{t^1}^{t^2} \frac{d\mathbf{r}}{dt}\,dt = \mathbf{r}(t^2) - \mathbf{r}(t^1) \;,\]
is an {\em exact} relation between the velocity in the time interval $[t^1,t^2]$ and the position of the tracer particles at $t^1$ and $t^2$, irrespective of the particular path traced out by the particle. So, if we could measure these two positions with infinite accuracy, we would have the time integral of the velocity correct. The approximation enters into the measurement by {\em assuming} that the particle moves at a constant velocity in a straight line from $\mathbf{r}(t^1)$ to $\mathbf{r}(t^2)$, which allows us to equate the velocity in the interval $[t^1,t^2]$ to
\[ v(s) \approx \frac{1}{t^2-t^1} \int_{t^1}^{t^2} \mathbf{v}\,dt  = \frac{\mathbf{r}(t^2) - \mathbf{r}(t^1)}{t^2-t^1} \;,\;\;\; s\in [t^1,t^2]\;.\]
We can decompose this velocity measurement in two consecutive steps: 1.) A reduction step, where we sample the position of the tracer particles at discrete time instants 2.) A reconstruction map, where we {\em assume} a certain behaviour of the particles between the two time instants.

These two steps, reduction and reconstruction, which implicitly play an important role in any measurement, will also turn out to be the two key ingredients in setting up mimetic discretizations.

Strictly speaking, the velocity measured in the PIV experiment is only the time-averaged velocity, but by assuming that the trajectory of the tracer particle is sufficiently smooth as a function of time, we can reduce the time interval such that we can quite accurately determine `the velocity' at a given position and at a given time instant. An alternative approach would be to sample at a larger number of time instants, say $t^1,\dots,t^n$ and reconstruct the trajectory based on the measured positions. We will call the first approach (reducing the time interval) {\em $h$-refinement} whereas the second approach (reconstruction of the trajectory using more time instants) will be called {\em $p$-refinement}. Similar ideas are used in discretization where refinement of the mesh is denoted by $h$-refinement while a reconstruction based on more samples is referred to as $p$-refinement.

Ultimately, in many physical theories, one takes the limit for all lengths and time intervals to zero, which enables physicists and engineers to talk about the velocity {\em in a point} at a {\em certain time instant}, $\mathbf{v}(t,x,y,z)$. Any connection with a distance and a time interval is lost after this limiting process. Another well-known example is mass contained in a volume, ${\mathcal V}$. The average density is the mass divided by the volume. By taking the limit for ${\mathcal V} \rightarrow 0$ we obtain the density in a point, $\rho(t,x,y,z)$. Again, the connection with the volume is lost after taking this limit.

Bear in mind that this limiting process is purely mathematical. If we consider the PIV experiment again to measure the local velocity, we always need a finite time interval in order to evaluate the average velocity. If we would reduce the time interval to zero, no velocity measurement could be made. So despite the fact that we can accurately determine the velocity in a flow at a certain location and at a certain time, this measured velocity will always be connected to a time interval and the displacement along a curve.

This association of physical variables with spatial and temporal geometric elements can be done for all physical variables. This is, however, beyond the scope of this paper, but the interested reader is referred to the work of \cite{tonti1975formal,bossavit:japanese_02,mattiussi2000finite} and especially the forthcoming book by Tonti, \cite{bookTonti}.

Once we acknowledge that there is such an association between physical variables and geometric objects, we need to take orientation of the geometric object into account. A curve with endpoints $A$ and $B$ possesses two orientations. Either the curve from $A \rightarrow B$ is taken as the positively oriented curve, or the curve $B\rightarrow A$ is taken to be positively oriented. The same holds for the orientation of surfaces (oriented clockwise or counter-clockwise) and three dimensional volumes (right-hand rule or left-hand rule). Note that the notion of orientation does not prompt itself when we only consider physical variables defined at time instants and in points, although it is useful to consider the orientation of points as well.

Physical variables are associated to geometric objects and geometric objects have an orientation, however, the physical quantity is {\em independent of the orientation of the associated geometric object}. If we choose the direction of time to be positive when pointing in the past, the integral value of the velocity changes sign
\[ \int_{t^2}^{t^1} \mathbf{v}\,dt = - \int_{t^1}^{t^2} \mathbf{v}\,dt \;.\]
So integral values (the ones that are measurable) are intimately connected to the orientation of geometry, while average values -- and in the limit densities -- are insensitive to the orientation of space and time, because
\[ \frac{1}{t^1-t^2} \int_{t^2}^{t^1} \mathbf{v}\,dt = \frac{1}{t^2-t^1} \int_{t^1}^{t^2} \mathbf{v}\,dt \;.\]
There also exist physical variables which do change when the orientation is reversed.
Therefore, we need to distinguish between two types of orientation: {\em inner-} and {\em outer} orientation. With {\em inner} orientation, we mean the orientation {\em in} the geometric object such as for instance the electric current {\em in} a wire or the rotation {\em in} a plane, whereas {\em outer} refers to the orientation {\em outside} the geometric object such as the Biot-Savart law {\em around} the wire and or the flux {\em through} the plane. 

It turns out that the association with oriented geometric objects is a vital ingredient in the description of physics and when we perform the limiting process to define all physical quantities in points and at time instants without reference to the associated geometric objects much of this rich structure of the physical model will be lost.

In this paper, therefore, we want to set up a framework in which we {\em mimic} the association of physical variables with oriented geometric objects for computational analysis. The aim is to develop families of numerical discretizations which work on general quadrilateral grids of arbitrary order. By `families' we mean that we formulate the basic requirements a numerical discretization needs to possess in order to be compatible with its associated geometry. This leads, among others, to exact discrete representations for the gradient operator, the curl operator and the divergence operator. Also, the explicit  distinction between inner- and outer-orientation with their associated cell complexes, will anatomy of the boundary of the domain. This, in turn, will clarify the issue of where and how to prescribe boundary values. 

The mimetic structure that will be introduced in this paper ensures that the reduced physical model behaves in the same way as the full infinite dimensional system. Let $A$ and $B$ be two physical quantities and $T$ a continuous operator which maps $A$ onto $B$, $T(A)=B$, then the following diagram commutes
\[
\begin{CD}
A @>T>> B\\
@V\pi VV @V\pi VV\\
A_h @>T>> B_h
\end{CD}
\]
Here is $\pi$ a suitably constructed projection operator which maps continuous variables in finite dimensional representations. So $\pi \circ T = T \circ \pi$, i.e. we can perform the operation $T$ at the continuous level and then discretize or first discretize and then apply the operator $T$. In that sense, operations at the discrete level truly mimic the behaviour of the operators at the continuous level.

These properties have quite some consequences for practical applications, but we have chosen not to present an extensive gallery of applications. Only a few simple examples will be given which serve to illustrate some of the claims in this paper. Applications of ideas presented in this paper can be found in \cite{bouman::icosahom2009,gerritsma:lssc2009,kreeft::stokes,jasper::eccomas2010,palha::icosahom2009,palha:lssc2009}.

\subsection{Prior and related work}
Over the years numerical analysts have developed numerical schemes which preserve some of the structure of the differential models they aim to approximate, so in that respect the whole mimetic idea is not new. A recent development is that the proper language in which to encode these structures/symmetries is the language of differential geometry. Another novel aspect of mimetic discretizations is the identification of the metric-free part of differential models, which can be conveniently described in terms of algebraic topology which employs the strong analogy between differential geometry and algebraic topology.

The relation between differential geometry and algebraic topology in physical theories was first established by Tonti, \cite{tonti1975formal}. Around the same time Dodziuk, \cite{Dodziuk76}, set up a finite difference framework for harmonic functions based on Hodge theory. Both Tonti and Dodziuk introduce differential forms and cochain spaces as the building blocks for their theory. The relation between differential forms and cochains is established by the Whitney map ($k$-cochains $\rightarrow$ $k$-forms) and the De Rham map ($k$-forms $\rightarrow$ $k$-cochains). The interpolation of cochains to differential forms on a triangular grid was already established by Whitney, \cite{Whitney57}. These interpolatory forms are now known as the {\em Whitney forms}.

Hyman and Scovel, \cite {HymanScovel90}, set up the discrete framework in terms of cochains, which are the natural building blocks of finite volume methods. Later Bochev and Hyman, \cite{bochev2006principles} extended this work and derived discrete operators such as the discrete wedge product, the discrete codifferential, the discrete inner product, etc. These operators are all cochain operators. 

In a finite difference/volume context Robidoux, Hyman, Steinberg and Shashkov, \cite{HymanShashkovSteinberg97,HymanShashkovSteinberg2002,HYmanSteinberg2004,RobidouxAdjointGradients1996,RobidouxThesis,bookShashkov,Steinberg1996,SteibergZingano2009} used symmetry considerations to discretize diffusion problems on rough grids and with non-smooth non-isotropic diffusion coefficients. In a recent paper by Robidoux and Steinberg \cite{RobidouxSteinberg2011} a discrete vector calculus in a finite difference setting is presented. It satisfies the discrete differential operators grad, curl and div exactly and the numerical approximations are all contained in the constitutive relations, which are already polluted by modeling and experimental error. This paper also contains an extensive list of references to mimetic methods.
For mimetic finite differences, see also Brezzi et al., \cite{BrezziBuffaLipnikov2009,brezzi2010}.

The application of mimetic ideas to unstructured staggered grids has been extensively studied by Perot, \cite{Perot2000,ZhangSchmidtPerot2002,perot2006mimetic,PerotSubramanian2007a,PerotSubramanian2007}. Especially the recent paper, \cite{perot43discrete}, lucidly describes the rationale of preserving symmetries in numerical algorithms.

Mattiussi, \cite{mattiussi1997analysis,mattiussi2000finite,Mattiussi02} puts the geometric ideas proposed by Tonti in an finite volume, finite difference and finite element context. The idea of switching between cochains and differential forms is also prominent in the work of Hiptmair, for instance \cite{hiptmair2001discrete}. This work also displays the close connection between finite volume methods and finite element methods.

Mimetic methods show a clear connection between the variables (differential forms) and the geometry in which these variables are defined. The most `geometric approach' is described in the work by Desbrun et al., \cite{desbrun2005discrete,ElcottTongetal2007,MullenCraneetal2009,PavlovMullenetal2010} and the thesis by Hirani, \cite{Hirani_phd_2003}. 

The `Japanese papers' by Bossavit, \cite{bossavit:japanese_01,bossavit:japanese_02}, serve as an excellent introduction and motivation for the use of differential forms in the description of physics and the use in numerical modeling. The field of application is electromagnetism, but these papers are sufficiently general to extend all concepts to other fields of expertise. 

In a series of papers by Arnold, Falk and Winther, \cite{arnold:Quads,arnold2006finite,arnold2010finite}, a finite element exterior calculus framework is developed. Just like in this paper, Arnold, Falk and Winther consider methods of arbitrary order. Higher order methods are also described by Rapetti, \cite{Rapetti2007,Rapetti2009} and Hiptmair, \cite{hiptmair2001}. Possible extensions to spectral methods were described by Robidoux, \cite{robidoux-polynomial} and applications of mimetic spectral methods can be found in \cite{bouman::icosahom2009,gerritsma:lssc2009,kreeft::stokes,jasper::eccomas2010,palha::icosahom2009,palha:lssc2009}.

Isogeometric reconstruction was used by Buffa et al., \cite{BuffaDeFalcoSangalli2011}, Evans, \cite{thesis_Evans} and Hiemstra, \cite{Hiemstra::grad_curl_div,Hiemstra::Harmonic}.

Although the cited literature is far from complete, the above references serve as excellent introduction into the field of mimetic discretization techniques.

\subsection{Scope and outline of this paper}

In the introduction and work cited above it has been revealed that geometry plays an important role in mimetic methods. In computational engineering one usually works with fields and densities, i.e. the variables obtained after the limiting process. The main reason fields have emerged as the preferred way of encoding physics is because physical laws can then be stated in terms of differential equations.

An alternative description is in terms of integral equations. The appeal of an integral approach lies in the fact that the physical laws can be stated without any limiting process involved, rendering them closer to the physical measurement process and more suited for a discrete treatment. Integration can be interpreted as duality pairing between geometry and variables connected to this geometry (differential forms).

An important differential operator in differential geometry  is the exterior derivative. The exterior derivative can be defined in terms of geometric concepts, i.c. the boundary operator, through the generalized Stokes Theorem. When written in terms of conventional vector calculus, the exterior derivative is either the gradient, the curl or the divergence, depending on the context. The introduction of the exterior derivative allows one to uniquely decompose the space of differential forms into a direct sum of sub-spaces. This {\em Hodge decomposition} generalizes the classical Helmholtz decomposition for non-contractible domains. So, if we want to incorporate geometry into our physical description, differential geometry is a concise and potent way to do so.
Therefore, in Section~\ref{sec:DifferentialGeometry} a brief introduction into differential geometry will be presented. Although this material can be found in any book on differential geometry, \cite{burke1985applied,flanders::diff_forms,frankel}, we include this section to introduce our notation and in the remainder of the paper we want to highlight which properties from differential geometry are retained at the discrete level.

Despite the fact that physics requires metric concepts like length, angles and area, many structures in physics are completely independent of the metric. These non-metric concepts are called {\em topological}. 
In Section~\ref{sec:AlgebraicTopology} elements from algebraic topology will be discussed which is required for the development of the mimetic spectral element framework. It will be shown that the structure of algebraic topology resembles the structure of differential geometry and therefore algebraic topology could serve as the discrete setting for our numerical framework.

In Section~\ref{mimeticoperators} the connection between algebraic topology and differential geometry will be established. Based on the existence of a suitable reduction map, $\reduction$, which maps $k$-forms onto $k$-cochains and a reconstruction map, $\reconstruction$, which converts $k$-cochains to $k$-forms, a general mimetic framework will be set up using the projection operator $\projection = \reconstruction \circ \reduction$, which maps the space of differential forms, $\Lambda^k$, to a finite dimensional space of differential forms, $\Lambda_h^k$. This section resembles the paper by Bochev and Hyman, \cite{bochev2006principles}, but the main difference is that in Section~\ref{mimeticoperators} the finite dimensional discrete space consists of differential forms while Bochev and Hyman take the cochains as their discrete variables. 

In Section~\ref{sec:MSEM} the actual polynomial reduction and reconstruction maps are presented which satisfy the requirements described in Section~\ref{mimeticoperators}. Their composition forms a bounded linear projection as proven in this section. This section is accompanied with many examples of the actions of the various operators.

In Section~\ref{sec:Conclusions} we will review the tools developed in this paper and look back to the introduction and see how this approach may enable us to faithfully simulate problems in physical sciences. Furthermore, potential future directions will be identified.

Although this outline suggests a collection of seemingly unrelated scientific fields, several ideas/concepts permeate throughout the paper. Ultimately, all concepts contribute to mimetic, numerical concepts:
\begin{enumerate}
\item The exterior derivative $\longrightarrow$ coboundary operator $\longrightarrow$ discrete gradient, curl and divergence;
\item The Hodge decomposition $\longrightarrow$ cohomology group $\longrightarrow$ discrete Helmholtz decomposition;
\item The wedge product $\longrightarrow$ tensor products $\longrightarrow$ basis functions on quadrilateral elements;
\item The behaviour under mappings: the pullback operator $\longrightarrow$ the cochain map $\longrightarrow$ mimetic discretization on highly deformed meshes;
\item Inner- and outer orientation $\longrightarrow$ the double De Rham complex $\longrightarrow$ boundary of the domain $\longrightarrow$ the trace operator/boundary values $\longrightarrow$ spectral method on a staggered grid;
\item The generalized Stokes Theorem $\longrightarrow$ discrete generalized Stokes $\longrightarrow$ exact conservation and existence of scalar and vector potentials.
\end{enumerate}

\section{Differential Geometric Concepts}\label{sec:DifferentialGeometry}
Differential geometry, along with algebraic topology presented in the next section, constitutes the basis of the numerical framework presented in this paper. In contrast to conventional vector calculus, which is a topic well known to all the readers differential geometry is not a familiar topic to most of the readers. Since differential geometry is essential, we include a short introduction in order to make this work as much as possible self-contained and to be able to draw analogies between differential geometry, algebraic topology and the mimetic scheme we will present. Only those concepts from differential geometry which will play a role in the remainder of this article are introduced.
We start by introducing the concept of manifold, which is the playground in which everything is defined, the geometry, and the concept of orientation is presented. Next, differential forms, their definition, the operators defined on them and their transformation under mappings are introduced subsequently. Differential form spaces will be introduced, including the concept of Hodge decomposition.
The matter presented here constitutes the mathematical tools with which the physical quantities and the physical laws will be represented in the continuous world, and what will be mimicked in the numerical framework presented.
For a more in depth treatment of differential geometry in physics, we refer to \cite{burke1985applied,flanders::diff_forms,frankel,hou::differential_geometry_physicists,isham,morita2001geometry,schutz1980geometrical,spivak1998calculus}.
	
	\subsection{Manifolds}\label{sec:manifolds}
		The concepts which will be introduced all exist associated to sets endowed with enough structure so that one can ``do calculus'' and which are denoted by manifolds. In $\mathbb{R}^{3}$ these are commonly referred to as points, lines, surfaces and volumes. Generalizing to any dimension a manifold can be defined in the following way.
	
		\begin{figure}[htp]
				\centering
				\includegraphics[width=0.4\textwidth]{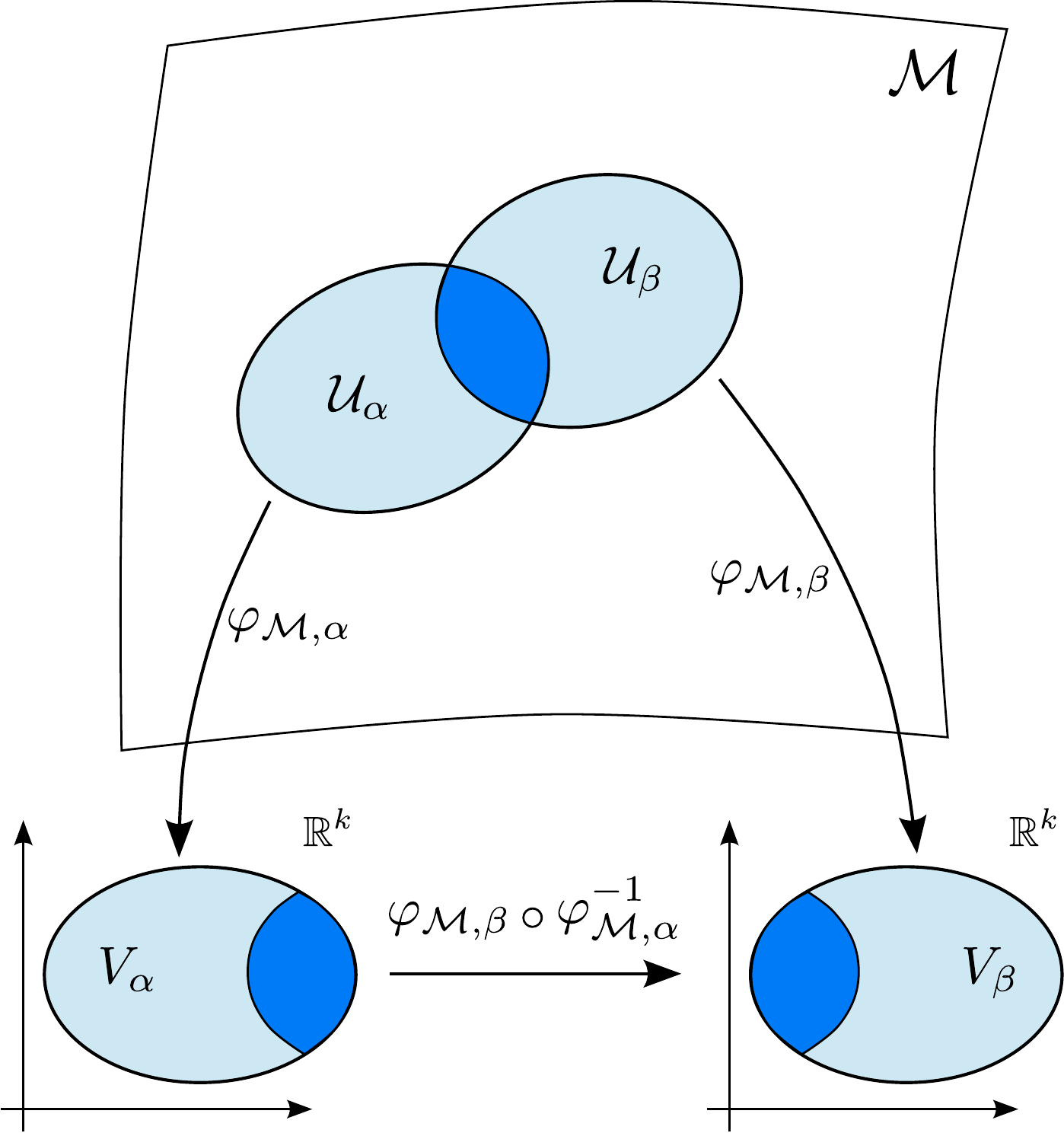}
					\caption{Coordinate charts on a manifold.}
					\label{fig::diffGeometry_mapping}
		\end{figure}	
		
\begin{definition}[\textbf{Manifold}]\cite{Olver}\label{def:manifold}
A {\em $k$-dimensional manifold} is a set $\manifold{M}$, together with a countable collection of subset $\manifold{U}_\alpha \subset \manifold{M}$, called {\em coordinate charts}, and one-to-one functions $\varphi_{\manifold{M},\alpha}\,:\,\manifold{U}_\alpha \rightarrow V_\alpha$ onto connected open subsets $V_\alpha$ of $\mathbb{R}^k$, called {\em local coordinate maps}, as in \figref{fig::diffGeometry_mapping}, which satisfy the following properties:
\begin{itemize}
\item[(1)] The coordinate charts cover $\manifold{M}$:
\[ \bigcup_\alpha \manifold{U}_\alpha = \manifold{M} \;.\]
\item[(2)] On the overlap of any pair of coordinate charts $\manifold{U}_\alpha \cap \manifold{U}_\beta$, the composite map
\[ \varphi_{\manifold{M},\beta} \circ \varphi_{\manifold{M},\alpha}^{-1}\,:\, \varphi_{\manifold{M},\alpha} ( \manifold{U}_\alpha \cap \manifold{U}_\beta ) \rightarrow \varphi_{\manifold{M},\beta} ( \manifold{U}_\alpha \cap \manifold{U}_\beta ) \;,\]
is a smooth (infinitely differentiable) function.
\item[(3)] If $x \in \manifold{U}_\alpha$ and $y \in \manifold{U}_\beta$ are distinct points in $\manifold{M}$, then there exist open subsets $W_\alpha$ of $\varphi_{\manifold{M},\alpha}(x)$ in $V_\alpha$ and $W_\beta$ of $\varphi_{\manifold{M},\beta}(x)$ in $V_\beta$ such that
\[ \varphi_{\manifold{M},\alpha}^{-1}(W_\alpha) \cap \varphi_{\manifold{M},\beta}^{-1}(W_\beta) = \emptyset \;.\]
\end{itemize}
\end{definition}
Since the image of each point $p\in (\manifold{U}_\alpha\cap\manifold{M})$ by $\varphi_{\manifold{M},\alpha}$ is a point in $\mathbb{R}^{k}$, it can be written as a $k$-tuple of real numbers: $\varphi_{\mathcal{M},\alpha}(p)=(x^{1}(p),\hdots,x^{n}(p))$. This $k$-tuple is called the local coordinates of $p$ and $\manifold{U}_\alpha\cap\manifold{M}$ the coordinate neighborhood. The pair $(\manifold{U}_\alpha,\varphi_{\manifold{M},\alpha})$ is called a {\em local chart} or {\em local coordinate system}. An {\em atlas} on a manifold $\manifold{M}$ is a collection $\mathcal{A}=\{(\manifold{U}_\alpha,\varphi_{\manifold{M},\alpha})\}$ of charts of $\manifold{M}$, such that $\bigcup_{\alpha}\manifold{U}_\alpha=\manifold{M}$; the collection of open sets $\{\manifold{U}_\alpha\}$ constitutes an open covering of the manifold $\manifold{M}$.

\begin{definition}[\textbf{Maps between manifolds}]\label{def:map_between_manifolds}
Let $\manifold{M}$ be a $k$-dimensional smooth manifold and $\manifold{N}$ an $l$-dimensional smooth manifolds. The map $\Phi:\manifold{M} \rightarrow \manifold{N}$ maps between manifolds, if for every coordinate chart $\varphi_{\manifold{M},\alpha}:\manifold{U}_\alpha \rightarrow V_\alpha \subset \mathbb{R}^k$ on $\manifold{M}$ and every chart $\varphi_{\manifold{N},\beta}: \manifold{U}_\beta \rightarrow V_\beta \subset \mathbb{R}^l$ on $\manifold{N}$, the composite map
\[ \varphi_{\manifold{N},\beta} \circ \Phi \circ \varphi_{\manifold{M},\alpha}^{-1}\,:\,\mathbb{R}^k \rightarrow \mathbb{R}^l \;,\]
is a smooth map wherever it is defined. See \figref{fig::Change_of_coordinates} for a pictorial representation of the mapping between manifolds.
\end{definition}
\begin{figure}[htp]
\centering
\includegraphics[width=0.6\textwidth]{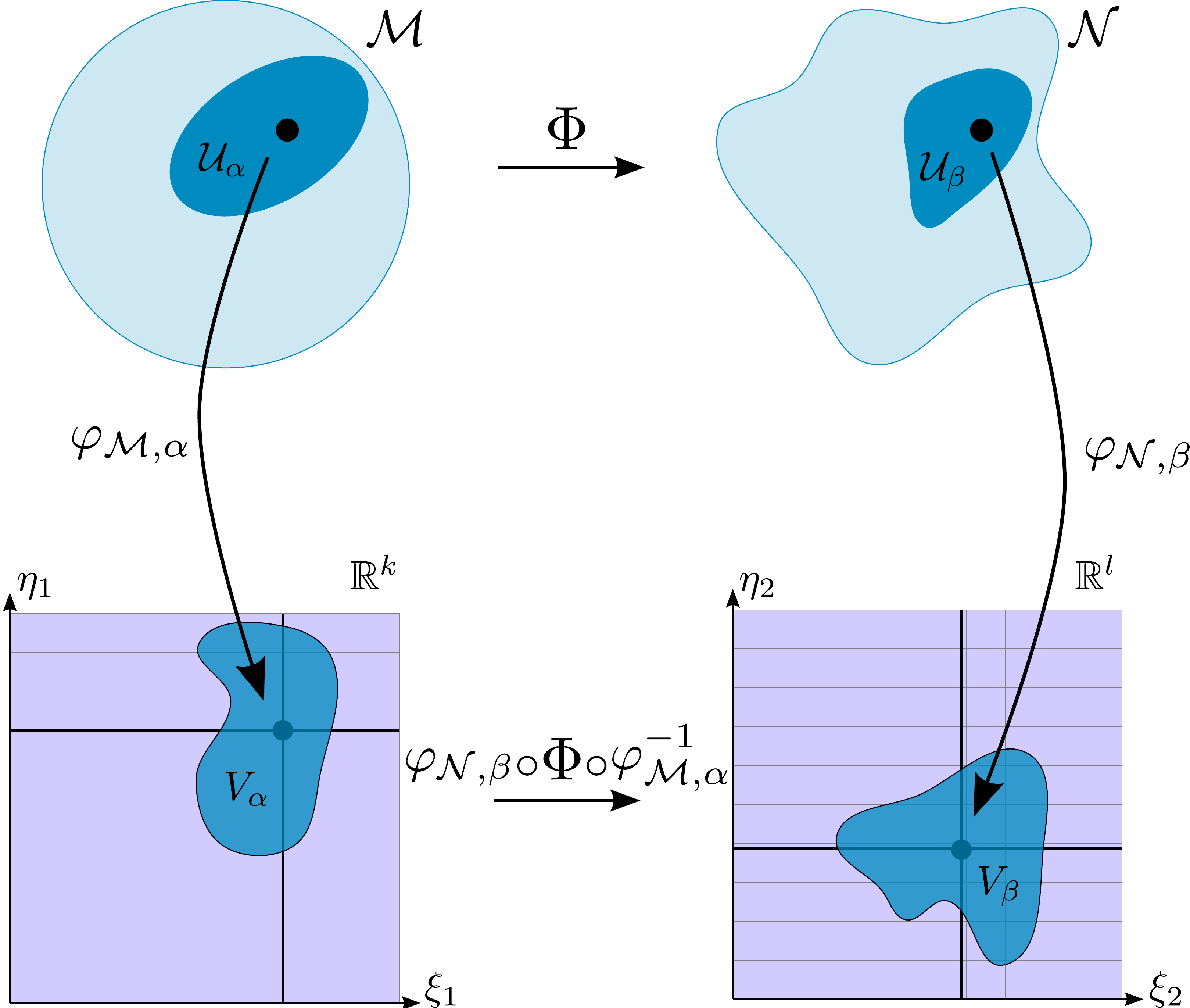}
\caption{Mapping between two manifolds,  $\manifold{M}$ and $\manifold{N}$.}
\label{fig::Change_of_coordinates}
\end{figure}

If the map $\varphi_{\manifold{N},\beta} \circ \Phi \circ \varphi_{\manifold{M},\alpha}^{-1}$ is of maximal rank at $p \in \manifold{M}$, then there are local coordinates $x=(x^1,\ldots,x^k)$ near $p$ and $y=(y^1,\ldots,y^l)$ near $q=\Phi(p)\in \manifold{N}$ such that these coordinates have the simple form
\[ y = (x^1,\ldots,x^k,0,\ldots,0)\;,\quad\quad \mbox{if } l>k \;,\]
or
\[ y = (x^1,\ldots,x^l)\;,\quad\quad \mbox{if } l\leq k \;,\]

\begin{definition}[\textbf{Submanifold}]\label{def:submanifold}
Let $\manifold{M}$ be a smooth manifold. A {\em submanifold} of $\manifold{M}$ is a subset $\manifold{S} \subset \manifold{M}$, together with a smooth one-to-one mapping $\Phi:\tilde{\manifold{S}} \rightarrow \manifold{S} \subset \manifold{M}$ satisfying the maximal rank condition everywhere, where the {\em parameter space} $\tilde{\manifold{S}}$ is some other manifold and $\manifold{S}=\Phi(\tilde{\manifold{S}})$ is the image of $\Phi$. In particular, the dimension of $\manifold{S}$ is the same as that of $\tilde{\manifold{S}}$, and does not exceed the dimension of $\manifold{M}$.
\end{definition}

An important concept is the boundary of a manifold. This concept plays an essential role in the generalized Stokes theorem, \theoremref{theor::difGeom_stokes_theorem}, to be introduced later on.

	\begin{definition}[\textbf{Complement, interior point, exterior point, boundary point, open set and closed set}]\label{def:interior,exterior,boundary}
		\cite{isham}  Given a subset $\manifold{S}$ of a manifold $\manifold{M}$, the complement of $\manifold{S}$ in $\manifold{M}$ is the set of points $\manifold{S}^{c}:=\{p\in \manifold{M}\,|\, p \notin \manifold{S}\}$. Let the ball $B_{\epsilon}(p) = \left \{ x \in \manifold{M} \,|\, d(x,p) < \epsilon \right \}$, then
		\begin{enumerate}
			\item A point $p$ is an interior point of $\manifold{S}$ if there exists $\epsilon > 0$ such that the neighborhood $B_{\epsilon}(p)$ around $p$ has the property that $B_{\epsilon}(p)\subset \manifold{S}$. One writes $p\in\mathrm{int}(\manifold{S})$.
			\item A point $p$ is an exterior point of $\manifold{S}$ if there exists $\epsilon > 0$ such that the neighborhood $B_{\epsilon}(p)$ around $p$ has the property that $B_{\epsilon}(p)\cap \manifold{S}=\emptyset$. One writes $p\in\mathrm{ext}(\manifold{S})$.
			\item A point $p$ is a boundary point of $\manifold{S}$ if every neighborhood $B_{\epsilon}(p)$ around $p$ with $\epsilon > 0$ has the property that $B_{\epsilon}(p)\cap \manifold{S}\neq 0$ and $B_{\epsilon}(p)\cap \manifold{S}^{c}\neq 0$. One write $p \in \partial \manifold{S}$.
		\end{enumerate}
		
		A set $\manifold{S}$ is open if, and only if, $\manifold{S}=\mathrm{int}(\manifold{S})$. A set $\manifold{S}$ is closed if, and only if, its complement $\manifold{S}^{c}$ is open.
	\end{definition}
	
	In order to introduce the concept of boundary one needs to introduce the closed upper half-space of dimension $n$:
	\[
		\mathbb{H}^{n} = \{ (x^{1}, \ldots, x^{n}) \in \mathbb{R}^{n} \,|\, x^{n} \geq 0 \},
	\]
	with the subspace topology of $\mathbb{R}^{n}$. From Definition~\ref{def:interior,exterior,boundary} it follows that points with $x^{n} > 0$ are the interior points of $\mathbb{H}^{n}$, the points with $x^{n}<0$ are the exterior points of $\mathbb{H}^{n}$ and the points with $x^{n}=0$ are the boundary points of $\mathbb{H}^{n}$. 
	
	\begin{proposition}\label{prop::diffGeom_interior_to_interior}
		\cite{tu2010introduction} Let $P$ and $Q$ be subsets of $\mathbb{H}^{n}$ and $\Phi: P\rightarrow Q$ a diffeomorphism. Then $\Phi$ maps interior points to interior points and boundary points to boundary points.
	\end{proposition}
%
	 
	 One can then define an $n$-manifold with boundary:
	 
	 \begin{definition}[\textbf{$n$-Manifold with boundary, interior point and boundary point of an $n$-manifold with boundary}]
 	\cite{tu2010introduction} An $n$-manifold with boundary, $\manifold{M}$, is a topological space which is locally $\mathbb{H}^{n}$. A point $p$ of $\manifold{M}$ is an interior point if there is a chart $\varphi_{\manifold{M},\alpha}$ in which $\varphi_{\manifold{M},\alpha}(p)$ is an interior point of $\mathbb{H}^{n}$. In the same way, a point $p$ is a boundary point of $\manifold{M}$ if $\varphi_{\manifold{M},\alpha}(p)$ is a boundary point of $\mathbb{H}^{n}$. The set of boundary points of $\manifold{M}$ is denoted by $\partial\manifold{M}$.
	 \end{definition}
	 
	 \begin{definition}[\textbf{Boundary operator}]\label{def:boundary_manifold}
	 	Given an $n$-manifold with boundary, $\manifold{M}$, the boundary operator $\partial$ is a map $\partial: \manifold{M}\rightarrow\partial\manifold{M}$.
	 \end{definition}
	 
	 \begin{corollary}[\textbf{The boundary of a submanifold}]\label{cor:boundary_submanifold}
	 Since any submanifold $\manifold{S} \subset \manifold{M}$ is a manifold in its own right, the boundary of a submanifold is defined as in \defref{def:boundary_manifold}.
	 \end{corollary}
	 \begin{proposition}[\textbf{Boundary is mapped into a boundary and the boundary is independent of chart}]\label{prop:boundary_indep_chart}
	 	Given two $n$-dimensional manifolds with boundary, $\manifold{M}$ and $\manifold{N}$, and a mapping (diffeomorphism) between them, $\Phi:\manifold{M} \rightarrow \manifold{N}$, then the interior points and boundary points of $\manifold{M}$ are mapped onto interior points and boundary points of $\manifold{N}$, respectively. That is:
		\begin{equation}
			\partial\Phi(\manifold{M}) = \Phi(\partial\manifold{M}). \label{eq::difGeom_mapping_boundary}
		\end{equation}
		Moreover, interior and boundary points are independent of the choice of chart.
 	 \end{proposition}
	 \begin{proof}
	 For the first statement, let $\varphi_{\manifold{M},\alpha}:\manifold{U}_{\alpha}\subset\manifold{M}\rightarrow\mathbb{R}^k$ and $\varphi_{\manifold{N},\beta}:\manifold{U}_{\beta}\subset \manifold{N} \rightarrow \mathbb{R}^k$, then $\varphi_{\manifold{N},\beta} \circ \Phi \circ \varphi_{\manifold{M},\alpha}^{-1}:\mathbb{R}^k \rightarrow \mathbb{R}^k$ is a diffeomorphism and according to \propref{prop::diffGeom_interior_to_interior} this maps boundary points onto boundary points and interior points onto interior points.

	 As for the second statement, one takes $\manifold{M}=\manifold{N}$ and in this case the two charts will be $\varphi_{\manifold{M},\alpha}$ and ${\varphi}_{\manifold{M},\beta}$ and the mapping between the two charts (change of coordinates) will be given by ${\varphi}_{\manifold{M},\beta} \circ \Phi \circ \varphi_{\manifold{M},\alpha}^{-1} = {\varphi}_{\manifold{M},\beta} \circ \varphi_{\manifold{M},\alpha}^{-1}$, since $\Phi$ in this case is the identity map, and again ${\varphi}_{\manifold{M},\beta} \circ \varphi_{\manifold{M},\alpha}^{-1}:\mathbb{R}^k \rightarrow \mathbb{R}^k$ is a diffeomorphism and according to \propref{prop::diffGeom_interior_to_interior} this maps boundary points onto boundary points and interior points onto interior points.
	\end{proof}
	 
	In $n$-dimensional space it is possible to define $n+1$ sub-manifolds of dimension 0, 1, \ldots, $n$, respectively. For the case $n=3$ one can define, points, lines, surfaces and volumes. Moreover, it is also possible to orient these objects. By orientation one means the generalization of concepts such as left and right, front and back, clockwise or counterclockwise, outward and inward, etc. The different kinds of orientation will be presented for all the geometric objects that exist in 3-manifolds (points, lines, surfaces and volumes), as well as generalizations for geometric objects of arbitrary dimension. The distinction between inner orientation (solely related to the geometric object) and outer orientation (related to both the object and the embedding space) will be given.
For a more detailed discussion on orientation we recommend \cite{tonti1975formal,mattiussi2000finite,abraham_diff_geom, burke1985applied, schutz1980geometrical}.

	The concept of orientation on manifolds is a generalization of the one for vector spaces and hence, by extension, to $\mathbb{R}^{k}$. One starts with the notion of orientation in vector spaces and the charts $\varphi$ will induce an orientation on the manifold.
	
	In $\mathbb{R}^{1}$ an orientation is one of the two possible directions, see \figref{fig:diffGeom_orientation1DVectorSpace}. In $\mathbb{R}^{2}$ an orientation is one of the two possible rotations, clockwise or counterclockwise, see \figref{fig:diffGeom_orientation2DVectorSpace}. In $\mathbb{R}^{3}$ an orientation is one of the two possible screw senses, upward clockwise or upward counterclockwise, see \figref{fig:diffGeom_orientation3DVectorSpace}.

\begin{figure}[htp]
\leftskip-0.8cm
\begin{tabular}{ccc}
\subfigure[1D vector space]{\includegraphics[width=0.27\textwidth]{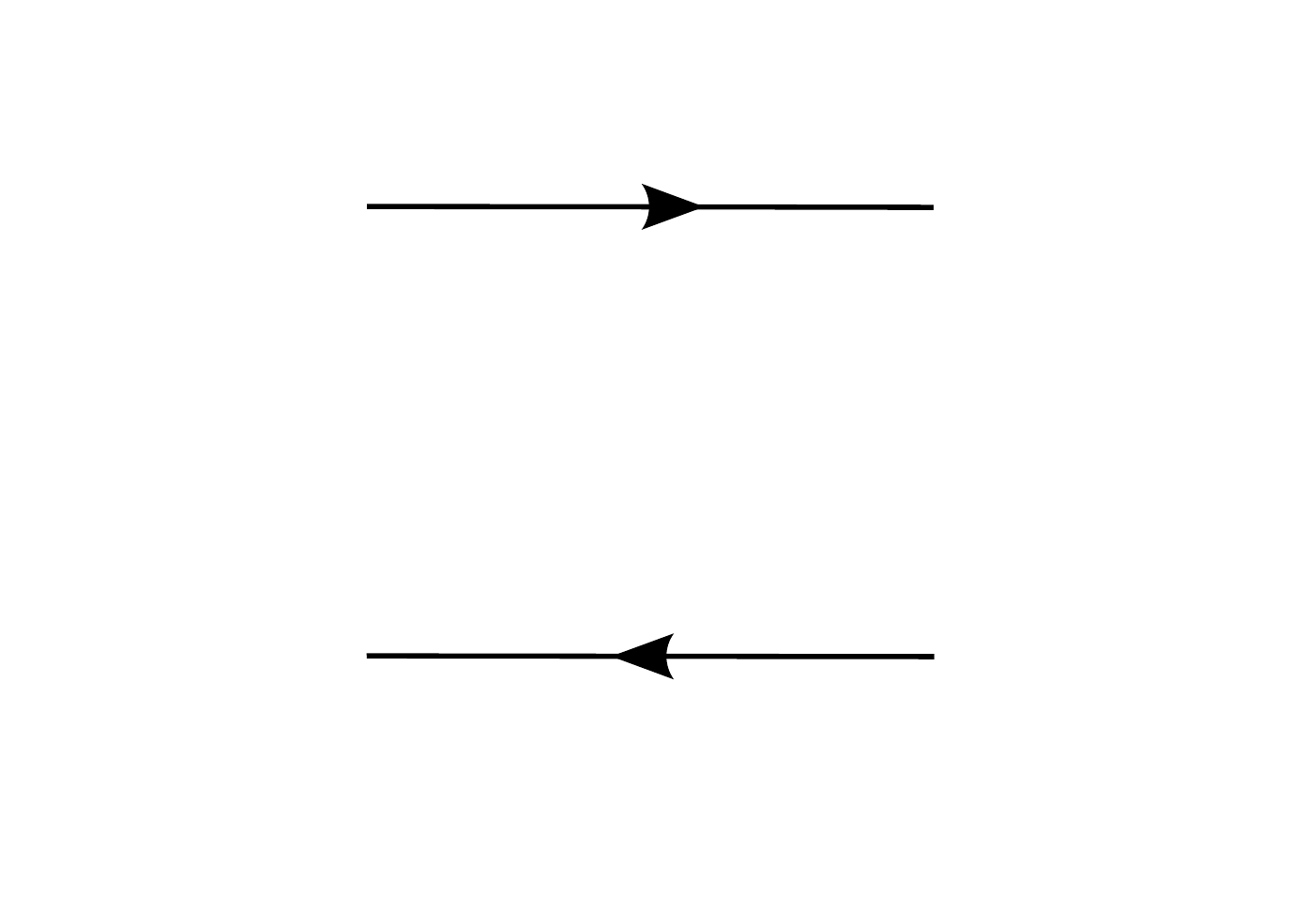}\label{fig:diffGeom_orientation1DVectorSpace}} & 
\subfigure[2D vector space]{\includegraphics[width=0.36\textwidth]{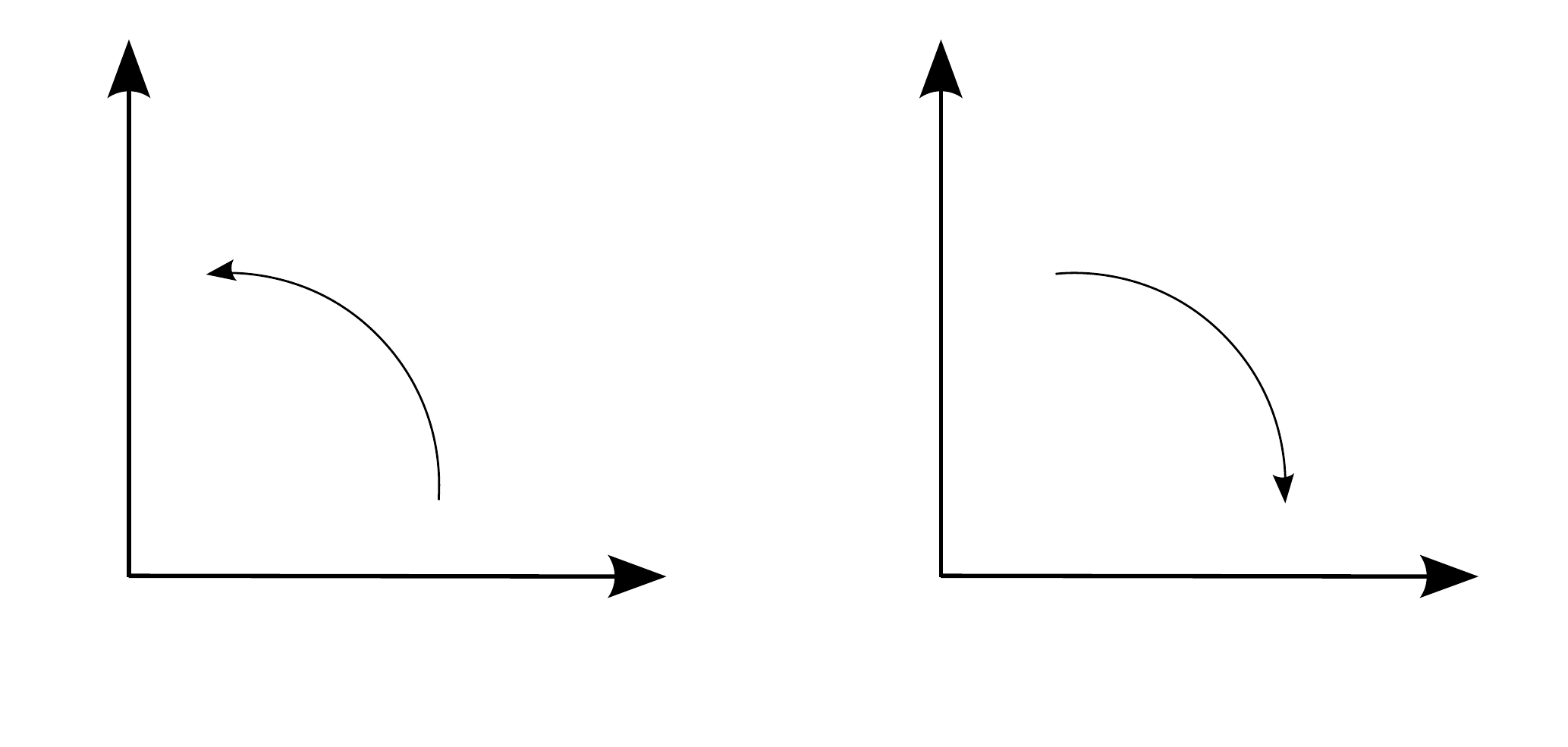}\label{fig:diffGeom_orientation2DVectorSpace}} & 
\subfigure[3D vector space]{\includegraphics[width=0.36\textwidth]{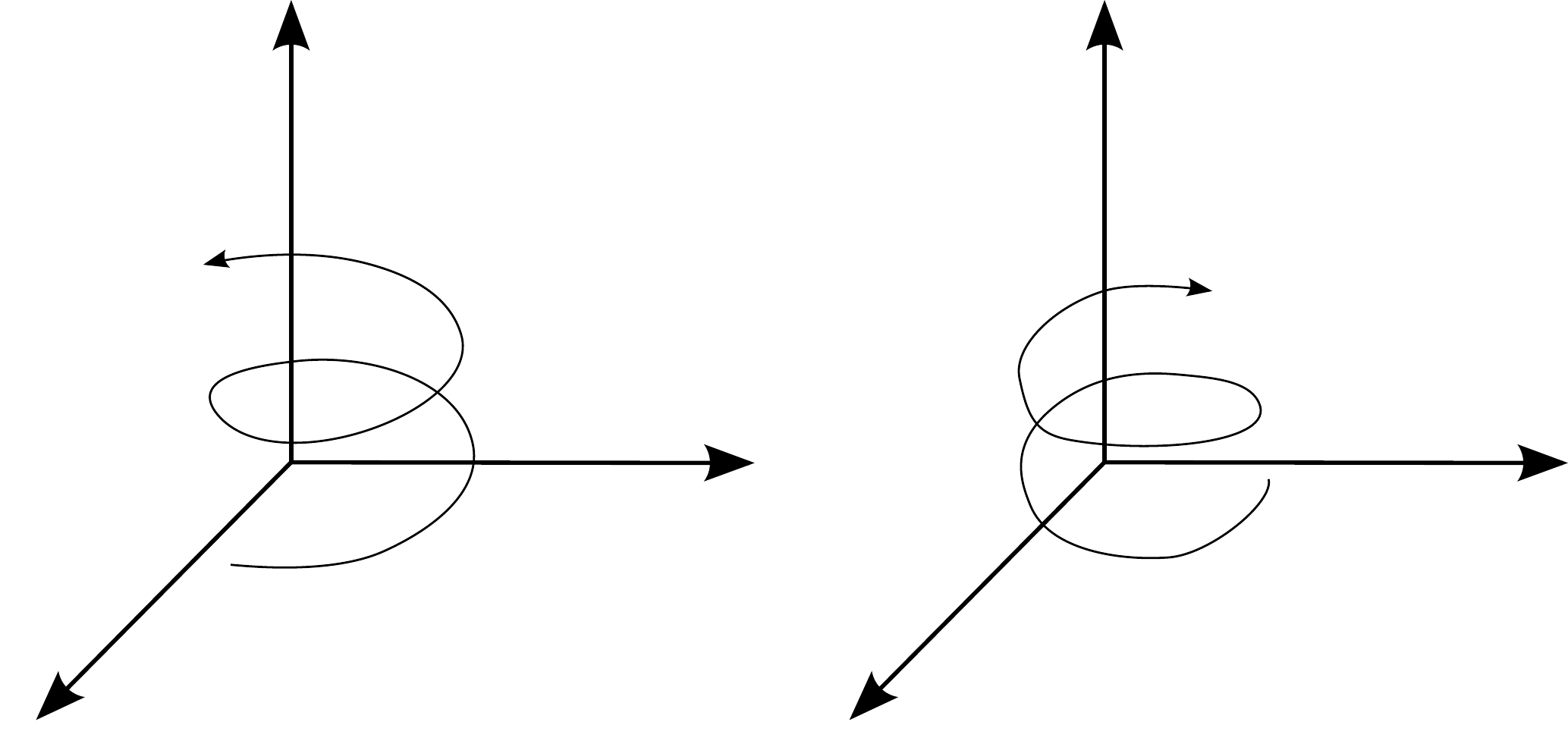}\label{fig:diffGeom_orientation3DVectorSpace}}\\
\end{tabular}
\caption{Possible orientations of vector spaces.}
\end{figure}

%

	The question is how to generalize this heuristic definition to higher dimensions. In a vector space of dimension $k$ one can transform one set of basis vectors, $\{\vec{v}_{1},\vec{v}_{2},\cdots,\vec{v}_{k}\}$, into another one, $\{\vec{u}_{1},\vec{u}_{2},\cdots,\vec{u}_{k}\}$, in the following way:
	\[
		\vec{u}_{i} = \sum_{j}S_{ij} \vec{v}_{j}\,,
	\]
	where $S_{ij}$ are the coefficients of the transformation matrix $S$ with $\mathrm{det}(S)\neq 0$.
	
	\begin{definition}[\textbf{Orientation}]\label{def:orientation}
		In a vector space of dimension $k$, orientation is an equivalence class of ordered sets of basis vectors whose equivalence relation states that two sets of basis vectors belong to the same equivalence class if the transformation matrix, $S$, between them has $\mathrm{det}(S)>0$.
	\end{definition}
	
	Since the determinant of a change of basis is either positive or negative, there are only two such classes. Hence, any vector space has only two orientations. If one chooses one of them, arbitrarily, the other one is said to be the opposite orientation. This also corresponds to assigning one of them as positive and the other as negative. It is simple to see that this definition of orientation is equivalent for the three cases presented above and generalizes this concept to any dimension $k$.
	
	As seen in \defref{def:manifold}, $k$-manifolds are locally like $T_p\mathbb{R}^k \cong\mathbb{R}^{k}$, and therefore, locally like a $k$-dimensional vector space. In this way, orientation on a manifold can be defined as:
	\begin{definition}[\textbf{Inner orientation on a manifold}]
		A manifold $\manifold{M}$ is said to be inner oriented if for any two overlapping charts, $(\mathcal{U}_{\alpha},\varphi_{\manifold{M},\alpha})$ and $(\mathcal{U}_{\beta},\varphi_{\manifold{M},\beta})$, of its atlas, $\mathcal{A}$, the Jacobian determinant of the transformation $\varphi_{\manifold{M},\beta}^{-1}\circ\varphi_{\manifold{M},\alpha}$ is positive. The inner orientation being the one of the equivalence classes of the sets of basis vectors of the tangent space at each point, $T_{p}\manifold{M}$, associated to these charts.
	\end{definition}
	
	If one considers the space in which the $k$-manifold $\manifold{M}$ is embedded to be $\mathbb{R}^{n}$, with dimension $n\geq k$, then at each point on the manifold there is a space perpendicular to the tangent space, $T_{p}\manifold{M}$, denoted by  $T^{\perp}_{p}\manifold{M}$ whose dimension is $(n-k)$; the normal bundle to the submanifold $\manifold{M}$ embedded in $\mathbb{R}^n$. Then $T_p\mathbb{R}^{n} = T_{p}\manifold{M}\oplus T^{\perp}_{p}\manifold{M}$. This allows one to define an outer orientation of a manifold as:
	\begin{definition}[\textbf{Outer orientation on a manifold}]
		Consider an oriented $k$-manifold $\manifold{M}$, with inner orientation $\{\vec{u}_{1},\vec{u}_{2},\cdots,\vec{u}_{k}\}$ at $T_p\manifold{M}$, embedded in an $n$-dimensional Euclidean space. An outer orientation of $\manifold{M}$ is an orientation for the perpendicular vector space at each point in the manifold, $T^{\perp}_{p}\manifold{M}$, $\{\vec{u}_{k+1},\vec{u}_{k+2},\cdots,\vec{u}_{n}\}$ such that all the oriented basis $\{\vec{u}_{1},\vec{u}_{2},\cdots,\vec{u}_{k}, \vec{u}_{k+1},\cdots,\vec{u}_{n}\}$ are in one of the two equivalence classes of the embedding space $T_p \mathbb{R}^n$.
	\end{definition}
	
	The particular cases of inner orientation of a 0-manifold (point) and outer orientation of a $n$-manifold embedded in a $n$-dimensional space are treated in a similar way. In both cases, the tangent space (points) and the perpendicular space ($n$-manifolds) have dimension zero. Therefore, the points and the $n$-manifolds are simply considered as sources or sinks and their orientation can be seen as simply induced by the inner orientation of the lines stemming out of them (points) or by the outer orientation of it's faces ($n$-manifold), see \cite{tonti1975formal,mattiussi2000finite}.
	
\begin{example}[\textbf{Outer orientation of points}]\label{ex:outer_orientation_points}
The outer orientation of points depends on the embedding space $\mathbb{R}^n$. In Figure~\ref{fig:diffGeom_orientationOuterPoint} a graphical representation of the outer orientation of a point in $\mathbb{R}^n$ is given for $n=0,\dots,3$
	\begin{figure}[!h]
		\centering
			\includegraphics[width=.6\textwidth]{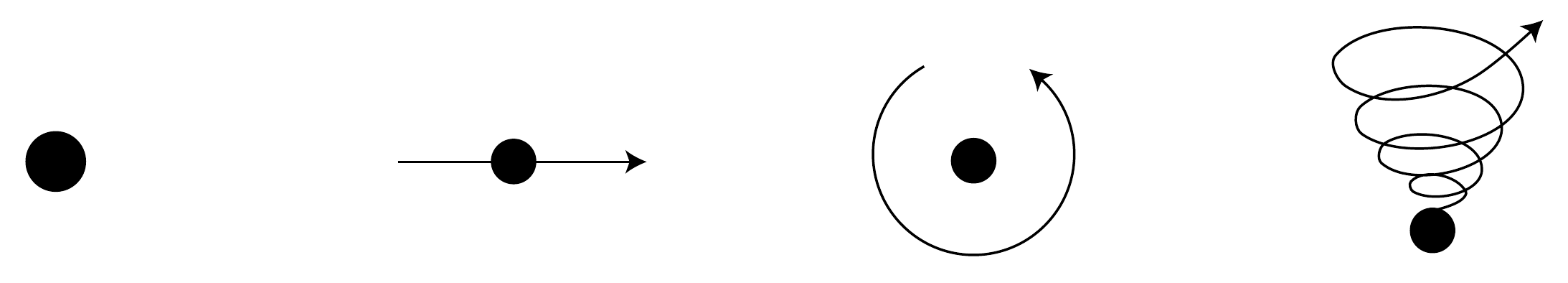}
			\caption{Outer orientation of a point embedded in $\mathbb{R}^{0}$, $\mathbb{R}^{1}$, $\mathbb{R}^{2}$ and $\mathbb{R}^{3}$ (from left to right).}
			\label{fig:diffGeom_orientationOuterPoint}
	\end{figure}
\end{example}	

\begin{example}[\textbf{Outer orientation of line segments}]\label{ex:outer_orientation_lines}
In Figure~\ref{fig:diffGeom_orientationOuterLine} the outer orientation of a line segment embedded in $\mathbb{R}^n$ is shown for $n=1,\ldots,3$.
\begin{figure}[!h]
		\centering
			\includegraphics[width=.6\textwidth]{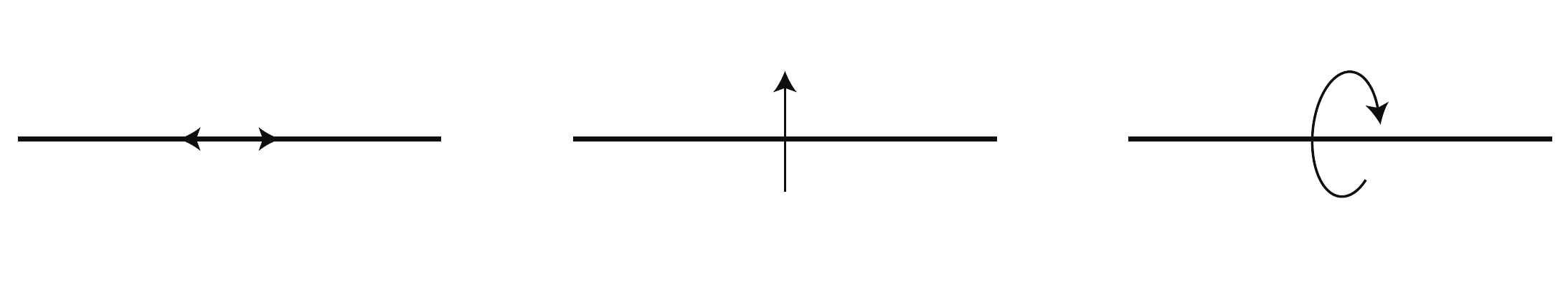}
			\caption{Outer orientation of a line segment embedded in $\mathbb{R}^{1}$, $\mathbb{R}^{2}$ and $\mathbb{R}^{3}$ (from left to right).}
			\label{fig:diffGeom_orientationOuterLine}
	\end{figure}
\end{example}

\begin{example}[\textbf{Outer orientation of surfaces}]	\label{ex:outer_orientation_surface}
Figure~\ref{fig:diffGeom_orientationOuterSurface} shows the embedding of a surface in $\mathbb{R}^n$, $n=2$ and $n=3$.
\begin{figure}[!h]
		\centering
			\includegraphics[height=.1\textheight]{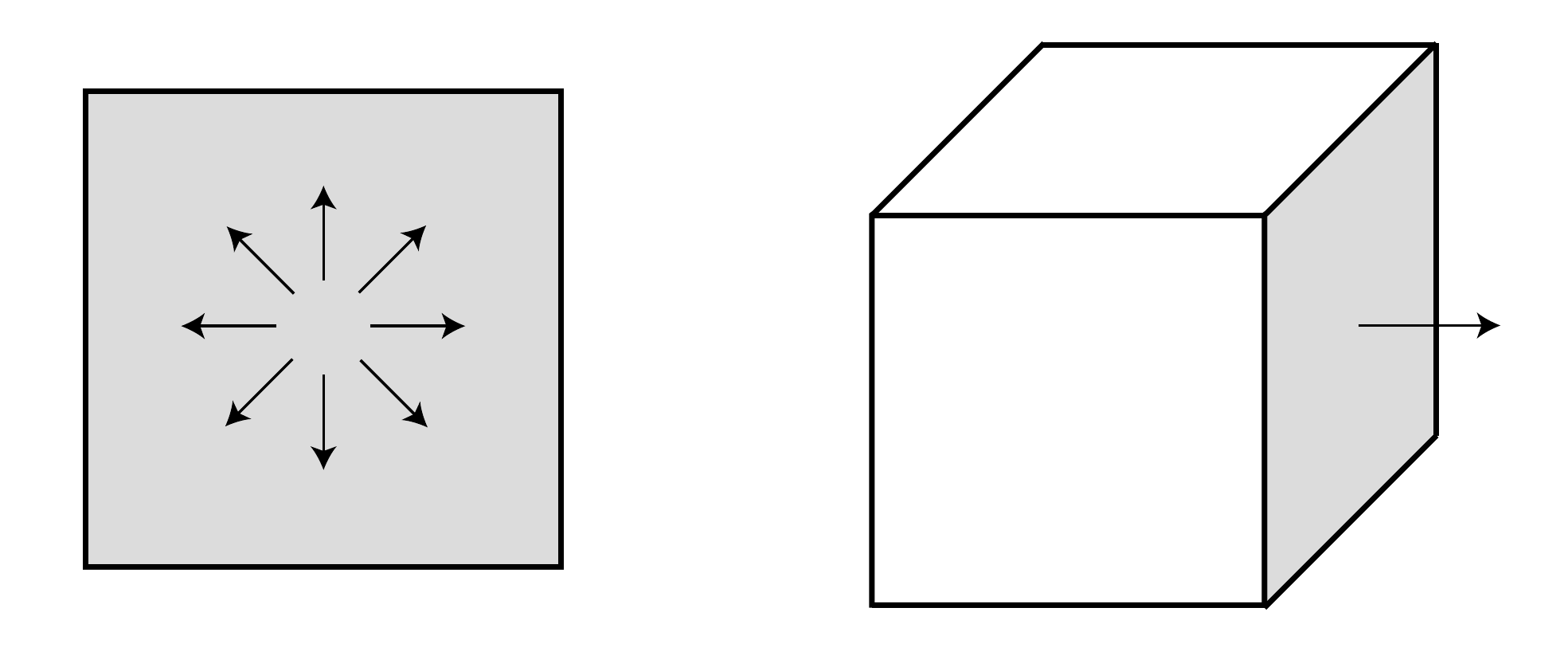}
			\caption{Outer orientation of a surface embedded in $\mathbb{R}^{2}$ and $\mathbb{R}^{3}$ (from left to right).}
			\label{fig:diffGeom_orientationOuterSurface}
	\end{figure}
\end{example}

\begin{example}[\textbf{Outer orientation of volumes}]	\label{ex:outer_orientation_volume}
Figure~\ref{fig:diffGeom_orientationOuterVolume} shows the outer orientation of a volume in $\mathbb{R}^3$.
\begin{figure}[!h]
		\centering
			\includegraphics[height=.1\textheight]{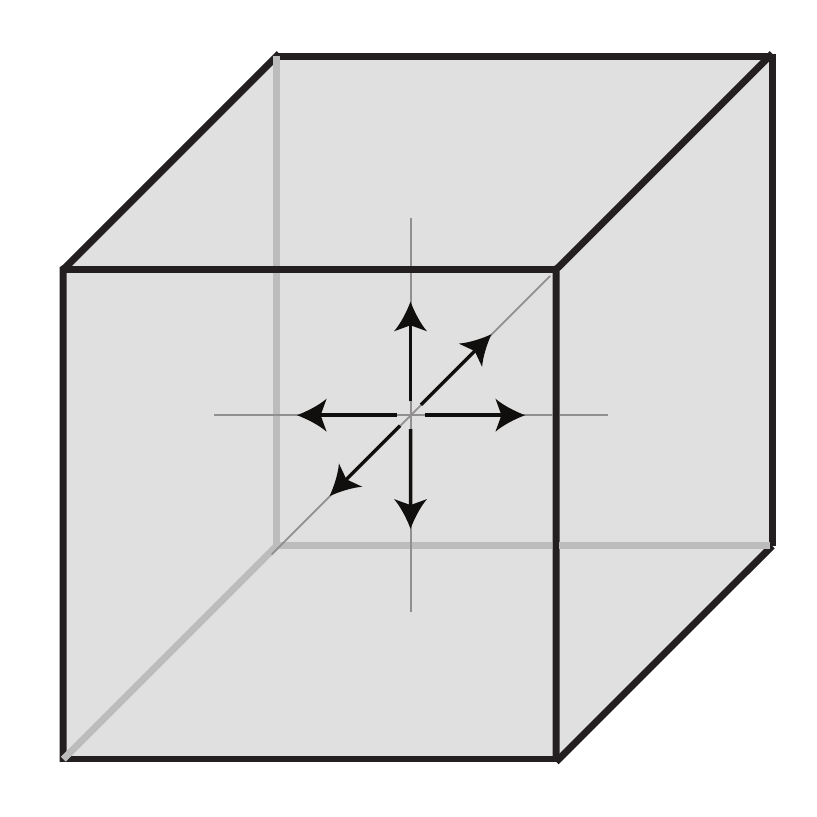}
			\caption{Outer orientation of a volume embedded in $\mathbb{R}^{3}$.}
			\label{fig:diffGeom_orientationOuterVolume}
	\end{figure}
\end{example}

	
	\figref{fig:diffGeom_orientation} presents above a sequence of outer-oriented geometric objects of increasing dimension, and below a sequence of inner-oriented geometric objects of decreasing order. The objects are aranged in such a way that it reveals the similarities with the double de Rham complex and action of the Hodge-$\star$ operator. Both are  introduced in this section later on.
	
	\begin{figure}[htp]
		\centering
			\includegraphics[width=0.7\textwidth]{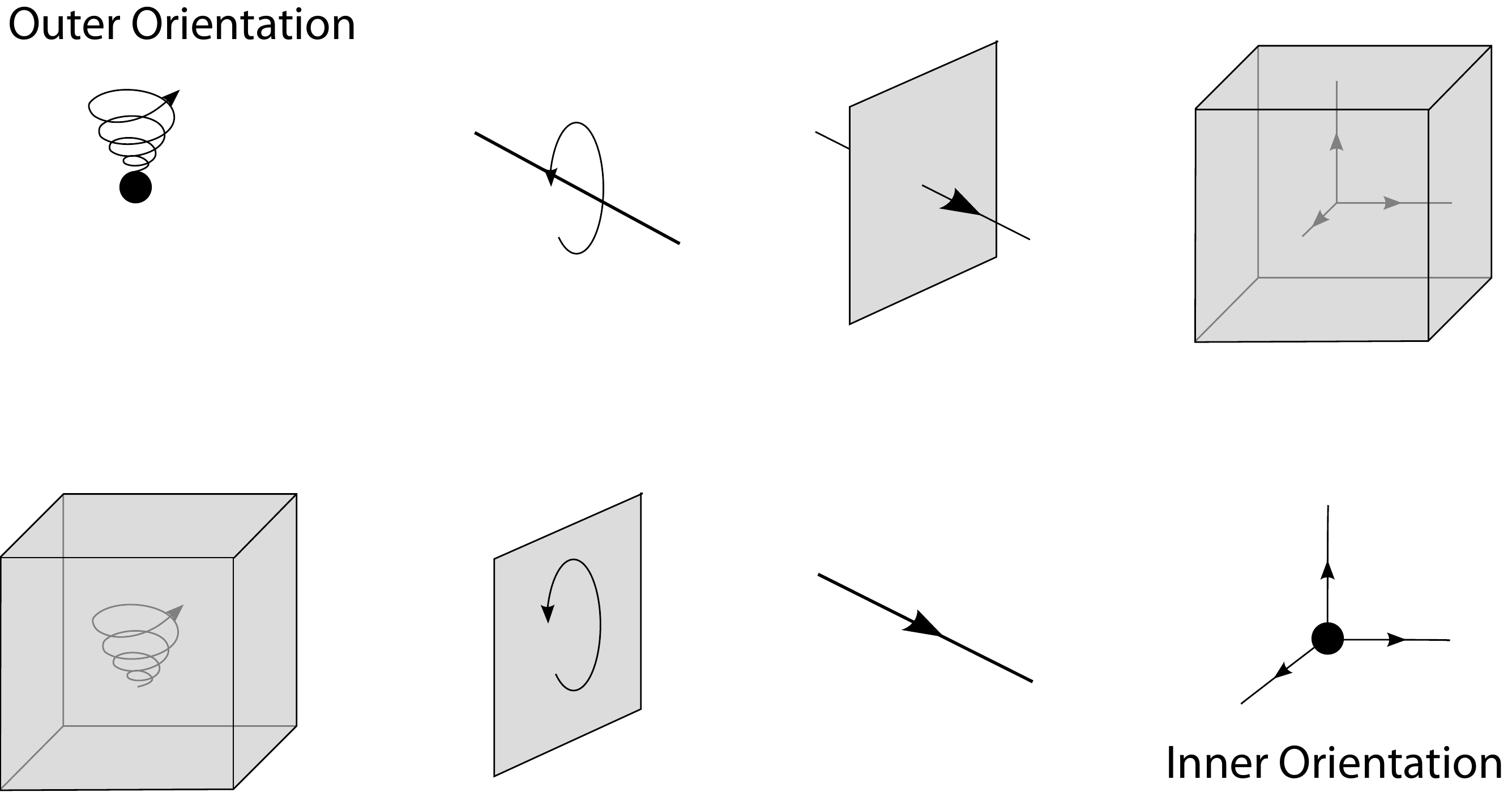}
			\caption{Inner and outer orientation of $k$-manifolds, $k=\{0,1,2,3\}$, in $\mathbb{R}^{3}.$}
			\label{fig:diffGeom_orientation}
	\end{figure}	
	
	Following the ideas pointed in \cite{bookTonti}, it is important to stress that orientation plays an essential role in integration. Although not explicitly expressed in the more common differential formulation, it appears implicitly by the usage of the right hand rule, for example, or when Stokes' theorem is considered, \theoremref{theor::difGeom_stokes_theorem}. In this case, the orientation of a $k$-manifold and its boundary have to be compatible when integrating, for the theorem to hold. Moreover, the distinction between inner and outer oriented manifolds will play a central role in the discretization to be presented in this paper. This distinction motivates the use of dual grids.
	
	This work is restricted to manifolds for which a global consistent orientation can be defined, called {\em orientable manifolds}.

      \subsection{Differential forms}
      		Differential forms will play a central role in the development of our numerical framework.
		\begin{definition}[\textbf{Differential forms}]
			\cite{abraham_diff_geom, isham,  spivak1998calculus} A differential $k$-form, $\kdifform{a}{k}$, $k\geq 1$ is a mapping:
			 \[
			 	\kdifform{a}{k}: \underbrace{\tangentspace{M}\times\cdots\times\tangentspace{M}}_{k}\spacemap\mathbb{R}\;,
			 \]
			 which is skew symmetric. That is, for any permutation $P$ of the indices $1,2,\cdots,k$:
			\[
				\kdifform{a}{k}(\vec{v}_{1},\cdots,\vec{v}_{k}) = \mathrm{sgn}(P)\,\kdifform{a}{k}(\vec{v}_{P(1)},\cdots,\vec{v}_{P(k)}) \;,
			\]
			for any $\vec{v}_{1},\cdots,\vec{v}_{k}\in\tangentspace{M}$, 
			see \cite{abraham_diff_geom}. A 0-form, $\kdifform{a}{0}$, is defined simply as a standard scalar function on $\manifold{M}$. The space of $k$-forms on the manifold $\manifold{M}$ is denoted by $\Lambda^{k}(\manifold{M})$. $\kdifform{a}{k} = 0$ when $k<0$ or $k>n$.
		\end{definition}
      		
		
		\begin{definition}[\textbf{Wedge product}]
			\cite{abraham_diff_geom, spivak1998calculus} The wedge product, $\wedge$, of two differential forms $\kdifform{a}{k}\in\Lambda^{k}(\manifold{M})$ and $\kdifform{b}{l}\in\Lambda^{l}(\manifold{M})$ is a mapping $\wedge : \Lambda^{k}(\manifold{M})\times\Lambda^{l}(\manifold{M})\spacemap\Lambda^{k+l}(\manifold{M})$, such that:
			\begin{subequations}
			\begin{align}
				(\kdifform{a}{k} + \kdifform{b}{l})\wedge\kdifform{c}{m} = \kdifform{a}{k}\wedge\kdifform{c}{m} + \kdifform{b}{l}\wedge\kdifform{c}{m} & \quad \text{(Distributivity)} \\
				(\kdifform{a}{k}\wedge\kdifform{b}{l})\wedge\kdifform{c}{m} = \kdifform{a}{k}\wedge(\kdifform{b}{l}\wedge\kdifform{c}{m}) & \nonumber \\
				 = \kdifform{a}{k}\wedge\kdifform{b}{l}\wedge\kdifform{c}{m} & \quad \text{(Associativity)} \label{wedgeassociativity}\\
				\alpha \kdifform{a}{k}\wedge\kdifform{b}{l} = \kdifform{a}{k}\wedge \alpha\kdifform{b}{l} = \alpha (\kdifform{a}{k}\wedge\kdifform{b}{l}) & \quad \text{(Multiplication by scalars)} \\
				\kdifform{a}{k}\wedge\kdifform{b}{l} = (-1)^{kl}\kdifform{b}{l}\wedge\kdifform{a}{k}& \quad \text{(Skew symmetry)} \label{wedge::skew_symmetry}
			\end{align}
			\end{subequations}
		\end{definition}
		From property \eqref{wedge::skew_symmetry}
		we we have
		\begin{equation}
			\kdifform{a}{k}\wedge\kdifform{a}{k} = 0,\quad \forall \kdifform{a}{k}\in\Lambda^{k}, \quad k \text{ is odd or } k > \frac{n}{2} \;. \label{awedgea}
		\end{equation}
				\begin{proposition}\cite{abraham_diff_geom}\label{diffGeometry::basis_elements}
			On a manifold $\manifold{M}$ of dimension $n$ the space of $1$-forms $\Lambda^{1}(\manifold{M})$ is a linear vector space of dimension $n$ that is spanned by $n$ basis elements. A canonical basis, given a local coordinate system $(x^{1}, \cdots, x^{n})$ is $\{\ederiv x^{1}, \cdots, \ederiv x^{n}\}$.
			The space of $k$-forms $\Lambda^{k}(\manifold{M})$ is a linear vector space of dimension $\frac{n!}{(n-k)!k!}$ and has a canonical basis given by:
			\[
				\{\ederiv x^{i_{1}}\wedge\cdots\wedge\ederiv x^{i_{k}} | 1\leq i_{1} < \cdots < i_{k}\leq n\} \;.
			\]	
		\end{proposition}
		
		\begin{example}
			Examples of 0-forms, 1-forms, 2-forms and 3-forms are given by:
			\begin{equation}
				\left\{
					\begin{array}{l}
						\kdifform{a}{0} = a(x,y,z) \\
						\kdifform{b}{1} = b_1(x,y,z) \ederiv x + b_2(x,y,z) \ederiv y + b_3(x,y,z) \ederiv z \\
						\kdifform{c}{2} = c_1(x,y,z) \ederiv y\!\wedge\!\ederiv z + c_2(x,y,z) \ederiv z\!\wedge\!\ederiv x + c_3(x,y,z) \ederiv x\!\wedge\!\ederiv y \\
						\kdifform{w}{3} = w(x,y,z) \ederiv x\!\wedge\!\ederiv y\!\wedge\!\ederiv z
					\end{array}
				\right.
			\end{equation}
		\end{example}

		\begin{example}[\textbf{$\kdifform{a}{1} \wedge \kdifform{a}{1}$ in $\mathbb{R}^3$}] $\ederiv x \wedge\ederiv x = \ederiv y \wedge\ederiv y = \ederiv z \wedge\ederiv z = 0 \;$.
		\end{example}
		\begin{example}[\textbf{$\kdifform{a}{2} \wedge \kdifform{a}{2}$ in $\mathbb{R}^4$}]
		In $\mathbb{R}^{4}$,
		 let $\kdifform{\alpha}{2} = \ederiv x \wedge \ederiv y + \ederiv z \wedge \ederiv t$, then
		\[
			\kdifform{a}{2} \wedge \kdifform{a}{2} = (\ederiv x \wedge \ederiv y + \ederiv z\wedge \ederiv t) \wedge (\ederiv x \wedge \ederiv y + \ederiv z\wedge \ederiv t) = 2\, \ederiv x \wedge \ederiv y \wedge \ederiv z\wedge \ederiv t \neq 0 \;.
		\]
		\end{example}
		
		
		\begin{definition}[\textbf{Inner product $k$-forms}] \label{def:inner_product_diff_forms}
			\cite{morita::geometry_diff_forms} The space of $k$-forms, $\Lambda^{k}(\manifold{M})$, can be equipped with a pointwise positive-definite inner product, $\inner{\cdot}{\cdot}:\Lambda^{k}(\manifold{M})\times\Lambda^{k}(\manifold{M})\spacemap\mathbb{R}$, such that:
			\[
				\inner{\kdifform{a}{1}_{1}\wedge\cdots\wedge\kdifform{a}{1}_{k}}{\kdifform{b}{1}_{1}\wedge\cdots\wedge\kdifform{b}{1}_{k}} := 
				\left|
					\left[
						\begin{array}{ccc} 
							\inner{\kdifform{a}{1}_{1}}{\kdifform{b}{1}_{1}} & \cdots & \inner{\kdifform{a}{1}_{k}}{\kdifform{b}{1}_{1}} \\
							\vdots & \vdots & \vdots \\
							\inner{\kdifform{a}{1}_{1}}{\kdifform{b}{1}_{k}} & \cdots & \inner{\kdifform{a}{1}_{k}}{\kdifform{b}{1}_{k}}
						\end{array}
					\right]
				\right| \;,
			\]
			where $\kdifform{a}{1}_{i},\kdifform{b}{1}_{j}\in\Lambda^{1}(\manifold{M})$, $i,j=1,\cdots,k$ and $\inner{\kdifform{a}{1}_{i}}{\kdifform{b}{1}_{j}}$ is the inner product of 1-forms, induced by the inner product on tangent vectors and by the duality pairing between vectors and 1-forms, see \cite{abraham_diff_geom,isham, spivak1998calculus}, and given by:
			\[
				\inner{\kdifform{a}{1}}{\kdifform{b}{1}} := \sum_{i,j}a_{i}b_{j}g^{ij} \;,
			\]
			where $\kdifform{a}{1} = \sum_{i}a_{i}\mathrm{d}x^{i}$, $\kdifform{b}{1} = \sum_{j}b_{j}\mathrm{d}x^{j}$ and $g^{ij}$ are the coefficients of the inverse of the metric tensor of rank two, see \cite{abraham_diff_geom,isham, spivak1998calculus}.
		\end{definition}
		
	\subsection{Differential forms under mappings}
	It is also important to determine how differential forms transform under a mapping between two manifolds, see \figref{fig::Change_of_coordinates} for an example for  $n=2$ and where $(x,y)=\mapping(\xi,\eta)$, and how integration of $k$-forms over manifolds is defined.
		
      \begin{definition}[\textbf{Pullback operator}] \label{def:pullback_kform}
      		\cite{abraham_diff_geom,flanders::diff_forms} Consider two manifolds of dimension $n$, $\manifold{M}$ and $\manifold{N}$ and a mapping between them, $\mapping: \manifold{M}\spacemap\manifold{N}$, such that local coordinates $\xi^{i}$ in $\manifold{M}$ are mapped into local coordinates $x^{i}=\mapping^{i}(\xi^{1},\cdots,\xi^{n})$ in $\manifold{N}$. Then the pullback, $\pullback$, of a $k$-form, $k\geq1$, $\kdifform{a}{k}$ is given by:
		\[
			\pullback(\kdifform{a}{k})(\vec{v}_{1}, \cdots, \vec{v}_{k}) := \kdifform{a}{k}(\pushforward(\vec{v}_{1}), \cdots, \pushforward(\vec{v}_{k})) \;,
		\]
		where $\pushforward(\vec{v})$ is the usual pushforward of a vector $\vec{v}$, see \cite{abraham_diff_geom,flanders::diff_forms}. The pullback of a 0-form, $\kdifform{a}{0}\in\Lambda^{0}(\manifold{N})$, is given simply by the composition of the maps
		\[
			\pullback (\kdifform{a}{0}) := a^{(0)}\circ\mapping \;.
		\]
      \end{definition}
      
      For a 1-form we get the following.
      \begin{proposition}\cite{flanders::diff_forms}
      		If $\kdifform{a}{1}\in\Lambda^{1}(\manifold{N})$ is given in a local coordinate system as $\kdifform{a}{1}=\sum_{i}\alpha_{i}\ederiv x^{i}$ then the pullback $\pullback:\Lambda^{1}(\manifold{N})\rightarrow\Lambda^{1}(\manifold{M})$ is given by:
		\begin{equation}
			\pullback(\kdifform{a}{1}) = \sum_{i,k}a_{i}\pderivative{\Phi^{i}}{\xi^{k}}\ederiv \xi^{k} \;. \label{eq::diffGeometry_pullback_one_form}
		\end{equation}
		
      \end{proposition}
	
	\begin{proposition}\cite{hou::differential_geometry_physicists}
		The pullback $\pullback$ has the following properties:
		\begin{subequations}
		\begin{align}
			\pullback(\kdifform{a}{k}+\kdifform{b}{k}) = \pullback(\kdifform{a}{k})+\pullback(\kdifform{b}{k}) & \quad \text{(Linearity)} \label{diffGeom::pullback_linearity} \\
			\pullback(\kdifform{a}{k}\wedge\kdifform{b}{k}) = \pullback(\kdifform{a}{k})\wedge\pullback(\kdifform{b}{k}) & \quad \text{(Algebra homomorphism)} \label{diffGeom::pullback_homomorphism} \\
			\pullbacks{(\mapping_{2}\circ\mapping_{1})} = \pullback_{1}\circ\pullback_{2} & \quad\text{(Composition)} \label{diffGeom::pullback_composition}
		\end{align}
		\end{subequations}
	\end{proposition}
	
	Integration of differential forms can now be defined in the following way.
	\begin{definition}[\textbf{Pullback integration formulation}] \label{def:pullback} \cite{spivak1998calculus}
	The integral of a differential $k$-form, $\kdifform{a}{k}$, over a manifold $\manifold{M}$ of dimension $k$ is:
		\begin{equation}
			\int_{\manifold{M}}\kdifform{a}{k} := \int_{S}\left ( \varphi^{-1} \right )^\star (\kdifform{a}{k}) \;, \label{eq:change_coordinates_integration}
		\end{equation}
		where $\varphi^{-1}: S \subset \mathbb{R}^{k}\spacemap\manifold{M}$ is the inverse of the global chart from $\manifold{M}$ to $S \subset \mathbb{R}^k$.  On the right hand side of (\ref{eq:change_coordinates_integration}) one has the usual integral in $S \subset \mathbb{R}^{k}$. 
	\end{definition}
		
		Integration can be considered a duality pairing between a differential $k$-form and a $k$-dimensional manifold in the following way
		\begin{equation}
			\duality{\kdifform{a}{k}}{\manifold{M}} := \int_{\manifold{M}}\kdifform{a}{k} \;. \label{diffgeom::duality_pairing}
		\end{equation}
		
		If integration is interpreted as the duality pairing between differential forms and geometry, then the relation between a mapping and the associated pullback satisfies
		\begin{proposition}\label{prop::pullback_integral}
			Given a mapping $\mapping:\manifold{M}\spacemap\manifold{N}$, its associated pullback $\pullback$ and a differential form $\kdifform{a}{k} \in \Lambda^{k}(\manifold{N})$ then the following holds:
			\begin{equation}
				\int_{\mapping(\manifold{M})}\kdifform{a}{k} = \int_{\manifold{M}}\pullback(\kdifform{a}{k})  \quad \Leftrightarrow \quad \duality{\kdifform{a}{k}}{\mapping(\manifold{M})} = \duality{\pullback\kdifform{a}{k}}{\manifold{M}}\;. \label{eq::diffGeom_pullback_integral_manifolds}
			\end{equation}
			So the pullback is the {\em formal adjoint} of the map $\Phi$ in this duality pairing.
			\end{proposition}
			
%
%
%
		
		\begin{definition}[\textbf{Inclusion map}] \label{def:inclusion_map}
			\cite{kubrusly2001elements} Let $\manifold{A}$ be a subset of $\manifold{M}$. The function $\iota: \manifold{A}\spacemap \manifold{M}$ defined by $\iota(x) = x$ for every $x\in \manifold{A}$, is the the inclusion map (or the embedding, or the injection) of $\manifold{A}$ into $\manifold{M}$. In other words, the inclusion map of a subset of $\manifold{M}$ is the restriction to that subset of the identity map on $\manifold{M}$.
		\end{definition}
		
		
		\begin{definition}[\textbf{Trace operator}]\label{def::diffGeometry_trace_operator}
			\cite{arnold2009geometric} Given two manifolds $\manifold{M}$ and $\manifold{M}'$ such that $\manifold{M}'\subset\manifold{M}$, the trace operator, $\mathrm{tr}$:
			\[
				\mathrm{tr}_{\manifold{M},\manifold{M}'}:\Lambda^{k}(\manifold{M})\rightarrow\Lambda^{k}(\manifold{M}') \;,
			\]
			is the pullback of the inclusion $\manifold{M}'\hookrightarrow\manifold{M}$. If the manifold $\manifold{M}$ is clear from the context, one may write $\mathrm{tr}_{\manifold{M}'}$ instead of $\mathrm{tr}_{\manifold{M},\manifold{M}'}$. If $\manifold{M}'$ is the boundary of $\manifold{M}$, $\partial\manifold{M}$, one just writes $\mathrm{tr}$.
		\end{definition}
		
		\begin{example}
			Consider an inclusion map $\iota$ and its associated pullback $\iota^{\star}$, such as the one depicted in \figref{fig::diffGeometry_inclusionMap} that generates the inclusion of the manifold $\mathcal{M}'$, a circle, on the manifold $\manifold{M}$, a disk. In local polar coordinates $\theta'$ and $(\theta, r)$ the inclusion map takes the form:
			\[
				\left\{
					\begin{array}{l}
						\theta = \iota_{\theta}(\theta') = \theta' \\
						r =  \iota_{r}(\theta') = 1 \\
					\end{array}
				\right. \;.
			\]
			
			\begin{figure}[htp]
				\centering
					\includegraphics[width=0.3\textwidth]{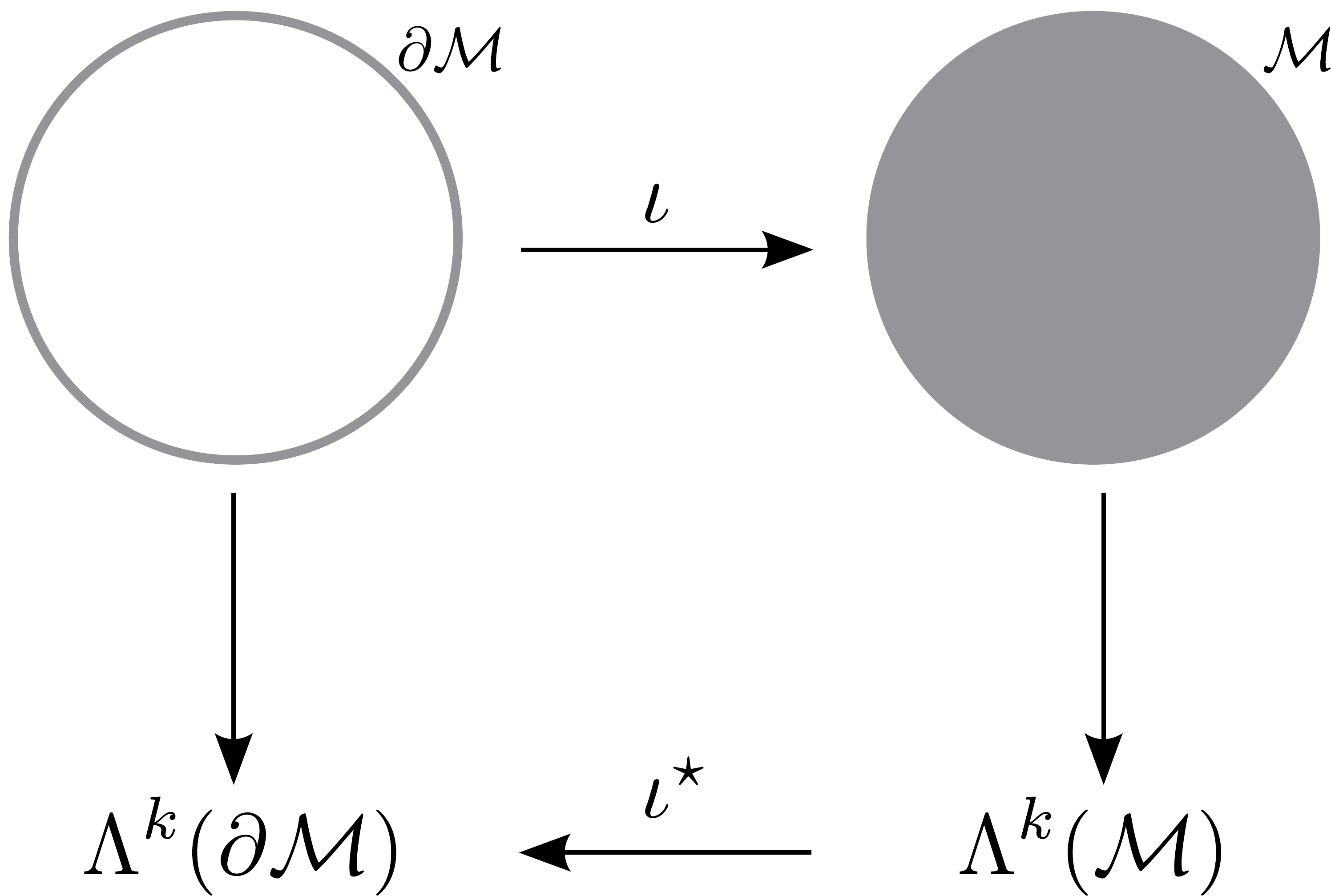}
					\caption{Pictorial view of the inclusion map $\iota$ and its associated pullback $\iota^{\star}$.}
					\label{fig::diffGeometry_inclusionMap}
			\end{figure}	
			The trace of the 1-form $\kdifform{a}{1} = a_{\theta}\mathrm{d}\theta + a_{r}\mathrm{d}r$ is given by:
			\[
				\mathrm{tr}\,\kdifform{a}{1} \stackrel{\mathrm{def.}~\ref{def::diffGeometry_trace_operator}}{=} \iota^{\star}\kdifform{a}{1} \stackrel{(\ref{eq::diffGeometry_pullback_one_form})}{=}(a_{\theta}\circ\iota)(\theta')\mathrm{d}\theta' \;.
			\]
		\end{example}
			
	\subsection{Exterior derivative}
		The exterior derivative, $\ederiv$, plays an important role in differential geometry and is defined in the following way:
		\begin{definition}[\textbf{Exterior derivative}]\label{def:exterior_derivative}
		\cite{abraham_diff_geom,flanders::diff_forms} The exterior derivative on a $n$-dimensional manifold $\manifold{M}$ is a mapping $\ederiv : \Lambda^{k}(\manifold{M})\rightarrow\Lambda^{k+1}(\manifold{M})$, $0\leq k\leq n-1$, which satisfies:
		      \begin{equation}
		          \ederiv\left(\kdifform{a}{k}\wedge\kdifform{b}{l}\right) = \ederiv\kdifform{a}{k}\wedge\kdifform{b}{l} + (-1)^{k}\kdifform{a}{k}\wedge \ederiv\kdifform{b}{l}, \qquad k+l<n \;, \label{eq::dif_and_wedge}
		      \end{equation}
		      and is nilpotent,
		      \begin{equation}
		      		\ederiv\ederiv\kdifform{a}{k}:=0,\quad\forall\kdifform{a}{k}\in\Lambda^{k}(\mathcal{M}) \;. \label{eq:double_dif}
		      \end{equation}
		      For $0$-forms, $a^{(0)}\in\Lambda^{0}(\manifold{M})$, the exterior derivative in a local coordinate system $(x^{1},\cdots,x^{n})$, is given by:
		      \begin{equation}
		      		\ederiv\kdifform{a}{0} :=  \sum_{i}\pderivative{a}{x^{i}}\,\ederiv x^{i} \label{eq::dzero_forms} \;.
		      \end{equation}
		      \end{definition}
		For a local coordinate system one has:
		\begin{example}
			In a 3-dimensional Euclidean space and in a local coordinate system $(x^{1}, x^{2},x^{3})$ the exterior derivative of a 1-form, $\kdifform{a}{1}=\sum_{i=1}^{3}a_{i}\ederiv x^{i}$, is given by:
			\begin{equation}
				\ederiv \kdifform{a}{1} = \left(\pderivative{a_{2}}{x^{3}} - \pderivative{a_{3}}{x^{2}}\right)\ederiv x^{2}\wedge\ederiv x^{3} + \left(\pderivative{a_{1}}{x^{3}} - \pderivative{a_{3}}{x^{1}}\right)\ederiv x^{3}\wedge\ederiv x^{1} + \left(\pderivative{a_{2}}{x^{1}} - \pderivative{a_{1}}{x^{2}}\right)\ederiv x^{1}\wedge\ederiv x^{2} \;.\label{eq::done_forms}
			\end{equation}
			And the exterior derivative of a 2-form, $\kdifform{b}{2}=b_{1}\ederiv x^{2}\wedge\ederiv x^{3} + b_{2}\ederiv x^{3}\wedge\ederiv x^{1} + b_{3}\ederiv x^{1}\wedge\ederiv x^{2}$, is given by:
			\begin{equation}
				\ederiv \kdifform{b}{2} = \left(\pderivative{b_{1}}{x^{1}}+\pderivative{b_{2}}{x^{2}}+\pderivative{b_{3}}{x^{3}}\right)\ederiv x^{1}\wedge\ederiv x^{2}\wedge\ederiv x^{3} \;. \label{eq::dtwo_forms}
			\end{equation}
			
			Recalling the vector calculus operators, gradient, curl, and divergence, one can see the similarity between these and the expressions \eqref{eq::dzero_forms}, \eqref{eq::done_forms} and \eqref{eq::dtwo_forms}, respectively. Indeed there exists a clear relation between the two sets of equations, for more details see \cite{abraham_diff_geom,flanders::diff_forms}.
		\end{example}
		
		It is possible to show, \cite{abraham_diff_geom,spivak1998calculus}, that the exterior derivative and the integration of a differential form are related by the following result:
		\begin{theorem}[\textbf{Generalized Stokes' Theorem}]\label{theor::difGeom_stokes_theorem}
			\cite{abraham_diff_geom,spivak1998calculus} Given a $k$-form $\kdifform{a}{k}$ on a (sub)-manifold $\manifold{M}$ of dimension $k+1$ that is paracompact and has a boundary, then
			\begin{equation}
				\int_{\manifold{M}}\ederiv\kdifform{a}{k} = \int_{\partial\manifold{M}}\kdifform{a}{k}. \label{eq::difGeom_stokes_theorem}
			\end{equation}
		\end{theorem}
		Two important aspects of this theorem can be stated. One is the fact that it condenses and generalizes three key theorems of vector calculus: the fundamental theorem of calculus, Stokes' theorem and Gauss' theorem, for 3-dimensional space, when applied to $0$-forms, $1$-forms and $2$-forms, respectively. The second is the fact that the exterior derivative is the formal adjoint of the boundary operator:
		\begin{equation}
			\ederiv = \partial^{\dag}  \;.\label{eq:exterior_derivative_transpose}
		\end{equation}
		This property can be obtained by \eqref{diffgeom::duality_pairing} and Stokes's Theorem:
		\begin{equation}
			\duality{\ederiv\kdifform{a}{k}}{\manifold{M}} \stackrel{(\ref{diffgeom::duality_pairing})}{=} \int_{\manifold{M}}\ederiv\kdifform{a}{k} \stackrel{(\ref{eq::difGeom_stokes_theorem})}{=} \int_{\partial\manifold{M}}\kdifform{a}{k} \stackrel{(\ref{diffgeom::duality_pairing})}{=} \duality{\kdifform{a}{k}}{\partial\manifold{M}} \;. \label{eq::diffGeometry_stokes_duality}
		\end{equation}
		This is one of the basic identities of the numerical framework to be presented.
		
		\begin{remark}\label{maifold_independence_generalized_Stokes}
Another important property of the generalized Stokes' Theorem is that if ${\mathcal M}$ and $\tilde{\mathcal M}$ are two manifolds of dimension $k+1$ with the same boundary, i.e. $\partial {\mathcal M} \equiv \partial \tilde{\mathcal M}$, we have
\begin{equation}
\int_{\mathcal M} \ederiv \kcochain{a}{k} = \int_{\partial {\mathcal M}} \kcochain{a}{k} = \int_{\partial \tilde{\mathcal M}} \kcochain{a}{k} = \int_{\tilde{\mathcal M}} \ederiv \kcochain{a}{k} \;,\;\;\; \forall \kcochain{a}{k} \in (\Lambda^k(\manifold{M})\cup\Lambda^k(\tilde{\manifold{M}}))\;.
\end{equation}
Another way of expressing this particular independence of the manifold is that we can always add to the manilfold ${\mathcal M}$ a manifold $\bar{\mathcal M}$ with $\partial \bar{\mathcal M}=0$ without effecting any change in the generalized Stokes' Theorem. For $k=0$, this corresponds to the gradient theorem which states that if two points $A$ and $B$ are connected by a curve ${\mathcal C}$, then
\begin{equation} \int_{\mathcal C} \nabla \phi \cdot ds = \phi(B) - \phi(A) \;.\label{classical_gradient_theorem}\end{equation}
This theorem holds for any curve connecting the points $A$ and $B$, provided that the domain is contractible. There is no preferred curve and therefore is seems reasonable to identify all curves ${\mathcal C}$ which satisfy (\ref{classical_gradient_theorem}). Such an identification will be formalized in the next section.
\end{remark}
		
		The exterior derivative satisfies the following commuting property:
		\begin{proposition}\label{prop:commutation_pullback_ederiv}
			Given a mapping $\mapping:\manifold{M}\spacemap\manifold{N}$ and the associated pullback, $\pullback$, the exterior derivative commutes with the pullback:
			\begin{equation}
				\pullback(\ederiv \kdifform{a}{k}) = \ederiv\pullback (\kdifform{a}{k}),\quad\forall \kdifform{a}{k}\in\Lambda^{k}(\mathcal{M}).
			\end{equation}
		and is illustrated as
		\[\begin{CD}
		\Lambda^k(\mathcal{M})@>\ederiv >>\Lambda^k(\mathcal{M})\\
		@V\pullback VV @V\pullback VV\\
		\Lambda^k(\mathcal{N})@>\ederiv >>\Lambda^k(\mathcal{N}).
		\end{CD}\]
		\end{proposition}
		\begin{proof}
			For all sub-manifolds $\manifold{A}$ in $\manifold{M}$, we have
			\[
				\duality{\pullback\ederiv\kdifform{a}{k}}{\manifold{A}} \stackrel{(\ref{eq::diffGeom_pullback_integral_manifolds})}{=} \duality{\ederiv\kdifform{a}{k}}{\mapping \left ( \manifold{A} \right ) } \stackrel{(\ref{eq::diffGeometry_stokes_duality})}{=} \duality{\kdifform{a}{k}}{\partial\mapping(\manifold{A})} \stackrel{(\ref{eq::difGeom_mapping_boundary})}{=} \duality{\kdifform{a}{k}}{\mapping(\partial\manifold{A})} = \duality{\ederiv\pullback\kdifform{a}{k}}{\manifold{A}} \;.
			\]
			Since $\manifold{A}$ was completely arbitrary and it needs to hold for all $\kdifform{a}{k}\in\Lambda^{k}(\mathcal{M})$, we have $\pullback \ederiv = \ederiv\pullback$. See \cite{flanders::diff_forms,tu2010introduction} for alternative proofs. 
			\end{proof}
			\begin{definition}
		A differential form $\kdifform{a}{k}$ is called {\em exact} if there exists a differential form $\kdifform{b}{k-1}$ such that $\kdifform{a}{k}=\ederiv\kdifform{b}{k-1}$. The space of exact $k$-forms is the range of the exterior derivative, i.e. \[\mathcal{B}(\ederiv;\Lambda^{k-1}):=\ederiv\Lambda^{k-1}(\mathcal{M})\subset\Lambda^k(\mathcal{M}).\]
		 A differential form $\kdifform{a}{k}$ is called {\em closed} if $\ederiv\kdifform{a}{k}=0$. The space of closed $k$-forms is the nullspace or kernel of the exterior derivative, i.e. 
		 \[ \mathcal{Z}(\ederiv;\Lambda^k):=\{\forall a^{(k)}\in\Lambda^k(\mathcal{M})\;|\;\ederiv a^{(k)}=0\}\subset\Lambda^k(\mathcal{M}).\]
		 \end{definition}
		  It follows from \eqref{eq:double_dif} that all exact differential forms are closed. The reverse is only true on contractible manifolds, and is known as Poincar\'e lemma.
		\begin{lemma}[\textbf{Poincar\'{e} Lemma}] \label{lemma:poincare} \cite{abraham_diff_geom}
		On a contractible manifold all closed differential forms are exact. That is:
		\[
		\text{For all}\ \kdifform{a}{k}\in\mathcal{Z}(\ederiv;\Lambda^k),\ \text{there exists}\ \kdifform{b}{k-1}\in\Lambda^{k-1}(\mathcal{M}) \text{ such that } \kdifform{a}{k} = \ederiv\kdifform{b}{k-1} \;.
		\]
		\end{lemma}
		
		\begin{remark}[\textbf{Potentials}]
		The relevance of Poincar\'{e}'s lemma has far-reaching implications. In electromagnetics it guarantees the existence of a solution to $\mathrm{curl}\,\vec{A} = \vec{B}$, where $\vec{B}$ is a given magnetic flux density and $\vec{A}$ is the magnetic potential. It also guarantees the existence of functions which give rise to the equations of motion of a mechanical system in Hamiltonian form. Although the Poincar\'{e} lemma is stated in terms of differential forms, the natural isomorphism between differential forms and vector fields allows one to state that the Poincar\'{e} lemma carries over to all the existence theorems of potential fields of vector calculus. That is, the existence of solutions for all the following equations, for $V$, $\vec{A}$ and $\vec{C}$: $\mathrm{grad}\,V = f$, $\mathrm{curl}\,\vec{A} = \vec{B}$ and $\mathrm{div}\,\vec{C} = f$.
		\end{remark}

		From \defref{def:exterior_derivative} the $(n+1)$-spaces of differential forms in an $n$-dimensional manifold $\manifold{M}$ satisfy the following sequence, called the {\em de Rham complex}, denoted by $(\Lambda,\ederiv))$:
\begin{equation}
\begin{array}{cccccccccccc}
\mathbb{R}\; \longrightarrow&\Lambda^0(\manifold{M})
&\stackrel{\ederiv}{\longrightarrow}& \Lambda^1(\manifold{M})
&\stackrel{\ederiv}{\longrightarrow} & \cdots \; &\stackrel{\ederiv}{\longrightarrow}\; &\Lambda^n(\manifold{M}) \;
&\stackrel{\ederiv}{\longrightarrow}\; &0 \;.
\end{array}
\label{singlecomplex}
\end{equation}
This sequence is {\em exact} on contractible domains, in which case $\mathcal{B}(\ederiv;\Lambda^{k-1})= \mathcal{Z}(\ederiv;\Lambda^{k})$. In general we have that $\mathcal{B}(\ederiv;\Lambda^{k-1})\subset \mathcal{Z}(\ederiv;\Lambda^{k})$, as depicted in \figref{fig::diffGeometry_deRhamComplex}.
\begin{figure}[htp]
\centering
\includegraphics[width=1\textwidth]{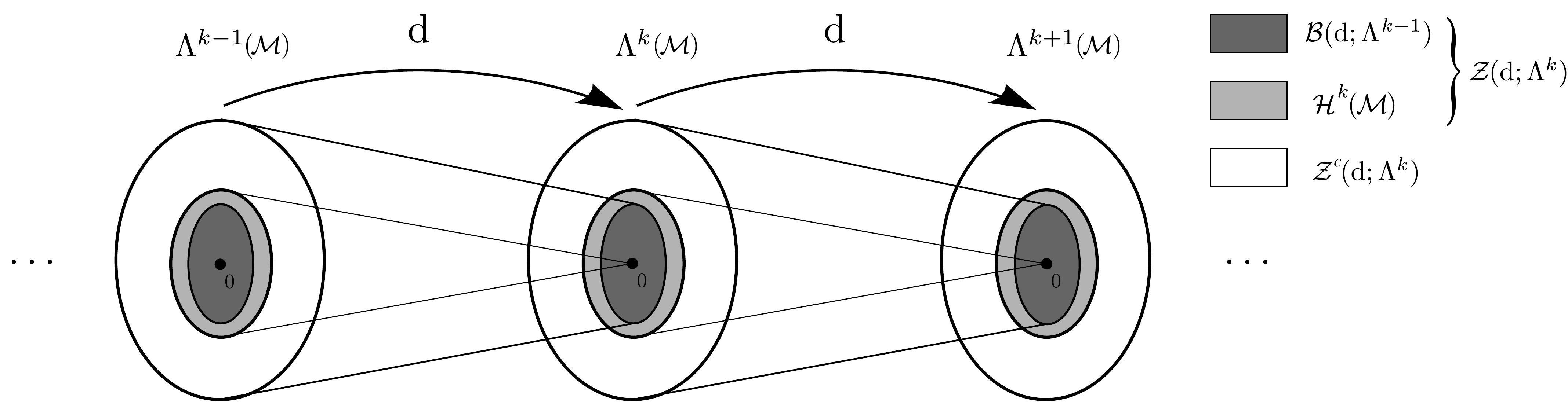}
\caption{Pictorial view of de Rham complex on differential forms. The exterior derivative $\ederiv$ maps the elements of $\Lambda^{k}$ into the kernel of $\ederiv$ of $\Lambda^{k+1}$, $\mathcal{B}(\ederiv;\Lambda^{k})\subset \mathcal{Z}(\ederiv;\Lambda^{k+1})$.}
\label{fig::diffGeometry_deRhamComplex}
\end{figure}
Hence, one can decompose the space of $k$-forms, $\Lambda^{k}(\manifold{M})$, as:
\begin{equation}
\Lambda^{k}(\manifold{M}) = \mathcal{B}(\ederiv; \Lambda^{k-1}) \oplus \mathcal{B}^{c}(\ederiv; \Lambda^{k-1}),
\label{hdranges}
\end{equation}
where $\mathcal{B}^{c}(\ederiv; \Lambda^{k-1})$ is the algebraic complement of $\mathcal{B}(\ederiv; \Lambda^{k-1})$ in $\Lambda^k(\manifold{M})$. One can write any $k$-form $\kdifform{a}{k}$ as:
\[
\kdifform{a}{k} = \ederiv\kdifform{b}{k-1} + \kdifform{c}{k}, \quad \kdifform{a}{k}\in\Lambda^{k}(\manifold{M}),\quad\kdifform{b}{k-1}\in\Lambda^{k-1}(\manifold{M})\quad\mathrm{and}\quad \kdifform{c}{k}\in\mathcal{B}^c(\ederiv;\Lambda^k).
\]
		
	\subsection{Hodge-$\star$ operator}
		An important concept for the definition of the Hodge-$\star$ operator is the the volume form.
		\begin{definition}[\textbf{Standard volume form}] \label{def::hodge_star}
			\cite{abraham_diff_geom} In an $n$-dimensional manifold $\manifold{M}$ with metric $g_{ij}$, a volume form $\kdifform{w}{n}$ is defined in a local coordinate system $x^{i}$ as:
			\[
				\kdifform{w}{n} := \sqrt{|\det(g_{ij})|}\,\ederiv x^{1}\wedge\cdots\wedge\ederiv x^{n} \;,
			\]
			where $\kdifform{w}{n}$ is the {\em standard volume form}.
		\end{definition}
		
		The Hodge-$\star$ operator is then defined in the following way.
		\begin{definition}[\textbf{Hodge-$\star$ operator}]\label{def::diffGeom_hodge}
			\cite{abraham_diff_geom,flanders::diff_forms} The Hodge-$\star$ operator in an $n$-dimensional manifold $\manifold{M}$ is an operator, $\star : \Lambda^{k}(\manifold{M})\spacemap\Lambda^{n-k}(\manifold{M})$, defined by
			\[
				\kdifform{a}{k}\wedge\star\kdifform{b}{k} := \inner{\kdifform{a}{k}}{\kdifform{b}{k}}\kdifform{w}{n} \;,
			\]
			where the inner product on the right is defined in Definition~\ref{def:inner_product_diff_forms}. Application of the Hodge-$\star$ to the unit 0-form yields	$\star 1:= \kdifform{w}{n}$.
		\end{definition}
		
		For $f, g\in C^{\infty}(\manifold{M})$ and $\kdifform{a}{k},\kdifform{b}{k}\in\Lambda^{k}(\mathcal{M})$ the Hodge-$\star$ operator satisfies, \cite{morita2001geometry} 
		\begin{subequations}
		\begin{align}
			&\star (f\kdifform{a}{k} + g \kdifform{b}{k}) = f \star \kdifform{a}{k} + g \star \kdifform{b}{k} \;,\\
			&\star\star\kdifform{a}{k} = (-1)^{k(n-k)}\kdifform{a}{k} \label{eq::diffGeom_double_hodge} \;,\\
			&\kdifform{a}{k}\wedge\star\kdifform{b}{k}=\kdifform{b}{k}\wedge\star\kdifform{a}{k}=\inner{\kdifform{a}{k}}{\kdifform{b}{k}}w^{n} \;,\\
			&\star(\kdifform{a}{k}\wedge\star\kdifform{b}{k}) = \star(\kdifform{b}{k}\wedge\star\kdifform{a}{k}) = \inner{\kdifform{a}{k}}{\kdifform{b}{k}} \;, \\
			&\inner{\star\kdifform{a}{k}}{\star \kdifform{b}{k}} = \inner{\kdifform{a}{k}}{\kdifform{b}{k}} \;.
		\end{align}
		\end{subequations}
		
		On $\mathbb{R}^2$ we have the following relations for the basis forms
		\[
			\star 1 = \ederiv x\!\wedge\!\ederiv y, \quad \star\ederiv x = \ederiv y, \quad \star\ederiv y = -\ederiv x, \quad \star\ederiv x \!\wedge\! \ederiv y = 1 \;.
		\]
		
		It is possible to define an integral inner product of $k$-forms in the following way:
		\begin{definition}
			\cite{abraham_diff_geom} The space of $k$-forms, $\Lambda^{k}(\manifold{M})$ can be equipped with an $L^{2}$ inner product, $\inner{\cdot}{\cdot}_{L^2\Lambda^k(\manifold{M})}: \Lambda^{k}(\manifold{M})\times\Lambda^{k}(\manifold{M})\spacemap \mathbb{R}$, given by:
			\begin{equation}
				\inner{\kdifform{a}{k}}{\kdifform{b}{k}}_{L^2\Lambda^k(\manifold{M})} := \int_{\manifold{M}} \kdifform{a}{k}\wedge\star\kdifform{b}{k}  \;.\label{eq::diffGeom_L2_inner}
			\end{equation}
		\end{definition}

		
    Suppose we have a map $\mapping: \manifold{M}\spacemap\manifold{N}$. How does the Hodge-$\star$ transform under this mapping? Let us call the transformed Hodge-$\star$, $\hat{\star}$.
		
		\begin{proposition}[\textbf{Transformation of the Hodge-$\star$ operator}]
			If $\kdifform{a}{k}\in\Lambda^{k}(\manifold{N})$, $\pullback\kdifform{a}{k}\in\Lambda^{k}(\manifold{M})$ and $\mapping: \manifold{M}\spacemap\manifold{N}$, then the Hodge operator in $\manifold{M}$, denoted $\hat{\star}:\Lambda^{k}(\manifold{M})\spacemap\Lambda^{n-k}(\manifold{M})$, is given by:
			\[
				\hat{\star} = \pullback\star(\pullback)^{-1} \;.
			\]
		\end{proposition}
		\begin{proof} 
			The objective is to find a $\hat{\star}$ such that the following diagram commutes, \cite{bouman::icosahom2009}:
			\[\begin{CD}
				\Lambda^{k}(\manifold{N}) @>\star>> \Lambda^{n-k}(\manifold{N}) \\
				@VV\pullback V  @VV\pullback V \\
				\Lambda^{k}(\manifold{M}) @>\hat{\star}>> \Lambda^{n-k}(\manifold{M}).
			\end{CD}\]
			Hence $\hat{\star}$ should be such that $\pullback\star = \hat{\star}\pullback$. Therefore $\hat{\star} = \pullback\star(\pullback)^{-1}$.
		\end{proof}
		

		The Hodge-$\star$ operator enables one to extend the single de Rham complex \eqref{singlecomplex} to a double de Rham complex connected by the Hodge-$\star$ operator
		\begin{equation}
			\begin{matrix}
				\mathbb{R} \longrightarrow&\Lambda^{0}(\manifold{M})
				&\stackrel{\ederiv}{\longrightarrow}& \Lambda^{1}(\manifold{M})
				&\stackrel{\ederiv}{\longrightarrow}& \hdots \;
				&\stackrel{\ederiv}{\longrightarrow}\; &\Lambda^{n}(\manifold{M}) \;
				&\stackrel{\ederiv}{\longrightarrow}\; &0  \\
				&\star\updownarrow & & \star\updownarrow &&  &
				&\star\updownarrow & &   \\
				0 \stackrel{\ederiv}{\longleftarrow}&{\Lambda}^{n}(\manifold{M})
				&\stackrel{\ederiv}{\longleftarrow}& {\Lambda}^{n-1}(\manifold{M})
				&\stackrel{\ederiv}{\longleftarrow}& \hdots \;
				&\stackrel{\ederiv}{\longleftarrow}\; &{\Lambda}^{0}(\manifold{M}) \;
				&\stackrel{}{\longleftarrow}\; &\mathbb{R}.
			\end{matrix} \label{eq::diffGeom_deRham_Complex}
		\end{equation}

		
		Note the similarity between this diagram and \figref{fig:diffGeom_orientation}.
		
		\begin{definition}[\textbf{The codifferential operator}]
			\cite{frankel} The codifferential operator, $\coderiv: \Lambda^{k}(\manifold{M})\spacemap\Lambda^{k-1}(\manifold{M}) $, is an operator that is  formal Hilbert adjoint of the exterior derivative, with respect to the inner product \eqref{eq::diffGeom_L2_inner}:
			\begin{equation}
				\inner{\ederiv\kdifform{a}{k-1}}{\kdifform{b}{k}}_{L^2 \Lambda^k(\manifold{M})} = \inner{\kdifform{a}{k-1}}{\coderiv\kdifform{b}{k}}_{L^2 \Lambda^k(\manifold{M})} \;. \label{eq::diffGeom_codiff_adjoint}
			\end{equation}
		\end{definition}
		\begin{proposition}\label{prop::diffgeom_codifferential_adjoint}
			If $\kdifform{a}{k-1}\in\Lambda^{k-1}(\mathcal{M})$, $\kdifform{b}{k}\in\Lambda^{k}(\mathcal{M})$ and
			\begin{equation}
				\int_{\manifold{M}}\ederiv(\kdifform{a}{k-1}\wedge\star\kdifform{b}{k}) = \int_{\partial\manifold{M}}\kdifform{a}{k-1}\wedge\star\kdifform{b}{k} = 0 \;, \label{eq::difGeom::boundary_zero_codifferential}
			\end{equation}
			then
			\begin{equation}
				\coderiv\kdifform{b}{k} := (-1)^{n(k+1)+1}\star\ederiv\star\kdifform{b}{k},\quad\forall\kdifform{b}{k}\in\Lambda^{k}(\manifold{M}) \;. \label{eq::diffGeom_codifferential_definition}
			\end{equation}
		\end{proposition}
		\begin{proof}
			The proof follows directly from the integrals:
			\[
				\inner{\ederiv\kdifform{a}{k-1}}{\kdifform{b}{k}}_{L^2 \Lambda^k(\manifold{M})} \stackrel{(\ref{eq::diffGeom_L2_inner})}{=} \int_{\manifold{M}}\ederiv\kdifform{a}{k-1}\wedge\star\kdifform{b}{k} \stackrel{(\ref{eq::dif_and_wedge})}{=} \int_{\manifold{M}}\ederiv(\kdifform{a}{k-1}\wedge\star\kdifform{b}{k}) - (-1)^{k-1}\int_{\manifold{M}}\kdifform{a}{k-1}\wedge\ederiv\star\kdifform{b}{k} \;.
			\]
			The first term on the right is zero, \eqref{eq::difGeom::boundary_zero_codifferential}, and the second one is simply:
			\[
				- (-1)^{k-1}\int_{\manifold{M}}\kdifform{a}{k-1}\wedge\ederiv\star\kdifform{b}{k} \stackrel{(\ref{eq::diffGeom_double_hodge})}{=} \int_{\manifold{M}}\kdifform{a}{k-1}\wedge\star[(-1)^{n(k+1)+1}\star\ederiv\star]\kdifform{b}{k} \stackrel{(\ref{eq::diffGeom_L2_inner})}{=} \inner{\kdifform{a}{k-1}}{\coderiv\kdifform{b}{k}}_{L^2 \Lambda^k(\manifold{M})} \;.
			\]
			
		\end{proof}
		
\begin{proposition}[\textbf{Nilpotency of codifferential}] \label{prop:nilpotency_codifferential}
The codifferential satisfies
\[
\coderiv\coderiv \kdifform{a}{k} = 0, \quad \forall\kdifform{a}{k}\in\Lambda^{k}(\manifold{M}) \;.
\]
\end{proposition}
\begin{proof}
It follows directly from \eqref{eq::diffGeom_codifferential_definition}, \eqref{eq::diffGeom_double_hodge} and \eqref{eq:double_dif}.
\end{proof}
Due to the \propref{prop:nilpotency_codifferential}, the $(n+1)$-spaces of differential forms satisfy the following sequence:
\begin{equation}
\begin{array}{cccccccccccc}
0\; \longleftarrow&\Lambda^0(\manifold{M})
&\stackrel{\coderiv}{\longleftarrow}& \Lambda^1(\manifold{M})
&\stackrel{\coderiv}{\longleftarrow} & \cdots \; &\stackrel{\coderiv}{\longleftarrow}\; &\Lambda^n(\manifold{M}) \;
&\stackrel{\coderiv}{\longleftarrow}\; &\mathbb{R} \;.
\end{array}
\label{diffsinglecomplex}
\end{equation}

\begin{definition}
A differential form $\kdifform{a}{k}$ is called {\em co-exact} if there exists a differential form $\kdifform{b}{k+1}$ such that $\kdifform{a}{k} = \coderiv \kdifform{b}{k+1}$. The space of co-exact $k$-forms is given by 
\[ \mathcal{B}(\coderiv;\Lambda^{k+1}):=\coderiv\Lambda^{k+1}(\mathcal{M})\subset\Lambda^k(\mathcal{M}).\]
A differential form $\kdifform{a}{k}$ is called {\em coclosed} if $\coderiv \kdifform{a}{k} = 0$ The space of coclosed $k$-forms is given by
\[ \mathcal{Z}(\coderiv;\Lambda^k):=\{\forall \kdifform{a}{k}\in\Lambda^k(\mathcal{M})\;|\;\coderiv \kdifform{a}{k}=0\}\subset\Lambda^k(\mathcal{M})\;.\] 
\end{definition}
\propref{prop:nilpotency_codifferential} implies that $\mathcal{B}(\coderiv;\Lambda^{k})\subset \mathcal{Z}(\coderiv;\Lambda^{k-1})$. Only when $\mathcal{B}(\coderiv;\Lambda^{k})= \mathcal{Z}(\coderiv;\Lambda^{k-1})$ is the sequence in (\ref{diffsinglecomplex}) exact.
\begin{figure}[htp]
\centering
\includegraphics[width=1.\textwidth]{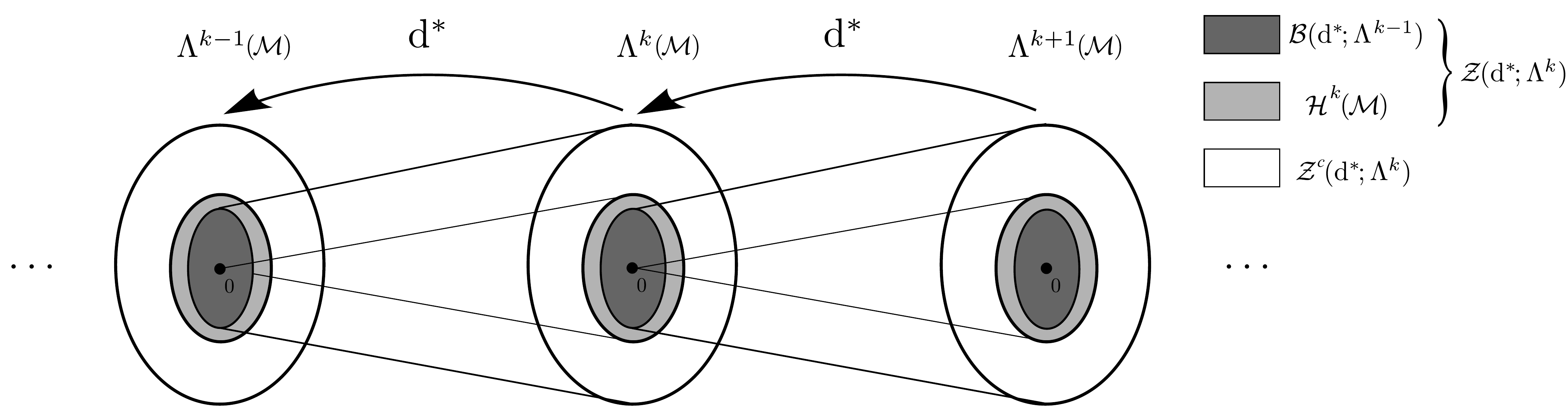}
\caption{Pictorial view of the complex of differential forms acted upon by the codifferential operator. The codifferential $\coderiv$ maps the elements of $\Lambda^{k}$ into the kernel of $\coderiv$ of $\Lambda^{k-1}$, $\mathcal{B}(\coderiv;\Lambda^{k})\subset \mathcal{Z}(\coderiv;\Lambda^{k-1})$.}
\label{fig::diffGeometry_deRhamComplex_codifferential}
\end{figure}	
One can decompose the space of $k$-forms, $\Lambda^{k}(\manifold{M})$, as 
\[
\Lambda^{k}(\manifold{M}) = \mathcal{B}(\coderiv; \Lambda^{k+1}) \oplus \mathcal{B}^{c}(\coderiv; \Lambda^{k+1}),
\]
where $\mathcal{B}^{c}(\coderiv; \Lambda^{k+1})$ is the algebraic complement of $\mathcal{B}(\coderiv; \Lambda^{k+1})$ in $\Lambda^k(\manifold{M})$. And one can write any $k$-form $\kdifform{a}{k}$ as:
\[
\kdifform{a}{k} = \coderiv\kdifform{b}{k+1} + \kdifform{c}{k}, \quad \kdifform{a}{k}\in\Lambda^{k}(\manifold{M}),\quad\kdifform{b}{k+1}\in\Lambda^{k+1}(\manifold{M})\quad\mathrm{and}\quad\kdifform{c}{k}\in\mathcal{B}^c(\coderiv;\Lambda^{k+1}).
\]
\begin{definition}[\textbf{Laplace-DeRham operator}]
\cite{abraham_diff_geom, burke1985applied} The Laplace-de Rham operator, $\Delta: \Lambda^{k}(\manifold{M})\spacemap\Lambda^{k}(\manifold{M})$, is a map given by
\begin{equation}
\Delta\kdifform{a}{k} := (\coderiv\ederiv + \ederiv\coderiv)\kdifform{a}{k},\quad\forall\kdifform{a}{k}\in\Lambda^{k}(\manifold{M}) \;. \label{eq::diffGeom_LaplacedeRham_definition}
\end{equation}
\end{definition}

\begin{proposition}[\textbf{Self-adjointness of Laplace-de Rham operator}]
Under the conditions of \eqref{eq::difGeom::boundary_zero_codifferential} the Laplace-de Rham operator satisfies:
\[
\inner{\Delta\kdifform{a}{k}}{\kdifform{b}{k}}_{L^2\Lambda^k(\manifold{M})} = \inner{\kdifform{a}{k}}{\Delta\kdifform{b}{k}}_{L^2\Lambda^k(\manifold{M})} \;,
\]
i.e. $\Delta^{*} = \Delta$. The Laplace-de Rham operator is self adjoint.
\end{proposition}
\begin{proof}
The proof follows directly from the definition of the Laplace-de Rham operator \eqref{eq::diffGeom_LaplacedeRham_definition} and \eqref{eq::diffGeom_codiff_adjoint}.
\end{proof}

\subsection{Hodge decomposition} \label{subsec:HodgeDecomposition}
A $k$-form is called closed if $\ederiv\kdifform{a}{k}=0$ and exact if there exists a $(k-1)$-form $\kdifform{b}{k-1}$ such that $\kdifform{a}{k}=\ederiv\kdifform{b}{k-1}$. Since $\ederiv\circ\ederiv=0$, \eqref{eq:double_dif}, every  exact form is closed. For manifolds which are non-contractible the Poincar\'{e} \lemmaref{lemma:poincare} states that not all closed forms are exact.

One can define the space of \emph{harmonic forms} as those differential forms which are closed, but not exact,
\begin{equation}
\mathcal{H}^k:=\mathcal{Z}(\ederiv;\Lambda^k)\cap\mathcal{B}^c(\ederiv;\Lambda^{k-1}) \;\;\; \Longrightarrow \;\;\; \mathcal{Z}(\ederiv;\Lambda^k) = \mathcal{H}^k(\mathcal{M}) \cup \mathcal{B}(\ederiv;\Lambda^{k-1})\;.
\label{harmonicspaceed}
\end{equation}
For contractible manifolds $\mathcal{H}^k=\emptyset$. Similar to \eqref{hdranges} we can decompose $\Lambda^k(\mathcal{M})$ into its nullspace and its complement,
\[
\Lambda^k(\mathcal{M})=\mathcal{Z}(\ederiv;\Lambda^k)\oplus\mathcal{Z}^c(\ederiv;\Lambda^k).
\]
If we combine this decomposition with  \eqref{harmonicspaceed}, we obtain the following \emph{Hodge decomposition},
\begin{equation}
\Lambda^k(\mathcal{M})=\mathcal{B}(\ederiv;\Lambda^{k-1})\oplus\mathcal{H}^k\oplus\mathcal{Z}^c(\ederiv;\Lambda^k).
\label{eq:generic_Hodge_decomposition}
\end{equation}

\begin{proposition}\label{prop:adjoint_spaces_null_range}
If $\tr\kdifform{a}{k}=0$ or $\partial \manifold{M} = \emptyset$, then
\begin{equation} \mathcal{Z}(\coderiv;\Lambda^k)=\mathcal{B}^c(\ederiv;\Lambda^{k-1}) \;\;\; \mbox{and}\;\;\;
\mathcal{Z}^c(\ederiv;\Lambda^k)=\mathcal{B}(\coderiv;\Lambda^{k+1}) \;.
\label{eq:adjoint_relation_null_range}
\end{equation}
\end{proposition}
\begin{proof}
For all $\kdifform{a}{k} \in \mathcal{Z}(\coderiv;\Lambda^k)$ and all $\kdifform{b}{k-1} \in \Lambda^k(\manifold{M})$ we have
\[ 0 = \left ( \coderiv \kdifform{a}{k}, \kdifform{b}{k-1} \right )_{L^2\Lambda^{k-1}(\manifold{M})} = \left (  \kdifform{a}{k}, \ederiv \kdifform{b}{k-1} \right )_{L^2\Lambda^k(\manifold{M})} \;.\]
Therefore, $\kdifform{a}{k} \perp \mathcal{B}(\ederiv;\Lambda^{k-1})$, i.e. $\kdifform{a}{k} \in \mathcal{B}^\perp(\ederiv;\Lambda^{k-1}) = \mathcal{B}^c(\ederiv;\Lambda^{k-1})$. Therefore, $\mathcal{Z}(\coderiv;\Lambda^k)\subset \mathcal{B}^c(\ederiv;\Lambda^{k-1})$. Conversely, if $\kdifform{a}{k} \in \mathcal{B}^\perp(\ederiv;\Lambda^{k-1})$, then $\left ( \coderiv \kdifform{a}{k}, \kdifform{b}{k-1} \right )_{L^2\Lambda^k(\manifold{M})}=0$ for all $\kdifform{b}{k-1}$ which implies that $\coderiv \kdifform{a}{k} = 0$, therefore $\mathcal{B}^c(\ederiv;\Lambda^{k-1}) \subset \mathcal{Z}(\coderiv;\Lambda^k)$. So we have $\mathcal{B}^c(\ederiv;\Lambda^{k-1}) = \mathcal{Z}(\coderiv;\Lambda^k)$.

For all $\kdifform{a}{k} \in \mathcal{Z}(\ederiv;\Lambda^k)$ and all $\kdifform{b}{k} \in \mathcal{B}(\coderiv;\Lambda^{k+1})$, i.e. there exists a $\kdifform{c}{k+1}$ such that $\kdifform{b}{k} = \coderiv \kdifform{c}{k+1}$. Then
\[ 0 \stackrel{\ederiv \kdifform{a}{k}=0}{=} \left ( \ederiv \kdifform{a}{k},\kdifform{c}{k+1} \right )_{L^2\Lambda^{k+1}(\manifold{M})} = \left ( \kdifform{a}{k},\coderiv \kdifform{c}{k+1} \right )_{L^2\Lambda^{k}(\manifold{M})} = \left ( \kdifform{a}{k},\kdifform{b}{k} \right )_{L^2\Lambda^{k}(\manifold{M})}\;.\]
This implies that $\mathcal{Z}^c(\ederiv;\Lambda^k)=\mathcal{B}(\coderiv;\Lambda^{k+1})$.
\end{proof}

Using Proposition~\ref{prop:adjoint_spaces_null_range} in (\ref{harmonicspaceed}) yields
\[
\mathcal{H}^k=\mathcal{Z}(\ederiv;\Lambda^k)\cap\mathcal{Z}(\coderiv;\Lambda^k),
\]
or
\begin{equation}
\mathcal{H}^k=\{\;\forall a\in\Lambda^k(\mathcal{M})\;|\;\ederiv a=0,\;\coderiv a=0\;\}.
\end{equation}
Harmonic forms are therefore both closed and coclosed. Using Proposition~\ref{prop:adjoint_spaces_null_range} and (\ref{eq:generic_Hodge_decomposition}) allows us to write the Hodge decomposition as
\begin{equation}
\Lambda^k(\mathcal{M})=\mathcal{B}(\ederiv;\Lambda^{k-1})\oplus\mathcal{H}^k\oplus\mathcal{B}(\coderiv;\Lambda^{k+1}).
\label{eq:Hodge_decomp_without_boundary}
\end{equation}

\begin{remark}
Note that the Hodge decomposition (\ref{eq:generic_Hodge_decomposition}) is true, whether $\manifold{M}$ has a boundary or not, whereas (\ref{eq:Hodge_decomp_without_boundary}) is only true if $\partial \manifold{M} = \emptyset$. The Hodge decomposition in the form (\ref{eq:generic_Hodge_decomposition}) will play an important role in the remainder of this paper.
\end{remark}

\begin{corollary}[\textbf{Hodge decomposition}]\cite{abraham_diff_geom}\label{cor:hodgedecomposition}
Let $\manifold{M}$ be a compact boundaryless oriented Riemannian manifold. Every $\kdifform{e}{k}\in\Lambda^{k}(\manifold{M})$ can be written in terms of $\kdifform{a}{k-1}\in\Lambda^{k-1}(\manifold{M})$, $\kdifform{b}{k+1}\in\Lambda^{k+1}(\manifold{M})$ and $\kdifform{c}{k}\in \mathcal{H}^k$ such that
\begin{equation}
\kdifform{e}{k} = \ederiv \kdifform{a}{k-1} + \coderiv \kdifform{b}{k+1} + \kdifform{c}{k} \;.\label{eq::diffGeom_hodge_decomposition}
\end{equation}
\end{corollary}

\begin{remark}\label{eq::diffGeom_harmonicForms_equivalence}
Let $\manifold{M}$ be a compact boundaryless oriented Riemannian manifold and $\kdifform{a}{k}\in\mathcal{H}^{k}$. Then it follows that $\kdifform{c}{k}$ is the harmonic solution of the Laplace-De Rham operator,
\begin{equation}
(\;\coderiv\kdifform{c}{k} = 0\ \text{ and }\ \ederiv\kdifform{c}{k} = 0\;) \quad \Longleftrightarrow \quad \Delta\kdifform{c}{k} = 0.
\end{equation}
\end{remark}
\begin{remark}
Although $\ederiv \kdifform{a}{k-1}$, $\coderiv \kdifform{b}{k+1}$ and $\kdifform{e}{k}$ are unique, the differential forms $\kdifform{a}{k-1}$ and $\kdifform{b}{k+1}$ are generally not unique, because we can replace $\kdifform{a}{k-1}$ by $\kdifform{a}{k-1}+ \ederiv \kdifform{p}{k-2}$ for any $\kdifform{p}{k-2}$ and $\kdifform{b}{k+1}$ by $\kdifform{b}{k+1}+\coderiv \kdifform{q}{k+2}$ and this will also satisfy the Hodge decomposition.
\end{remark}

\begin{remark}
The $k$-th de Rham cohomology group of $\manifold{M}$, $H^{k}$, is defined as:
\[
H^{k}:=\frac{\mathcal{Z}(\ederiv;\Lambda^k)}{\mathcal{B}(\ederiv;\Lambda^{k-1})}.
\]
It is possible to prove that the space of harmonic $k$-forms, $\mathcal{H}^{k}$, is isomorphic to the cohomology group $H^{k}$, see \cite{abraham_diff_geom}. Moreover, the de Rham theorem, see \cite{thorpe1976lecture}, states that for a finite dimensional compact manifold this group is isomorphic to the homology group defined in algebraic topology, the isomorphism being given by integration. This connection between differential geometry and algebraic topology plays an essential role in the development of the numerical scheme presented here since it enables the representation of differential geometric structures as finite dimensional algebraic topological structures suitable for using in the discretization process, as it will be seen in the following sections.
\end{remark}

For a manifold with boundary, the Hodge decomposition theory takes a somewhat different form. A $k$-form $\kdifform{a}{k}\in\Lambda^{k}(\manifold{M})$ is called parallel or tangent to a manifold $\manifold{Q}\subset\manifold{M}$ if $\mathrm{tr}_{\manifold{M},\manifold{Q}}(\star\kdifform{a}{k})=0$ and is called perpendicular or normal to a manifold $\manifold{Q}\subset\manifold{M}$ if $\mathrm{tr}_{\manifold{M},\manifold{Q}}(\kdifform{a}{k})=0$. In this way the following spaces can be introduced
\begin{align}
\Lambda^{k}_{t}(\manifold{M}) &= \{\kdifform{a}{k}\in\Lambda^{k}(\manifold{M}) \, | \, \kdifform{a}{k} \text{ is tangent to } \partial\manifold{M}\} \;,\\
\Lambda^{k}_{n}(\manifold{M}) &= \{\kdifform{a}{k}\in\Lambda^{k}(\manifold{M}) \, | \, \kdifform{a}{k} \text{ is perpendicular to } \partial\manifold{M}\} \;.
\end{align}
Note that for the space of harmonic forms $\ederiv\kdifform{a}{k} = 0$ and $\coderiv\kdifform{a}{k} = 0$ is a stronger condition than $\Delta\kdifform{a}{k} = 0$ when $\manifold{M}$ has a boundary.
\begin{proposition}[\textbf{Hodge decomposition theorem for manifolds with boundary}]\cite{abraham_diff_geom}
Let $\manifold{M}$ be a compact oriented manifold with boundary. The following decomposition holds
\begin{equation}
\Lambda^{k}(\manifold{M}) = \ederiv\Lambda^{k-1}_{t}(\manifold{M}) \oplus \mathcal{H}^{k}(\manifold{M}) \oplus \coderiv\Lambda^{k+1}_{n}(\manifold{M}). \label{eq::diffGeom_Hodge_decomposition_boundary}
\end{equation}
\end{proposition}

\subsection{Hilbert spaces}
On an oriented Riemannian manifold, we can define Hilbert spaces for differential forms. Let all $f_i(x_{i_1},\hdots,x_{i_n})$ be functions in $L^2(\mathcal{M})$, then $a^{(k)}$ be a $k$-form in the Hilbert space $L^2\Lambda^k(\Omega)$ is given by
\[
a^{(k)}=\sum_if_i(x_{i_1},\hdots,x_{i_n})\;\ederiv x_{i_1}\wedge\ederiv x_{i_2}\wedge\hdots\wedge\ederiv x_{i_k}.
\]
The norm corresponding to the space $L^2\Lambda^k(\mathcal{M})$ is $\Vert a^{(k)}\Vert_{L^2\Lambda^k}^2=\left(a^{(k)},a^{(k)}\right)_{L^2\Lambda^k(\mathcal{M})}$. Although extension to higher Sobolev spaces are possible, we focus here to the Hilbert space corresponding to the exterior derivative. The Hilbert space $H\Lambda^k(\mathcal{M})$ is defined by
\[
H\Lambda^k(\mathcal{M})=\{a^{(k)}\in L^2\Lambda^k(\mathcal{M})\;|\;\ederiv a^{(k)}\in L^2\Lambda^{k+1}(\mathcal{M})\}.
\]
and the norm corresponding to $H\Lambda^k(\mathcal{M})$ is defined by
\[
\Vert a^{(k)}\Vert^2_{H\Lambda^k}:=\Vert a^{(k)}\Vert^2_{L^2\Lambda^k}+\Vert\ederiv a^{(k)}\Vert^2_{L^2\Lambda^{k+1}}.
\]
The $L^2$ de Rham complex, or the Hilbert version of the de Rham complex, is the sequence of maps and spaces given by
\[
\mathbb{R}\hookrightarrow H\Lambda^0(\mathcal{M})\stackrel{\ederiv}{\longrightarrow} H\Lambda^1(\mathcal{M})\stackrel{\ederiv}{\longrightarrow}\cdots\stackrel{\ederiv}{\longrightarrow} H\Lambda^n(\mathcal{M})\stackrel{\ederiv}{\longrightarrow}0.
\]
An important inequality in stability analysis, relating both norms, is Poincar\'e inequality.
\begin{lemma}[\textbf{Poincar\'e inequality}]\cite{arnold2010finite}\label{poincareinequality}
Consider the de Rham complex $(\Lambda,\ederiv)$, then the exterior derivative is a bounded bijection from $\mathcal{Z}^{c}(\ederiv,L^2\Lambda^k)$ to $\mathcal{B}(\ederiv,L^2\Lambda^{k-1})$, and hence, by Banach's bounded inverse theorem, there exists a constant $c_P$ such that
\begin{equation}
\Vert a\Vert_{H\Lambda^k}\leq c_P\Vert\ederiv a\Vert_{L^2\Lambda^k},\quad\forall a\in \mathcal{Z}^{c}(\ederiv,L^2\Lambda^k),
\end{equation}
Which we refer to as the \emph{Poincar\'e inequality}. We remark that the condition $\mathcal{B}(\ederiv,L^2\Lambda^{k-1})$ is closed is not only sufficient, but also necessary to obtain this result.
\end{lemma}

\section{An introduction to Algebraic Topology}\label{sec:AlgebraicTopology}
An important step in the our numerical framework is the discretization of a manifold (physical space) in which the physical laws are embedded. By this we mean the partitioning of a manifold in a collection of non-overlapping subspaces (cells) such that their union is the whole manifold under study. This partitioning into a set of distinct subspaces yields a representation of a manifold in terms of a finite number of points (0-cells), lines (1-cells), surfaces (2-cells), volumes (3-cells) and their analogues of dimension $k$ ($k$-cells). A collection of $k$-cells of such a partition is called a $k$-chain. Several types of geometrical objects ($k$-cells) can be used for the subdivision. In this work we will focus on quadrilaterals and their generalizations to higher dimensions, (singular $k$-cubes).

In this section we intend to formally introduce the representation of these collections of $k$-cells and to associate them with a suitable algebraic structure that enables a correct discrete representation of the original manifold. The branch of mathematics that provides such a formal discrete representation of a compact manifold is algebraic topology. 

Having defined a discrete representation of the manifold in terms of $k$-chains, we can assign values	to the various space elements by defining the dual space of the chains space; the so-called {\em $k$-cochains}. Duality pairing between chains and cochains allows us to introduce operations on cochains as formal adjoint of operations on chains. In this way, the introduction of operations on $k$-cochains mimics the operations on $k$-forms discussed in the previous section.

The main reason this introduction on algebraic topology is given here, is that we will explicitly employ the close relationship between differential geometry in a continuous setting and algebraic topology in the discrete setting in the following sections. This relationship is also employed by many others, for instance, \cite{bochev2006principles,Dodziuk76,mattiussi2000finite,tonti1975formal}. For a more thorough treatment of algebraic topology the reader is referred to \cite{hatcher2002algebraic,Massey1,Massey2,munkres1984elements,thorpe1976lecture}.
	

In this paper we assume that the dimension of all manifolds is finite and dim$(\manifold{M})=n$.

\subsection{Cell complexes and the boundary operator}

	\begin{definition}[\textbf{$k$-cell}]\label{def:k-cell}
		\cite{hatcher2002algebraic,munkres1984elements} A $k$-cell, $\tau_{(k)}$, of a manifold $\manifold{M}$ of dimension $n\geq k$ is a set of points of $\manifold{M}$ that is homeomorphic to a closed $k$-ball $B_{k}=\{\mathbf{x}\in\mathbb{R}^{k}:\|\mathbf{x}\|\leq 1\}$.
		
		The boundary of a $k$-cell $\tau_{(k)}$, $\partial\tau_{(k)}$, is the subset of $\manifold{M}$ associated by the above mentioned homeomorphism to the boundary $\partial B_{k}=\{\mathbf{x}\in\mathbb{R}^{k}:\|\mathbf{x}\|=1\}$ of $B_{k}$.
	\end{definition}
	
	The topological description of a manifold can be done in terms of {\em simplices}, see for instance \cite{munkres1984elements,thorpe1976lecture,Whitney57} or in terms of {\em singular $k$-cubes}, see \cite{Massey1,Massey2,bookTonti}. From a topological point of view both descriptions are equivalent, see \cite{Dieudonne}. Despite this equivalence of simplicial complexes and cubical complexes, the reconstruction maps to be discussed in Section~\ref{mimeticoperators} differ significantly. For mimetic methods based on simplices see \cite{arnold2006finite,arnold2010finite,desbrun2005discrete,Hirani_phd_2003,Rapetti2009,Rapetti2007}, whereas for mimetic methods based on singular cubes see \cite{arnold:Quads,bochev2006principles,BeauceSen2004,HYmanSteinberg2004,HymanShashkovSteinberg2002,RobidouxSteinberg2011}.
	
	Here we list the terminology to set a homology theory in terms of $k$-cubes as given by \cite{Massey2}:
	\begin{list}{\quad}{}
	\item $\mathbb{R}$ = real line.
	\item $I$ = closed interval $[-1,1]$.
	\item $\mathbb{R}^k = \mathbb{R} \times \mathbb{R} \times \dots \times \mathbb{R}$ ($k$ factors, $k\geq 0$) Euclidean $k$-space.
	\item $I^k = I \times I \times \dots \times I$ ($k$ factors, $k\geq 0$) unit $k$-cube.
	\end{list}
	By definition $I^0$ is a space consisting of a single point.
	
	\begin{definition}[\textbf{Singular $k$-cube}]\label{def:kcube}
		\cite{Massey2} A {\em singular $k$-cube} in a $n$-dimensional manifold $\manifold{M}$ is a continuous map $\tau_{(k)}\,:\,I^k\,\longrightarrow \, \manifold{M},\; 0\leq k \leq n$.	
	\end{definition}
	
	\begin{definition}[\textbf{Degenerate singular $k$-cube}]
		\cite{Massey2} If for the map $\tau_{(k)}\,:\,I^k\,\longrightarrow \, \manifold{M},\; 0\leq k\leq n$ there exists an $i$, $1\leq i \leq k$ such that	$\tau_{(k)}(x_1,x_2,\dots,x_k)$ does not depend on $x_i$, then the singular $k$-cube is {\em degenerate}.
	\end{definition}
	
	\begin{corollary}
	Any non-degenerate $k$-cube is a $k$-cell as defined by Definition~\ref{def:k-cell}.
	\end{corollary}
	
	\begin{remark}
	A non-singular $k$-cube is a submanifold of $\manifold{M}$ according to Definition~\ref{def:submanifold}.
	\end{remark}
	
	\begin{remark}\label{rem:inner_orientation_k_cell}
	Select an orientation in $\mathbb{R}^k$, then the determinant of the Jacobian of the maps $\tau_k$ determines the (inner) orientation of the $k$-cell according to Definition~\ref{def:orientation}.
	\end{remark}
	
	\begin{figure}[!ht]
			\centering
				\includegraphics[width=0.6\textwidth]{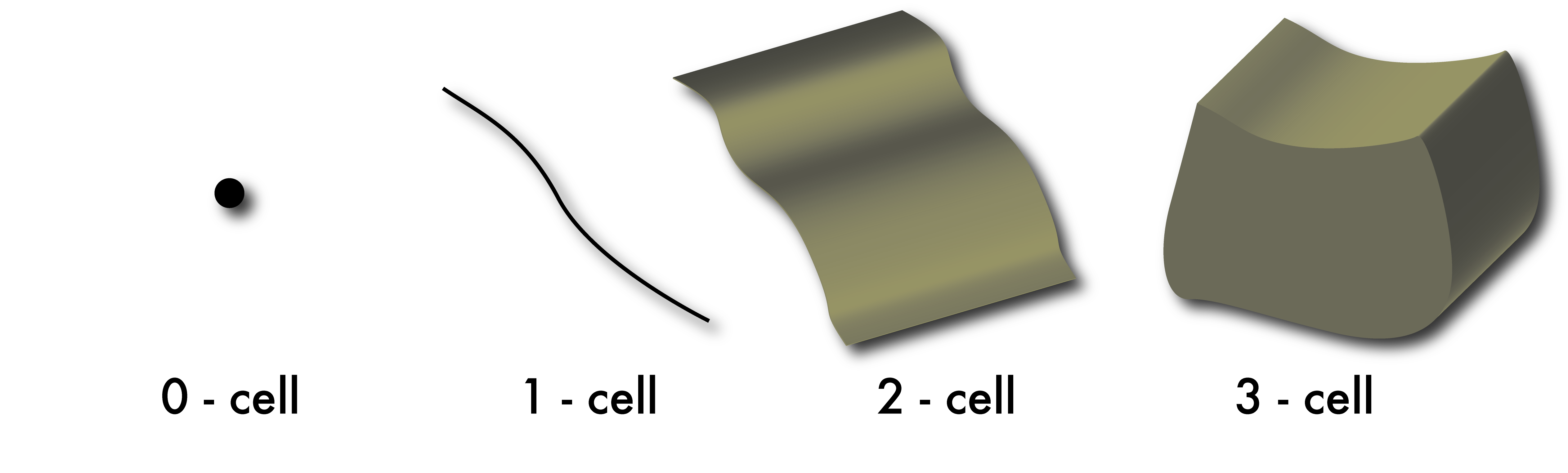}
				\caption{Example of a 0-cell, a 1-cell, a 2-cell and a 3-cell.}
				\label{fig:algTopology_cell_examples}
		\end{figure}	
		\figref{fig:algTopology_cell_examples} depicts some examples of $k$-cells in a manifold $\manifold{M}=\mathbb{R}^{3}$. The $k$-cells are geometric objects which represent the geometric objects shown in Figure~\ref{fig:diffGeom_orientation}.
	Before we formally define the boundary of $k$-cubes, we first define faces of a $k$-cube.
	\begin{definition}[\textbf{The faces of a singular $k$-cube}]\label{def:Boundary_on_k_cube} \cite{Massey2,spivak1998calculus}
		For $0<k \leq n0$ let $\tau_{(k)}$ be a singular $k$-cube in $\manifold{M}$. For $i=1,2,\dots,k$, we define the singular $(k-1)$-cubes $A_i\tau_{(k-1)},\: B_i\tau_{(k-1)} \,:\, I^{k-1}\rightarrow\manifold{M}$, by the formulae (face maps)
	\[ A_i\tau_{(k-1)}(x_1,x_2,\dots,x_{k-1}) = \tau_{(k)}( x_1,\dots,x_{i-1},-1,x_i,\dots,x_{k-1}) \;,\]
	\[ B_i\tau_{(k-1)}(x_1,x_2,\dots,x_{k-1}) = \tau_{(k)}( x_1,\dots,x_{i-1},+1,x_i,\dots,x_{k-1}) \;.\]
	$A_i \tau_{(k-1)}$ is called the {\em front $i$-face} and $B_i\tau_{(k-1)}$ is called the {\em back $i$-face} of $\tau_{(k)}$.
	\end{definition}
	
	\figref{fig:algTopology_cell_faces_examples} depicts some examples of faces of $k$-cells in a manifold $\manifold{M}=\mathbb{R}^{3}$.
		\begin{figure}[!ht]
		\centering
			\includegraphics[width=0.5\textwidth]{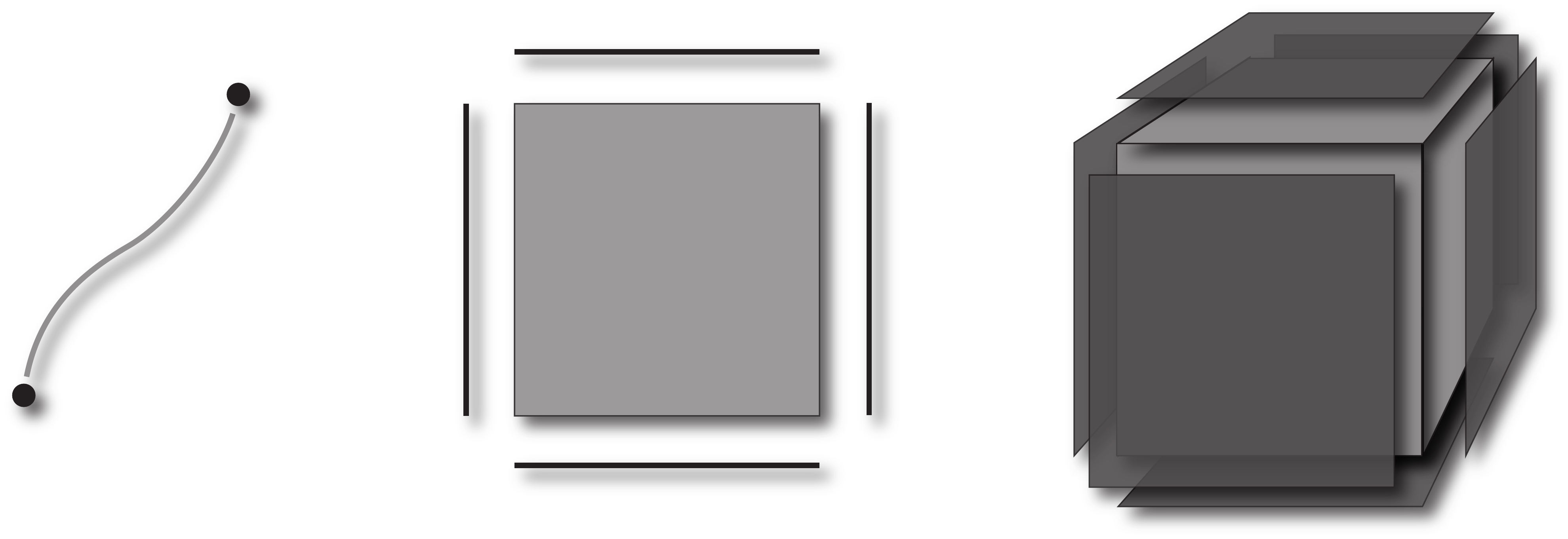}
			\caption{Examples of the faces (in dark) of a 1-cell, a 2-cell and a 3-cell in $\mathbb{R}^{3}$}
			\label{fig:algTopology_cell_faces_examples}
	\end{figure}	
	
	\begin{remark}
	Note that the face maps, $A_i\tau_{(k-1)}$ and $B_i\tau_{(k-1)}$ defined in Definition~\ref{def:Boundary_on_k_cube} are inclusion maps as defined in Definition~\ref{def:inclusion_map}.
	\end{remark}
	
	
	\begin{definition}[\textbf{The boundary of a singular $k$-cube}]\label{def:boundary_k_cell} \cite{Massey2} The boundary $\partial$ of a singular $k$-cube $\tau_{(k)}$, $k>0$ is given by
	\begin{equation}
	\partial \tau_{(k)} := \sum_{i=1}^k (-1)^i \left [ A_i\tau_{(k-1)} - B_i \tau_{(k-1)} \right ] \;.
	\end{equation}
	\end{definition}
This definition describes the boundary of the submanifold $\tau_{(k)}$, see Corollary~\ref{cor:boundary_submanifold}.


	
	
	\begin{definition}[\textbf{Cell complex}]\label{cellcomplex}
		\cite{hatcher2002algebraic,mcmullen2002abstract} A cell complex, $\ccomplex{D}$, in a compact manifold $\manifold{M}$ is a finite collection of cells such that:
		\begin{enumerate}
			\item $\ccomplex(D)$ is a covering of $\manifold{M}$.
			\item Every face of a cell of $\ccomplex{D}$ is in $\ccomplex{D}$.
			\item The intersection of any two $k$-cells, $\tau_{(k)}$ and $\sigma_{(k)}$ in $\ccomplex{D}$ is either
			  \begin{itemize}
			    \item $\tau_{(k)}$ and $\sigma_{(k)}$ share a common face;
			    \item $\tau_{(k)} \cap\sigma_{(k)}=\sigma_{k}=\tau_{(k)}$, or;
			    \item $\tau_{(k)} \cap\sigma_{(k)}= \emptyset$.
			  \end{itemize}
		\end{enumerate}
	\end{definition}
In this framework we will focus on: points (0-cells), curves (1-cells), surfaces (2-cells) and $k$-dimensional generalizations ($k$-cells), considered as non-degenerate $k$-cubes. \figref{fig:algTopology_cell_complex_examples} depicts an example of a cell complex in a compact manifold $\manifold{M} \subset \mathbb{R}^{3}$. The above presented definitions constitute a formalization of the concept of discretization of space. 
	\begin{figure}[!ht]
		\centering
			\includegraphics[width=1.\textwidth]{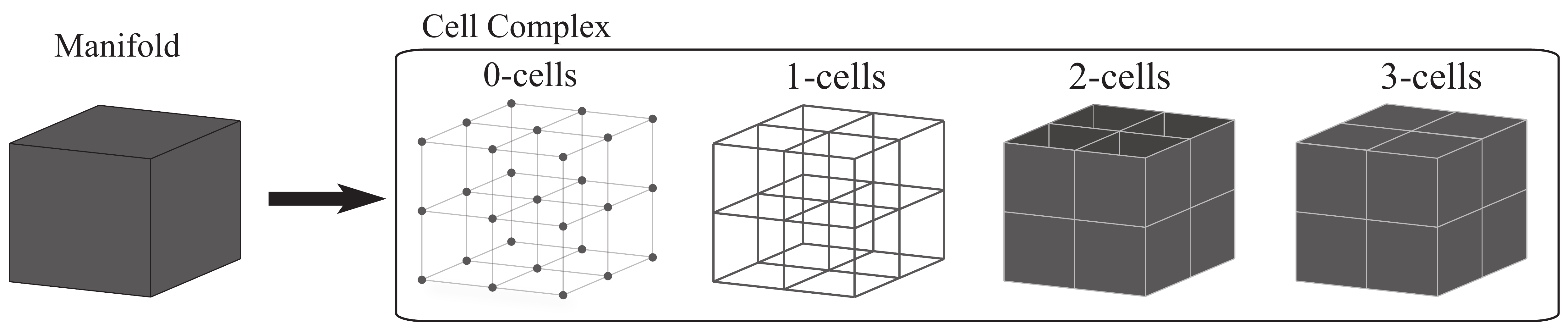}
			\caption{Example of a cell complex. Left:  a three dimensional compact manifold. Right: the $k$-cells that constitute the cell complex.}
			\label{fig:algTopology_cell_complex_examples}
	\end{figure}	
	
%

	
	
	
	\begin{definition}[\textbf{$k$-chain}]\label{def::kchain}
		\cite{hatcher2002algebraic} Given a cell complex $\ccomplex{D}$, the space of $k$-chains of $D$, $\kchainspacedomain{k}{\ccomplex{D}}$, is the free Abelian group written additivively, (see \cite{Massey2}), generated by a basis consisting of all the oriented, non-degenerate $k$-cells of $\ccomplex{D}$. A $k$-chain $\kchain{c}{k}$ in $\ccomplex{D}$ is an element of $\kchainspacedomain{k}{\ccomplex{D}}$.
	\end{definition}

	A $k$-chain, $\kchain{c}{k}\in\kchainspacedomain{k}{\ccomplex{D}}$, is a formal sum of $k$-cells, $\tau_{(k),i}\in\ccomplex{D}$:
	\[
		\forall \kchain{c}{k}\in\kchainspacedomain{k}{\ccomplex{D}} \Rightarrow  \kchain{c}{k}= \sum_{i}c^{i}\tau_{(k),i}, \quad \tau_{(k),i}\in D \;.
	\]
	Formally, given a set of $k$-cells, it is possible to generate any $k$-chain by specifying the coefficients in the chain. Although this is possible for arbitrary fields, we will, in the description of geometry, restrict ourselves mainly to chains with coefficients in $\mathbb{Z}/3=\{ -1,0,1\}$. 
	The meaning of these coefficients is : 1 if the cell is in the chain with the same orientation as its orientation in the cell complex, -1 if the cell is in the chain with the opposite orientation to the orientation in the cell complex and 0 if the cell is not part of the chain. The orientation of the cell is implied from the orientation of $\mathbb{R}^k$ and the map $\tau_k$ as pointed out in Remark~\ref{rem:inner_orientation_k_cell}. For an example of $k$-chains, see \figref{fig:algTopology_chain}.
		\begin{figure}[!ht]
			\centering
				\includegraphics[width=01.\textwidth]{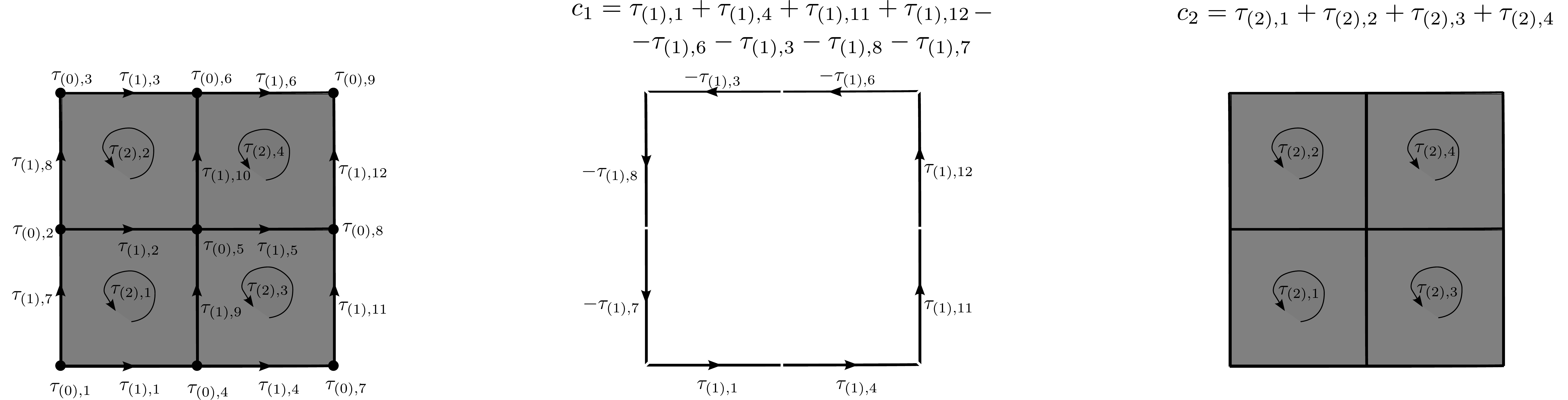}
				\caption{Example of a cell complex, a 1-chain and a 2-chain.}
				\label{fig:algTopology_chain}
		\end{figure}
		
		\begin{definition}[\textbf{Coefficient vector for chains}]\label{def:coeff_vector_for_chains}
The space of $k$-chains, $\kchainspacedomain{k}{\ccomplex{D}}$, can be represented by a column vector containing only the coefficients of the chain. That is, there is an isomorphism $\psi$:
	\begin{equation}
		\psi: \kchainspacedomain{k}{\ccomplex{D}}\mapsto\mathbb{R}^{p}, \quad p=\mathrm{rank}(\kchainspacedomain{k}{\ccomplex{D}})\;,\label{eq::algTop_isomorphism_map}
	\end{equation}
	defined by
	\begin{equation}
		\psi(\kchain{c}{k}) = \psi(\sum_{i}c^{i}\tau_{(k),i}) =  [c^{1} \cdots c^{p}]^{T}, \quad p=\text{rank}(\kchainspacedomain{k}{\ccomplex{D}})\;,\label{eq::algTop_isomorphism_definition}
	\end{equation}
	where the rank of $\kchainspacedomain{k}{\ccomplex{D}}$ is the number of $k$-cells in the cell complex $D$ and the $c^i$ are the coefficients of the $k$-chain $\mathbf{c}_{(k)}$. The $k$-chain, $\kchain{c}{k}$, is printed in boldface, whereas the vector $c_{(k)}$ of coefficients is printed in regular face.
	\end{definition}
We can now extend the boundary operator applied to a $k$-cell, Definition~\ref{def:boundary_k_cell}, to the boundary of a $k$-chain.
	\begin{definition}[\textbf{Boundary of a $k$-chain}] \label{algebraic::boundary_operator}
		\cite{hatcher2002algebraic,munkres1984elements} The boundary operator, $\partial:\kchainspacedomain{k}{D}\spacemap\kchainspacedomain{k-1}{D}$, is an homomorphism, defined by
		\begin{equation}
			\partial \kchain{c}{k} = \partial \sum_{i}c^{i}\tau_{(k),i} := \sum_{i}c^{i} \partial \left ( \tau_{(k),i} \right ) \;,
		\end{equation}
		where the action of the boundary operator on the $k$-cubes is given in Definition~\ref{def:boundary_k_cell}.
		\end{definition}
The boundary of a $k$-cell $\tau_{(k)}$ will then be a $(k-1)$-chain formed by the faces of $\tau_{(k)}$. The coefficients of this ($k-1$)-chain associated to each of the faces is given by the orientations.
		\[
			\partial \tau_{(k),j} = \sum_{i}e^{i}_{j}\tau_{(k-1),i} \;,
		\]
		with
		\[
			\left\{
				\begin{array}{l}
					e^{i}_{j} = 1, \text{ if }  \tau_{(k-1),i}\text{ has the same orientation as the face of } \tau_{(k),j} \;;\\
					e^{i}_{j} = -1, \text{ if }  \tau_{(k-1),i}\text{ has the opposite orientation as the face of } \tau_{(k),j} \;;\\
					e^{i}_{j} = 0, \text{ if }  \tau_{(k-1),i}\text{ is not a face of } \tau_{(k),j} \;.\\
				\end{array}
			\right.
		\]
		And the boundary of a $0$-cell is $\emptyset$. 
		\begin{equation}
			\partial \kchain{c}{k} = \partial(\sum_{j}c^{j}\tau_{(k),j}) = \sum_{j}c^{j}\partial\tau_{(k),j} = \sum_{ij}c^{j}e^{i}_{j}\tau_{(k-1),i} \;.\label{eq::algTop_boundary}
		\end{equation}
	\begin{example}\label{ex:incidence_matrix}
	The boundary of the $2$-cell $\tau_{(2),1}$ in \figref{fig:algTopology_chain} has the following boundary:
	\[
		\partial\tau_{(2),1} = \tau_{(1),1} + \tau_{(1),9} - \tau_{(1),2} - \tau_{(1),7}\;.
	\]
	
One clearly sees that, in this case, we have: $e^{1}_{1} = 1$, $e^{9}_{1} = 1$, $e^{2}_{1} = -1$, $e^{7}_{1} = -1$ and $e^{i}_{1} = 0$ for $i\notin \{1,2,7,9\}$.
	\end{example}
	\begin{definition}[\textbf{Incidence matrix for chains}]\label{def:incidence_matrix_chains}
	The coefficients $e^i_j$ constitute an {\em rank}$(\kchainspace{k-1}) \times${\em rank}$(\kchainspace{k})$ {\em incidence matrix} $\incidenceboundary{k-1}{k}$ with $\left (\incidenceboundary{k-1}{k} \right )_{ij} = e^i_j$.
	\end{definition}
	\begin{proposition}[\textbf{Boundary operator and incidence matrix}]\label{prop:boundary_incidence_matrix}
	Let $\psi(\kchain{c}{k})=c_{(k)}$, then $\psi (\partial \kchain{c}{k} ) = \incidenceboundary{k-1}{k} c_{(k)}$.
	\end{proposition}
	\begin{proof}
	\[
	\psi( \partial \kchain{c}{k})  \stackrel{(\ref{eq::algTop_boundary})}{=}   \psi \left ( \sum_{ij}c^{j}e^{i}_{j}\tau_{(k-1),i} \right ) \stackrel{\rm Def.~\ref{def:coeff_vector_for_chains}}{=}\sum_{ij}c^{j}e^{i}_{j} \stackrel{\rm Def.~\ref{def:incidence_matrix_chains}}{=}\incidenceboundary{k-1}{k} c_{(k)}\;.
	\]
	\end{proof}
	
	
	An essential result that follows from the definition of the boundary operator is that applying the boundary operator twice results in a zero chain.
	
	\begin{theorem}[\textbf{The boundary of the boundary is empty}] \label{algebraic_topology::boundary_boundary}
		The boundary operator satisfies:
		\begin{equation}
			\partial\partial\kchain{c}{k} = \kchain{0}{k-2}, \quad \forall\kchain{c}{k}\in\kchainspacedomain{k}{D}. \label{eq::algTop_double_boundary}
		\end{equation}
	\end{theorem}
	\begin{proof}
		The proof follows from repeated application of Definition~\ref{def:boundary_k_cell} to $k$-cells which extends by Definition~\ref{algebraic::boundary_operator} to $k$-chains. See also \cite{spivak1998calculus,Massey2} for this derivation. 
	\end{proof}
This simple result will have profound consequences for the discrete operators to be defined. $\partial \circ \partial \equiv 0$ is illustrated in Figure~\ref{fig:boundary_boundary}.
	\begin{figure}[!ht]
			\centering
				\includegraphics[width=0.7\textwidth]{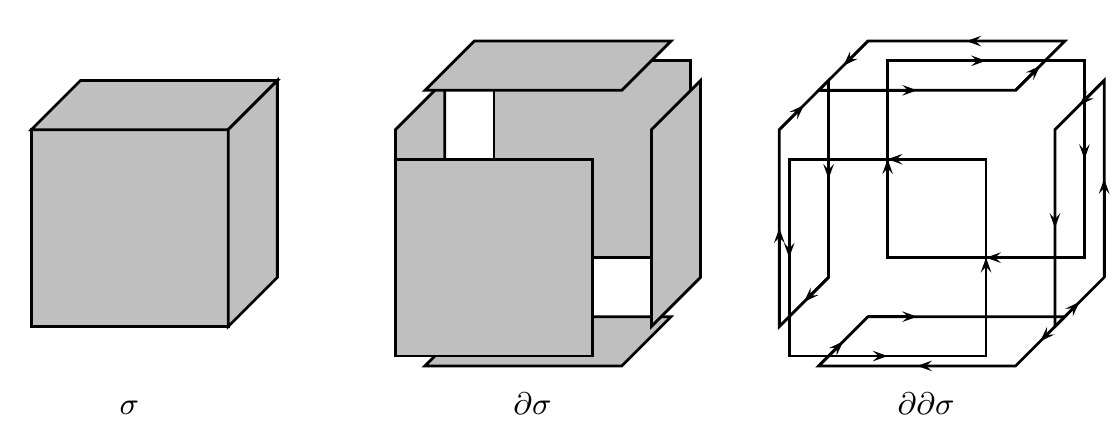}
				\caption{The boundary of the boundary of $2$-cell is zero because all edges have opposite orientations.}
				\label{fig:boundary_boundary}
		\end{figure}

\begin{proposition} \label{prop:Incidence_incidence_is_zero}
\[ \incidenceboundary{k-1}{k} \incidenceboundary{k}{k+1} \equiv 0 \;.\]
\end{proposition}
\begin{proof} For all $\kchain{c}{k+1} \in C_{k+1}(D)$ we have
\[ 0\stackrel{\rm Th.~\ref{algebraic_topology::boundary_boundary}}{=}\psi ( \partial \partial \kchain{c}{k+1}) \stackrel{\rm Prop.~\ref{prop:boundary_incidence_matrix}}{=} \incidenceboundary{k-1}{k} \psi( \partial \kchain{c}{k+1} ) \stackrel{\rm Prop.~\ref{prop:boundary_incidence_matrix}}{=} \incidenceboundary{k-1}{k} \incidenceboundary{k}{k+1} \psi(\kchain{c}{k+1}) \;.\]
\end{proof}
\begin{example}	Apply the boundary operator to the $2$-chain in \figref{fig:algTopology_chain}: $\kchain{c}{2} = \tau_{(2),1} + \tau_{(2),2} + \tau_{(2),3} + \tau_{(2),4}$.
	\begin{eqnarray*}
		\partial \kchain{c}{2} & = & \partial(\tau_{(2),1} + \tau_{(2),2} + \tau_{(2),3} + \tau_{(2),4}) \\
		& = & \tau_{(1),1} + \tau_{(1),4} + \tau_{(1),11} + \tau_{(1),12} - \tau_{(1),6} - \tau_{(1),3} - \tau_{(1),8} - \tau_{(1),7} \;.
	\end{eqnarray*}
	This is the $1$-chain depicted in the middle in \figref{fig:algTopology_chain}. Applying the boundary operator again we get:
	\[
		\partial\partial \kchain{c}{2} = \partial(\tau_{(1),1} + \tau_{(1),4} + \tau_{(1),11} + \tau_{(1),12} - \tau_{(1),6} - \tau_{(1),3} - \tau_{(1),8} - \tau_{(1),7}) = 0 \;.
	\]	
	This result could also be obtained by the use of incidence matrices, since they are a matrix representation of the topological boundary operator: $\incidenceboundary{k-2}{k-1}\incidenceboundary{k-1}{k} = 0$. Consider the cell complex in \figref{fig:algTopology_chain}. The incidence matrices $\incidenceboundary{0}{1}$ and $\incidenceboundary{1}{2}$ are
	\[
	\incidenceboundary{0}{1}=
		\tiny{
		\left[
			\begin{array}{cccccccccccc}
				-1 & 0 & 0 & 0 & 0 & 0 & -1 & 0 & 0 & 0 & 0 & 0 \\
				0 & -1 & 0 & 0 & 0 & 0 & 1 & -1 & 0 & 0 & 0 & 0 \\
				0 & 0 & -1 & 0 & 0 & 0 & 0 & 1 & 0 & 0 & 0 & 0 \\
				1 & 0 & 0 & -1 & 0 & 0 & 0 & 0 & -1 & 0 & 0 & 0 \\
				0 & 1 & 0 & 0 & -1 & 0 & 0 & 0 & 1 & -1 & 0 & 0 \\
				0 & 0 & 1 & 0 & 0 & -1 & 0 & 0 & 0 & 1 & 0 & 0 \\
				0 & 0 & 0 & 1 & 0 & 0 & 0 & 0 & 0 & 0 & -1 & 0 \\
				0 & 0 & 0 & 0 & 1 & 0 & 0 & 0 & 0 & 0 & 1 & -1 \\
				0 & 0 & 0 & 0 & 0 & 1 & 0 & 0 & 0 & 0 & 0 & 1 \\
			\end{array}
		\right]
		}, \quad
	\incidenceboundary{1}{2} =
		\tiny{
		\left[
			\begin{array}{cccc}
				1 & 0 & 0 & 0 \\
				-1 & 1 & 0 & 0 \\
				0 & -1 & 0 & 0 \\
				0 & 0 & 1 & 0 \\
				0 & 0 & -1 & 1 \\
				0 & 0 & 0 & -1 \\
				-1 & 0 & 0 & 0 \\
				0 & -1 & 0 & 0 \\
				1 & 0 & -1 & 0 \\
				0 & 1 & 0 & -1 \\
				0 & 0 & 1 & 0 \\
				0 & 0 & 0 & 1
			\end{array}
		\right]
		}
		\;,
	\]
which produces $\incidenceboundary{0}{1}\incidenceboundary{1}{2} = 0$, as expected. 
\end{example}

\begin{definition}[\textbf{Cycles and boundaries}]\label{def:cycles_boundaries} \cite{bookTonti}
A $k$-chain $\kchain{c}{k}$ for which $\boundary \kchain{c}{k} = \kchain{0}{k-1}$ is called a {\em cycle}.
A $k$-chain $\kchain{c}{k}$ which is the boundary of a $(k+1)$-chain $\kchain{b}{k+1}$, i.e. $\kchain{c}{k} = \boundary \kchain{b}{k+1}$ is called a {\em boundary}.\\
\end{definition}
The space of $k$-cycles in a cell complex $D$ is denoted by $Z_k(D)$ and the space of $k$-boundaries in $D$ is denoted by $B_k(D)$. Theorem~\ref{algebraic_topology::boundary_boundary} implies that every boundary is a cycle, $B_k(D) \subset Z_k(D)$, but the converse is generally not true. Therefore, consider the factor space $H_k(D)$ consisting of those cylces which are not boundaries
\begin{equation}
H_k(D) = \frac{Z_k(D)}{B_k(D)} \;.
\end{equation}
$H_k(D)$ is an equivalence class and any two cycles $\kchain{c}{k}$ and $\kchain{d}{k}$ are associated with the same element whenever the difference $\kchain{c}{k}-\kchain{d}{k}$ is a boundary. $H_k(D)$ is called the {\em homology group} and two elements which differ by a boundary are called {\em homologous}, \cite{tonti1975formal}. In many applications, dim$(H_k(D))$ is finite and this dimension is called the $k$-th {\em Betti number}. Two spaces can only be topologically the same, if the their respective Betti numbers are the same. 

\begin{example}\label{ex:cell_complex_with_hole}
Consider the cell-complex depicted in \figref{fig:hole} which contains a `hole' in the middle.
\begin{figure}[htb]
\includegraphics[width=0.3\textwidth]{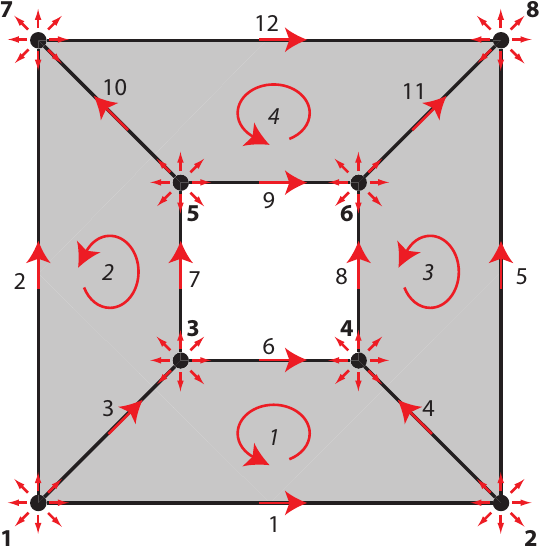}
\caption{A non contractible cell complex.}
\label{fig:hole}
\end{figure}
In \figref{fig:hole} all 0-cells have positive orientation by default. The orientation of the 1- and 2-cells are indicated in the figure. The incidence matrices which relates the 0-cells to the 1-cells and the 1-cells to the 2-cells, are given by
\[
\incidenceboundary{0}{1}=
\tiny{
\left[
\begin{array}{cccccccccccc}
-1&-1&-1&0&0&0&0&0&0&0&0&0\\
1&0&0&-1&-1&0&0&0&0&0&0&0\\
0&0&1&0&0&-1&-1&0&0&0&0&0\\
0&0&0&1&0&1&0&-1&0&0&0&0\\
0 & 0 & 0 & 0 & 0 & 0 & 1 & 0 & -1 & -1 & 0 & 0\\
0&0&0&0&0&0&0&1&1&0&-1&0\\
0&1&0&0&0&0&0&0&0&1&0&-1\\
0&0&0&0&1&0&0&0&0&0&1&1
\end{array}
\right]
}
,\quad\quad
\incidenceboundary{1}{2}=
\tiny{
\left[
\begin{array}{cccc}
1&0&0&0\\
0&-1&0&0\\
-1&1&0&0\\
1&0&-1&0\\
0&0&1&0\\
-1&0&0&0\\
0&1&0&0\\
0&0&-1&0\\
0&0&0&1\\
0&1&0&-1\\
0&0&-1&1\\
0&0&0&-1
\end{array}
\right]
}.
\]
The matrix $\incidenceboundary{0}{1}$ describes the connectivity between the 8 nodes and the 12 line segments. The range of the matrix is spanned by its 12 column vectors. Not all these column vectors are linearly independent. The rank of this matrix is 7, hence the range of this matrix has dimension 7. Since the dimension of the null space of matrix $\incidenceboundary{1}{2}$ is 8, there is one element in the null space of $\incidenceboundary{1}{2}$ that is not in the range of $\incidenceboundary{0}{1}$. This confirms that on a non-contractible domain, not every cycle is a boundary. The dimension of the homology group is in this example equal to 8-7=1. This corresponds to the number of `holes' in the domain (formally known as the Betti number). The harmonic chain corresponding to \figref{fig:hole} is given by
\[
\kchain{h}{1}=(1,-1,0,0,1,1,-1,1,-1,0,0,-1)^T\;,
\]
The topological Hodge decomposition of the space of $k$-chains is given by
\[
C_k=B_k\oplus H_k\oplus Z_k^c\, \;\;\; Z_k^c = C_k \backslash Z_k \;.
\]
The harmonic $k$-chain that belongs to the space $H_k$ is defined as
\[
\kchain{h}{k}=\{\partial\kchain{h}{k}=0\;|\;\nexists \kchain{a}{k+1}\in C_{k+1},\ \mathrm{such\ that}\ \kchain{h}{k}=\partial \kchain{a}{k+1}\}.
\]
\end{example}

\begin{remark}
Note that $H_k$ can only be obtained from {\em global} considerations. Therefore, the decomposition $C_k=B_k\oplus H_k\oplus Z_k^c$ can only be obtained globally. 
\end{remark}

The boundary operator defines a differential graded algebra of degree -1, $\left (\kchainspacedomain{k}{D},\partial \right )$,  from all the $k$-chain spaces of the cell complex, giving rise to a chain complex.
\begin{definition}[\textbf{Graded algebra of degree -1, $\left (\kchainspacedomain{k}{D},\partial \right )$ }] \cite{hatcher2002algebraic} A chain complex is a family $\{\kchainspacedomain{k}{D},\partial\}$ of $k$-chains and boundary operators, that constitutes a sequence
\begin{equation}
\cdots \stackrel{\partial}{\longleftarrow} \kchainspacedomain{k-1}{D} \stackrel{\partial}{\longleftarrow} \kchainspacedomain{k}{D} \stackrel{\partial}{\longleftarrow} \kchainspacedomain{k+1}{D} \stackrel{\partial}{\longleftarrow} \cdots.
\end{equation}
\end{definition}
Only when all $H^k(D) = \emptyset$ is this sequence an {\em exact sequence}.

In the previous section, we discussed how differential forms and operators acting on these differential forms behave under a mapping $\Phi$. Similar constructions are also feasible in algebraic topology. These maps are called {\em chain maps} and {\em cochain maps}.
\begin{definition}[\textbf{Chain maps}] \label{def:chain_map_f_sharp}
Let $\Phi\,:\,\manifold{M} \rightarrow \manifold{N}$ be a homeomorphism between manifolds $\manifold{M}$ and $\manifold{N}$, then this map induces a homomorphism between $k$-cubes in $\manifold{M}$ and $k$-cubes in $\manifold{N}$ given by
\begin{diagram}
I^k &\rTo^{\quad \Phi_\sharp \left ( \tau_{(k)} \right ) \quad}& C_k(D_\manifold{N})\\
\dTo^{\tau_{(k)}} & \ruTo_{\Phi} & \\
C_k(D_\manifold{M}) & &
\end{diagram}
where $\Phi_\sharp \left ( \tau_{(k)} \right ) = \Phi \circ \tau_{(k)}$ for any singular $k$-cube $\tau_{(k)}\,:\,I^k \,\rightarrow\, \manifold{M}$, for $k=0,1,2,\dots,n$. $\Phi_\sharp$ is extended to chains to be the homomorphism $\Phi_\sharp\,:\,C_k(D_\manifold{M}) \rightarrow C_k(D_\manifold{N})$, the {\em induced chain homomorphism}, by putting
\[ \Phi_\sharp (\kchain{c}{k}) = \Phi_\sharp \left ( \sum_i c^i \tau_{(k),i} \right ) = \sum_i c^i \Phi_\sharp (\tau_{(k),i}) \;.\]
\end{definition}

It can be shown, \cite{Massey1}, that $\Phi_\sharp$ maps non-degenerate $k$-cubes in $\manifold{M}$ into non-degenerate $k$-cubes in $\manifold{N}$ and that the diagram
\[
\begin{CD}
C_{k-1}(D_\manifold{N}) @<\partial<< C_k(D_\manifold{N})\\
@A\Phi_\sharp AA @A\Phi_\sharp AA\\
C_{k-1}(D_\manifold{M}) @<\partial<< C_{k}(D_\manifold{M}).
\end{CD}
\]
commutes. This follows directly from the fact that the face maps, Definition~\ref{def:Boundary_on_k_cube}, commute with the chain map: $\Phi_\sharp \left (A_i \tau^k \right ) = A_i \left (\Phi_\sharp \tau^k \right )$ and $\Phi_\sharp \left (B_i \tau^k \right ) = B_i \left (\Phi_\sharp \tau^k \right )$. Therefore, we have that
\begin{equation}
\partial \circ \Phi_\sharp = \Phi_\sharp \circ \partial,
\label{chainmapboundary}
\end{equation}
the chain map $\Phi_\sharp$ commutes with the boundary operator. The image of the boundary is the boundary of the image. This is the discrete equivalent of Proposition~\ref{prop:boundary_indep_chart}.

Since the boundary operator, $\partial$, commutes with the chain map, $\Phi_\sharp$, $\Phi_\sharp$ maps the cycles $Z_k(D_\manifold{M})$ into $Z_k(D_\manifold{N})$ and the boundaries $B_k(D_\manifold{M})$ into $B_k(D_\manifold{N})$, for all $k=0,\ldots,n$. Therefore, a chain map does not alter the topology. In particular, the incidence matrices for the cell complex $D_\manifold{M}$ and $D_\manifold{N}$ are identical.

\subsection{Cochains}
	\begin{definition}[\textbf{Cochains}]\label{def:cochains}
		\cite{hatcher2002algebraic} The space of $k$-cochains, $\kcochainspacedomain{k}{D}$, is the space  dual to the space of $k$-chains, $\kchainspacedomain{k}{D}$, defined as the set of all the linear maps (linear functionals), $\kcochain{c}{k}: \kchainspacedomain{k}{D}\spacemap\mathbb{R}$, and we write
		\begin{equation}
			 \duality{\kcochain{c}{k}}{\kchain{c}{k}} := \kcochain{c}{k}(\kchain{c}{k}) \;, \label{algtop::duality_pairing}
		\end{equation}
		to represent the duality pairing.
	\end{definition}
	
	\begin{remark} \label{rem:similarity_duality_algtop_and_diffgeom}
	Note the resemblance between (\ref{algtop::duality_pairing}) and (\ref{diffgeom::duality_pairing}). This similarity forms the basis of the mimetic framework discussed in this paper.
	\end{remark}
	
	\begin{remark}
	Strictly speaking $k$-cochains are {\em homomorphisms} from $C_k(D)$ into $\mathbb{R}$: $C^k(D) := \mbox{Hom} \left (C_k(D);\mathbb{R} \right )$.
	\end{remark}
	
	\begin{remark}
	Since the cochain spaces are the dual spaces of the chain spaces, we might as well write $C^k := C_k^*$.
	\end{remark}
	
	\begin{proposition}[\textbf{Canonical basis $k$-cochains}]\label{Canonical_basis_k_cochains}
		Given a basis of $\kchainspacedomain{k}{D}$, $\{\tau_{(k),i}\}$ with $i=1,\cdots,p$ with $p=\text{rank}(\kchainspacedomain{k}{\ccomplex{D}})$, there is a dual basis of $\kcochainspacedomain{k}{D}$, $\{\tau^{(k),i}\}$, with $i=1,\cdots,p$ with $p=\text{rank}(\kchainspacedomain{k}{\ccomplex{D}})$, such that:
		\begin{equation}
			\tau^{(k),i}(\tau_{(k),j}) = \delta^{i}_{j}\;.\label{eq:Canonical_basis_cochains}
		\end{equation}
		All linear functionals can be represented as a linear combination of the basis elements:
		\[
			\forall\kcochain{c}{k}\in\kcochainspacedomain{k}{D}\Rightarrow \kcochain{c}{k} = \sum_{i}c_{i}\tau^{(k),i} \;.
		\]
	\end{proposition}
		The proof of this proposition can be found in any book on algebraic topology, for instance, \cite{munkres1984elements}.
	With the duality relation between chains and cochains, we can define the formal adjoint of the boundary operator which constitutes a sequence on the spaces of $k$-cochains in the cell complex.
	
	\begin{definition}[\textbf{Coboundary operator}]\label{def:coboundary_operator}
		\cite{hatcher2002algebraic} The coboundary operator, $\delta:\kcochainspacedomain{k}{D}\spacemap\kcochainspacedomain{k+1}{D}$, is defined as the formal adjoint of the boundary operator:
		\begin{equation}
			\duality{\delta\kcochain{c}{k}}{\kchain{c}{k+1}} := \duality{\kcochain{c}{k}}{\partial\kchain{c}{k+1}}, \quad\forall\kcochain{c}{k}\in\kcochainspacedomain{k}{D} \text{ and  } \,\forall\kchain{c}{k+1}\in\kchainspacedomain{k+1}{D} \;. \label{eq::algTop_codifferential_dual}
		\end{equation}
	\end{definition}
	
	\begin{proposition}\label{coboundary_coboundary_is_zero}
		The coboundary operator is nilpotent,
		\begin{equation}
			\delta\delta\kcochain{c}{k} = 0, \quad \forall\kcochain{c}{k}\in\kcochainspacedomain{k}{D}\;.\label{algtop::stokes}
		\end{equation}
	\end{proposition}
	\begin{proof}
		The proof follows directly from \theoremref{algebraic_topology::boundary_boundary}: $\forall\kcochain{c}{k}\in\kcochainspacedomain{k}{D} \text{ and  } \,\forall\kchain{c}{k}\in\kchainspacedomain{k}{D}$
		\[
			\duality{\delta\delta\kcochain{c}{k}}{\kchain{c}{k}} \stackrel{(\ref{eq::algTop_codifferential_dual})}{=} \duality{\kcochain{c}{k}}{\partial\partial\kchain{c}{k}} \stackrel{(\ref{eq::algTop_double_boundary})}{=} 0\;\; \Longrightarrow \;\; \delta\delta\kcochain{c}{k} = 0 \;.
		\]
	\end{proof}
	
\begin{definition}[\textbf{Cochain complex}]
\cite{hatcher2002algebraic} A cochain complex is a family $(\kcochainspacedomain{k}{D},\delta )$ of $k$-cochains and coboundary operators, that constitutes a sequence
\begin{equation}
\cdots \stackrel{\delta}{\longrightarrow}  \kcochainspacedomain{k-1}{D} \stackrel{\delta}{\longrightarrow} \kcochainspacedomain{k}{D} \stackrel{\delta}{\longrightarrow} \kcochainspacedomain{k+1}{D} \stackrel{\delta}{\longrightarrow} \cdots.
\end{equation}
\end{definition}
	
%
%
%
	Analogous to Definition~\ref{def:coeff_vector_for_chains}, we have
	\begin{definition}[\textbf{Coefficient vector for cochains}]\label{def:coeff_vector_for_cochains}
The space of $k$-cochains, $\kcochainspacedomain{k}{\ccomplex{D}}$, can be represented by a column vector containing only the coefficients of the chain. That is, there is an isomorphism $\bar{\psi}$:
	\begin{equation}
		\bar{\psi}\,:\, \kcochainspacedomain{k}{\ccomplex{D}}\mapsto\mathbb{R}^{p}, \quad p=\mathrm{rank}(\kchainspacedomain{k}{\ccomplex{D}})\;,\label{eq::algTop_isomorphism_map_cochains}
	\end{equation}
	defined by
	\begin{equation}
		\bar{\psi}(\kcochain{c}{k}) = \bar{\psi}(\sum_{i}c_{i}\tau^{(k),i}) =  [c_{1} \cdots c_{p}]^{T}, \quad p=\text{rank}(\kcochainspacedomain{k}{\ccomplex{D}})\;,\label{eq::algTop_isomorphism_definition_cochains}
	\end{equation}
	where the rank of $\kcochainspacedomain{k}{\ccomplex{D}}$ is the number of basis $k$-cochains in the cell complex $D$ and the $c_i$ are the coefficients of the vector $\mathbf{c}^{(k)}$ in Proposition~\ref{Canonical_basis_k_cochains}. The $k$-cochain, $\kcochain{c}{k}$, is printed in boldface, whereas the vector $c^{(k)}$ of coefficients is printed in regular face.
	\end{definition}
	
	\begin{proposition}[\textbf{Duality pairing in terms of coefficients}]\label{prop:Dual_pair_coeff}
	Duality pairing of $k$-cochains with $k$-chains, (\ref{algtop::duality_pairing}), in terms of the coefficients, $\bar{\psi}(\kcochain{c}{k})=c^{(k)}$ and $\psi(\kchain{c}{k})=c_{(k)}$, is given by
	\begin{equation}
	\left \langle \kcochain{c}{k},\kchain{c}{k} \right \rangle = \left ( c^{(k)} \right )^T c_{(k)} \;.
	\end{equation}
	\end{proposition}
	\begin{proof}
	\[
	\left \langle \kcochain{c}{k},\kchain{c}{k} \right \rangle = \sum_i \sum_j c_i c^j \left \langle \tau^{(k),i},\tau_{(k),j} \right \rangle \stackrel{(\ref{eq:Canonical_basis_cochains})}{=} \sum_i \sum_j c_i c^j \delta^i_j = \sum_i c_i c^i = \left ( c^{(k)} \right )^T c_{(k)}\;.
	\]
	\end{proof}
	\begin{proposition}[\textbf{Incidence matrix for coboundary operator}]\label{prop:Incidence_matrix_coboundary}
	Let $\bar{\psi}(\kcochain{c}{k}) = c^{(k)}$, then $\bar{\psi}\left ( \delta \kcochain{c}{k} \right ) = \incidencederivative{k+1}{k}c^{(k)}$, where
	\begin{equation} \incidencederivative{k+1}{k} = (\incidenceboundary{k}{k+1})^{T}\;.\end{equation}
	\end{proposition}
	\begin{proof} For all $\kcochain{c}{k} \in C^k(D)$ and for all $\kchain{c}{k+1} \in C_{k+1}(D)$ we have
	\begin{eqnarray*}
	\left ( \bar{\psi} (\delta \kcochain{c}{k}) \right )^T c_{(k+1)} & \stackrel{\rm Prop.~\ref{prop:Dual_pair_coeff}}{=} & \langle \delta \kcochain{c}{k},\kchain{c}{k+1} \rangle  \stackrel{\eqref{eq::algTop_codifferential_dual}}{=}  \langle \kcochain{c}{k},\partial \kchain{c}{k+1} \rangle \stackrel{\rm Prop.~\ref{prop:Dual_pair_coeff}}{=} \\
	\left (c^{(k)} \right )^T \left ( \incidenceboundary{k}{k+1}c_{(k+1)} \right )
	& = & \left [ \left ( \incidenceboundary{k}{k+1} \right )^T c^{(k)}\right ]^T c_{(k+1)}:= \left ( \incidencederivative{k+1}{k} c^{(k)} \right )^T c_{(k+1)}
	\end{eqnarray*}
	\[ \Longrightarrow \;\; \bar{\psi} (\delta \kcochain{c}{k}) = \incidencederivative{k+1}{k} c^{(k)} \;.\]
	\end{proof}
\begin{proposition}
Let $\incidencederivative{k+1}{k}$ be the incidence matrices for the coboundary operator, then
\begin{equation} \incidencederivative{k+1}{k}\incidencederivative{k}{k-1}= 0\;.\end{equation}
\end{proposition}
\begin{proof}
\[ \incidencederivative{k+1}{k}\incidencederivative{k}{k-1} \stackrel{\rm Prop.~\ref{prop:Incidence_matrix_coboundary}}{=} \left ( \incidenceboundary{k}{k+1} \right )^T \left ( \incidenceboundary{k-1}{k} \right )^T = \left ( \incidenceboundary{k-1}{k} \incidenceboundary{k}{k+1} \right )^T \stackrel{\rm Prop.~\ref{prop:Incidence_incidence_is_zero}}{=}0. \]
\end{proof}
	
	
	
	Recalling the duality pairing of differential forms and manifolds \eqref{diffgeom::duality_pairing} and the duality pairing of chains and cochains \eqref{algtop::duality_pairing} one clearly notices the algebraic similarity between the continuous and the discrete worlds. Expression \eqref{eq::algTop_codifferential_dual} is nothing but a discrete Stokes' theorem.
	
	\begin{definition}[\textbf{Cocycles, $k$-coboundaries and cohomologous cochains}] \cite{tonti1975formal}
The cochains $\kcochain{c}{k}$ for which $\delta \kcochain{c}{k} = \kcochain{0}{k+1}$ are called {\em cocycles}. The set of all $k$-cocycles is denoted by $Z^k(D)$. A $k$-cochain that can be written as the coboundary of a $(k-1)$-cochain, $\kcochain{c}{k} = \delta \kcochain{d}{k-1}$, is called a {\em $k$-coboundary}. The space of all $k$-coboundaries is denoted by $B^k(D)$.
\end{definition}

\begin{corollary}[\textbf{Cochain space decomposition}]\label{cochainspacedecomposition}
From Proposition~\ref{coboundary_coboundary_is_zero} is follows that $B^k(D) \subset Z^k(D)$. We therefore consider the {\em cohomology group} 
\begin{equation} H^k(D) = \frac{Z^k(D)}{B^k(D)} \;,\end{equation}
of all $k$-cocycles which are not $k$-coboundaries. The space $H^k(D)$ is the topological equivalent of the space of harmonic forms ${\mathcal H}^k({\mathcal M})$ defined in Section~\ref{subsec:HodgeDecomposition}.
The space of $k$-cochains can be decomposed as
\[ C^k = B^k \oplus H^k \oplus \left  ( Z^k \right )^c \;.\]
Compare this decomposition with the Hodge decomposition given in (\ref{eq:generic_Hodge_decomposition}).
\end{corollary}	

\begin{proposition}[\textbf{Cochain well-posedness}]\label{uniquecochain}
Consider the following cochain relation, $\delta\kcochain{a}{k}=\kcochain{f}{k+1}$, with $\kcochain{a}{k}\in (Z^k)^c$ and $\kcochain{f}{k+1}\in B^{k+1}$. Then there exists a solution $\kcochain{a}{k}$ and the solution is determined up to a cochain in $Z^k$.
\end{proposition}
\begin{proof}
There exists a solution $\kcochain{a}{k}$, because $\kcochain{f}{k+1}\in B^{k+1}$. Now, suppose there exists two solutions, $\kcochain{a}{k}_1$ and $\kcochain{a}{k}_2$, such that $\delta\kcochain{a}{k}_1=\kcochain{f}{k+1}$ and $\delta\kcochain{a}{k}_2=\kcochain{f}{k+1}$. Then $\delta(\kcochain{a}{k}_1-\kcochain{a}{k}_2)=\kcochain{0}{k+1}$. Since $\kcochain{a}{k}_1-\kcochain{a}{k}_2\in Z^k$, the solution for $\delta\kcochain{a}{k}=\kcochain{f}{k+1}$ is unique for $\kcochain{a}{k}\in (Z^k)^c$.
\end{proof}
\begin{example}[\textbf{Example~\ref{ex:cell_complex_with_hole} continued}]\label{ex:domain_with_hole_continued}
The dimension of the cohomology is equal to the dimension of the homology due to the duality pairing. The harmonic cochain corresponding to \figref{fig:hole} is given by
\[
\kcochain{h}{1}=\alpha(1,-1,0,0,1,1,-1,1,-1,0,0,-1)^T,\quad\alpha\in\mathbb{R}.
\]
This harmonic cochain models a circulation around the hole in \figref{fig:hole}, with strength $\alpha$.
\end{example}

\begin{remark}
Note that the harmonic cochain is \emph{global}. Only global considerations lead to the determination of the harmonic cochains, which is the main reason these solutions cannot be represented by local methods, such as finite difference, finite volume or finite element methods.
\end{remark}

\begin{definition}[\textbf{Induced cochain map}]\label{Induced_cochain_map} 
Let $\Phi_\sharp\,:\, C_k(D_\manifold{M}) \rightarrow C_k(D_\manifold{N})$, $k=0,1,\ldots,n$, then for every $k$-cochain. $\kcochain{c}{k}_n \in C^k(D_\manifold{N})$,there exists a $k$-cochain $\kcochain{c}{k}_m \in C^k(D_\manifold{M})$ such that
\[
\langle \kcochain{c}{k}_n,\Phi_\sharp \kchain{c}{k} \rangle = \langle \kcochain{c}{k}_m,\kchain{c}{k} \rangle := \langle \Phi^\sharp \kcochain{c}{k}_n \;,\kchain{c}{k} \rangle\;,\quad \forall \kchain{c}{k} \in C_k(D_\manifold{M})\;,
\]
with the {\em cochain map} $\Phi^\sharp$ defined by $\kcochain{c}{k}_m = \Phi^\sharp \kcochain{c}{k}_n$.
\end{definition}
\begin{remark}\label{rem:relation_f_sharp_and_pullback}
The cochain map $\Phi^\sharp$ satisfies $\langle \kcochain{c}{k}_m,\Phi_\sharp \kchain{c}{k} \rangle = \langle \Phi^\sharp \kcochain{c}{k}_n, \kchain{c}{k} \rangle$. Compare this relation with the duality between a space map $\Phi$ and its pullback $\Phi^\star$ as discussed in Proposition~\ref{prop::pullback_integral}.
\end{remark}

\begin{corollary}
The cochain map $\Phi^\sharp$ commutes with the coboundary operator $\delta$,
\begin{equation}
\delta \circ \Phi^\sharp = \Phi^\sharp \circ \delta,
\end{equation}
\[
\begin{CD}
C^{k-1}(D_\mathcal{N})@>\delta>>C^k(D_\mathcal{N})\\
@V\Phi^\sharp VV @V\Phi^\sharp VV \\
C^{k-1}(D_\mathcal{M})@>\delta>>C^k(D_\mathcal{M}).
\end{CD}
\]
\end{corollary}
\begin{proof}
\begin{gather*}
\langle \delta \Phi^\sharp \kcochain{c}{k},\kchain{c}{k} \rangle \stackrel{\rm Def.~\ref{def:coboundary_operator}}{=} \langle \Phi^\sharp \kcochain{c}{k},\partial \kchain{c}{k}
\rangle \stackrel{\rm Def.~\ref{Induced_cochain_map}}{=} \langle \kcochain{c}{k},\Phi_\sharp \partial \kchain{c}{k} \rangle \stackrel{\eqref{chainmapboundary}}{=} \langle \kcochain{c}{k},\partial \Phi_\sharp  \kchain{c}{k} \rangle\\
\stackrel{\rm Def.~\ref{def:coboundary_operator}}{=} \langle \delta \kcochain{c}{k},\Phi_\sharp  \kchain{c}{k} \rangle \stackrel{\rm Def.~\ref{Induced_cochain_map}}{=} \langle \Phi^\sharp \delta \kcochain{c}{k},\kchain{c}{k} \rangle\;,\;\;\; \forall \kcochain{c}{k} \in C^k(D_\manifold{N})\;, \forall \kchain{c}{k} \in C_k(D_\manifold{M}) \;.
\end{gather*}
\end{proof}

\begin{remark}
The commutative property $\delta \circ \Phi^\sharp = \Phi^\sharp \circ \delta$ is the discrete analogue of $\ederiv \circ \Phi^\star = \Phi^\star \circ \ederiv$, see Proposition~\ref{prop:commutation_pullback_ederiv}.
\end{remark}

Since the cochain map $\Phi^\sharp$ commutes with the coboundary operator, $\Phi^\sharp$ maps cocycles onto cocycles and $k$-coboundaries onto $k$-coboundaris, for all $k=0,\ldots,n$. This property ensures that the discrete Hodge decomposition retains its structure when we deform the cell complex with a chain map.

\subsection{Dual complex}\label{sec:dualcomplex}
With every cell complex $D$ one can associate a compatible dual cell-complex, $\tilde{D}$, \cite{tonti1975formal}. Before we can define a dual cell-complex, the adjoint of the faces should be defined.
\begin{definition}[\textbf{Faces and cofaces}] \cite{tonti1975formal} The faces of a $k$-cell are those $(k-1)$-cells that form the boundary of the $k$-cell. The \emph{cofaces} of a $k$-cell are those $(k+1)$-cells which have the $k$-cell as common face.
\end{definition}
\begin{example}
Consider two neighboring rooms in a hotel. The two neighboring rooms have the wall in between them as common face. Therefore, the cofaces of the wall between the rooms are the two rooms.
\end{example}
\begin{definition}[\textbf{Dual cell}]\label{def:dualcell} With every $k$-cell, $\tau_{(k)}$, in $D$ there corresponds a $(n-k)$-cell, $\tilde{\tau}_{(n-k)}$, in the dual complex $\tilde{D}$, such that the dual cells of the boundary of a $k$-cell are the cofaces of the dual cell $\tilde{\tau}_{(n-k)}$.
\end{definition}
\begin{definition}[\textbf{Dual cell-complex}]\label{def:dualcellcomplex}
A dual cell-complex, $\tilde{D}$, is the smallest cell complex according to \defref{cellcomplex}, that contains all dual cells $\tilde{\tau}_{(n-k)}$ as defined in \defref{def:dualcell}.
\end{definition}

If we denote the map, which associates a cell $\tau_{(k)}$ with its dual $\tilde{\tau}_{(n-k)}$, by $*$, i.e. $*\tau_{(k)}=\tilde{\tau}_{(n-k)}$ and if we denote the coface operator by $\partial^*$ then \defref{def:dualcell} states that the following diagram commutes.
\[
\begin{CD}
\tau_{(p)}@>\partial>>\partial\tau_{(p)} \\
@VV*V @VV*V\\
\tilde{\tau}_{(n-p)}@>\partial^*>>\partial^*\tilde{\tau}_{(n-p)}
\end{CD}
\]
Therefore we have $*\partial=\partial^**$. An example of these relations is shown in \figref{fig:marcDualCell}.
\begin{figure}[htb]
\centering
\includegraphics[width=0.3\textwidth]{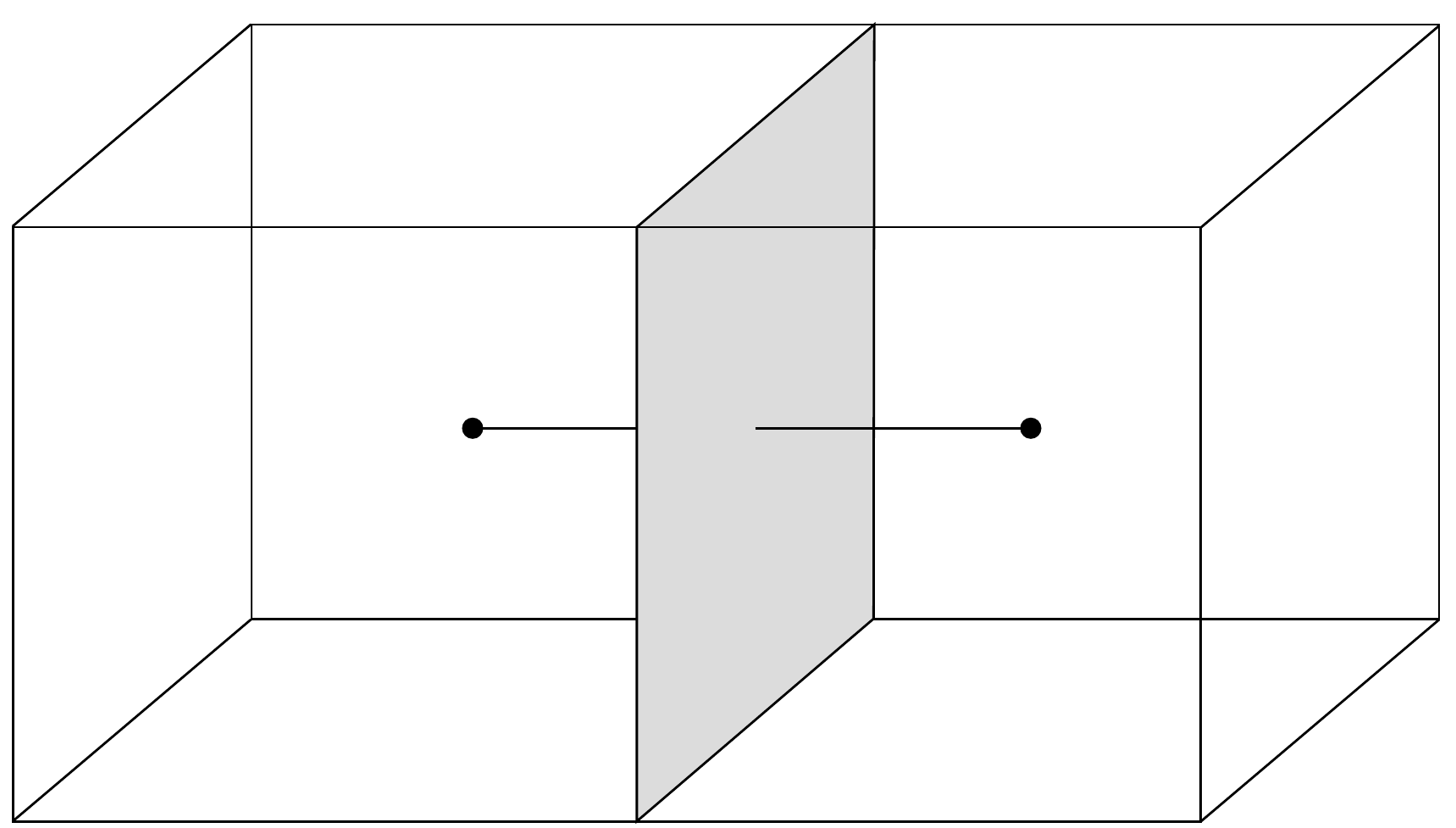}
\caption{Pictorial view of \defref{def:dualcell}: The dual of a line segment in $D\subset\mathbb{R}^3$ is gray shaded surface in $\tilde{D}$. The cofaces of this gray surface are the two adjoining volumes in $\tilde{D}$. The boundary of the line segment are the two endpoints in $D$. The dual of the two endpoints are again the volumes in $\tilde{D}$ surrounding these points.}
\label{fig:marcDualCell}
\end{figure}

We make a distinction in cell-complexes with and without boundary.
\begin{definition}[\textbf{Dual cell complex without boundary}]\label{def:dualgridwithout}
If the collection of all dual cells of $D$, denoted by $*D$, constitutes a cell complex according to \defref{cellcomplex}, then the dual cell complex $\tilde{D}\equiv *D$ and both the cell complex $D$ and its dual $\tilde{D}$ are cell complexes \emph{without boundary}.
\end{definition}

\begin{corollary}
For a cell complex without boundary the cell-complex $\tilde{D}$ is dual to the cell complex $D$ and the cell complex $D$ is dual to the cell complex $\tilde{D}$ modulo orientation. As a consequence $*^{-1}$ exists and is $*^{-1}=\pm *$. This allows us to express the coface operator in terms of the dual operator $*$ and the boundary operator $\partial$ by
\begin{equation}
 *\partial=\partial^**\quad\Rightarrow\quad\partial^*=*\partial *^{-1}=\pm *\partial *.
\end{equation}
\end{corollary}

\begin{figure}[htp]
\centering
\includegraphics[width=1.0\textwidth]{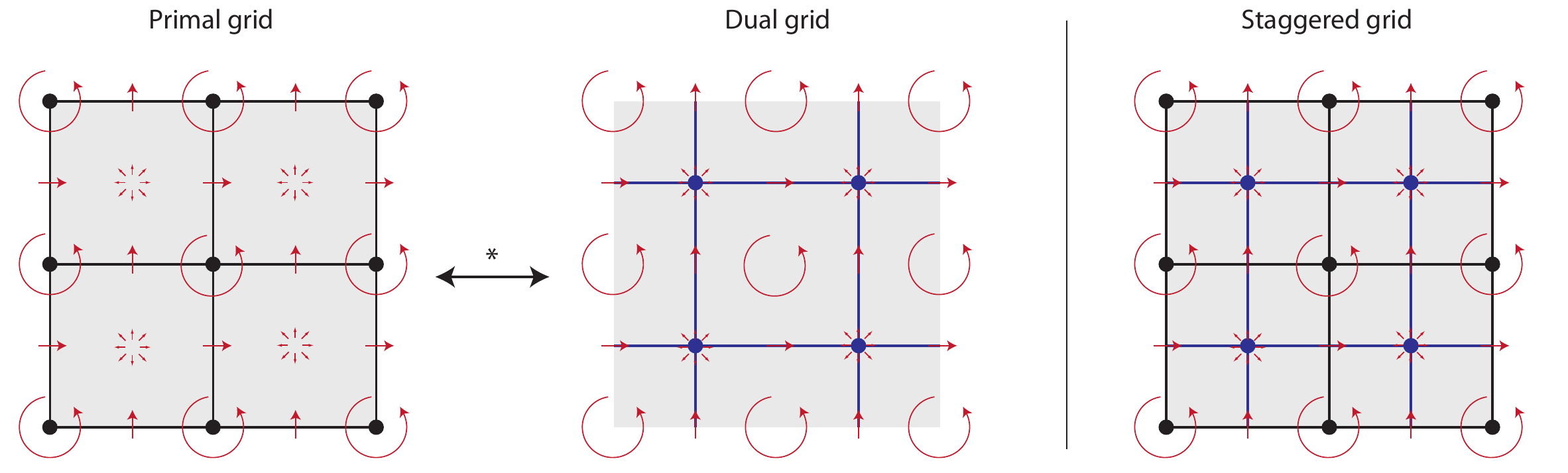}
\caption{Example of a primal grid and dual grid is shown (not necessarily a cell complex), with for every cell $\tau_{(k)}$ the associated dual cell $\tilde{\tau}_{(n-k)}$ and corresponding orientation. On the right the staggered grid, an overlap of both grids, is shown.}
\label{fig:algTopology_primal_dual_grids}
\end{figure}
\begin{remark}
The dual grid in \figref{fig:algTopology_primal_dual_grids} is only a cell complex in case the top side is connected to the bottom side and the left side to the right side. This makes this cell complex on the surface of a torus.
\end{remark}
\begin{remark}
Note that the primal cell complex, $D$, is chosen outer oriented. Then by duality, the dual cell complex, $\tilde{D}$, is inner oriented. In fact, the orientation itself does not change, only the corresponding cell changes. This was shown before in \figref{fig:diffGeom_orientation}.
\end{remark}
\begin{remark}
If we equip the cell complex, $D$, with an outer orientation, the the dual complex, $\tilde{D}$, models geometric objects with inner orientation. Alternatively, inner orientation could be represented on the cell complex $D$, in which case the outer representation is modeled on the dual cell complex $\tilde{D}$. In this respect, dual cell complexes are able to model the inner- and outer-orientation as discussed in \secref{sec:manifolds}.
\end{remark}

\subsection{The boundary of cell complexes}\label{sec:AT_boundary_cell_complex}

A collection of $k$-cells forms a cell complex if the boundary of these $k$-cells are also in the cell complex. Therefore, cell complexes consist of an interior $D_i$ and a boundary part $D_b$, where $D=D_i\cup D_b$. The boundary part contains cells up to degree $(n-1)$. Since $\tilde{D}$ is dual $D$, we have that $\tilde{D}_i\cup\tilde{D}_b$ is dual to $D_i\cup D_b$. More precise, $\tilde{D}_i=\ast D_i$ and $\tilde{D}_b = \ast D_b$. The individual parts do not need to be cell complexes themselves. In case $D_i$ is a cell complex on a manifold with boundary, then $\tilde{D}_i$ is not a cell complex, because not all faces of the $k$-cells in $\tilde{D}_i$, $k=1,\hdots,n$, are also in $\tilde{D}_i$. 
Examples can be found in \cite{thorpe1976lecture} and Examples~\ref{ex:1D_boundary_of_cell_complex} and \ref{ex:2D_boundary_of_cell_complex} below.

\begin{definition}\textbf{(The boundary of cell complexes)}\label{def:Boundary_of_cell_complexes}
Let $D$ be a cell complex. If $\ast D$ is not a cell complex, let $\tilde{D}$ be the smallest cell complex which contains all the $k$-cells of $*D$. All the $k$-cells in $\tilde{D} \setminus *D$ form a $(n-1)$-dimensional cell complex, $\tilde{D}_b:=\partial \tilde{D}$ called the {\em boundary of $\tilde{D}$}. The dual cells of $\partial \tilde{D}$ with respect to the $(n-1)$-dimensional embedding space, $D_b:=\partial D = \ast ( \partial \tilde{D} )$ form a $(n-1)$-dimensional subcomplex of $D$, called the boundary of $D$. A dual cell in $\tilde{D}_b$ is given by $\tilde{\tau}_{(n-1-k)}=*\tau_{(k)}$. The boundary and its dual are a cell complexes, because the boundary of the boundary is empty and Definition~\ref{def:dualgridwithout}.
\end{definition}

\begin{example}[\textbf{1D primal- and dual- cell complex}]\label{ex:1D_boundary_of_cell_complex}
Let the interior part of the primal cell complex, $D_i$, be given by three 0-cells, connected by two 1-cells, see \figref{fig:dualcellcomplex1D}. Then its dual, $\tilde{D}_i$, consists of two 0-cells and three 1-cells. The outer two 1-cells are open ended, and therefore $\tilde{D}_i$ is not a cell complex. To make it a cell complex, $\tilde{D}_i$ must be closed by adding the boundary cells, $\tilde{D}_b$. The boundary part $\tilde{D}_b$ consists of two 0-cells. The dual of $\tilde{D}_b$ are the boundary cells of the primal cell complex. In this one-dimensional example, $D_b$ also consists of two 0-cells. Note that for the primal cell complex, the boundary cells in $D_b$ coincide with the cells in $D_i$. Although the boundary points in $D_b$ were already contained in $D_i$, the addition of {\em orientation} requires us to treat the boundary points as distinct points from $D_i$. This is because outer orientation of a point in a 1D embedding space differs from outer orientation of a point in a 0D embedding space, see \exampleref{ex:outer_orientation_points}. So $D_i$ and $D_b$ together form the primal cell complex $D$, and $\tilde{D}_i$ and $\tilde{D}_b$ together form the dual cell complex $\tilde{D}$.
\begin{figure}[htp]
\centering
\includegraphics[width=0.7\textwidth]{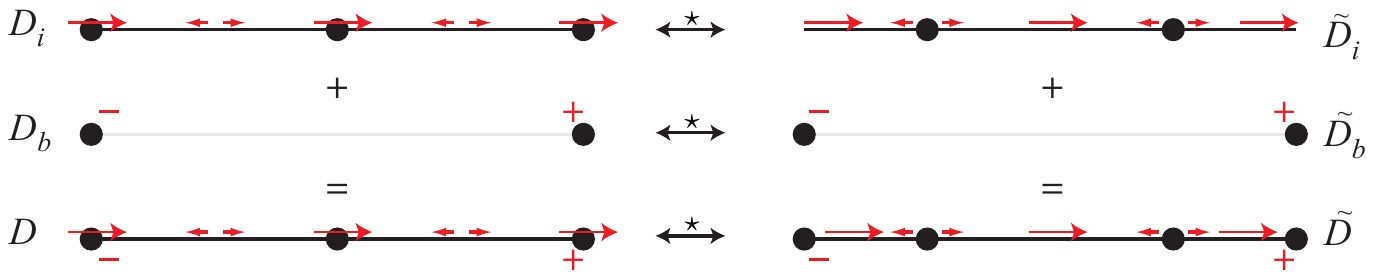}
\caption{Primal and dual cell complexes, split into their interior and boundary part. The corresponding orientation to all $k$-cells are indicated. Note that the end-points in $D$ have two types of orientation.}
\label{fig:dualcellcomplex1D}
\end{figure}
\end{example}

\begin{example}[\textbf{2D primal- and dual cell complex}]\label{ex:2D_boundary_of_cell_complex}
Now consider in two dimensions the interior part of the primal cell complex, $D_i$, as given in \figref{fig:dualcellcomplex2D}. The corresponding dual is the interior part $\tilde{D}_i$, of the dual cell complex $\tilde{D}$. Then $\tilde{D}$ becomes a cell complex by adding the boundary cells $\tilde{D}_b$. Note that there do \emph{not} exist 0-cells at the corners. The dual of $\tilde{D}_b$ are the boundary cells $D_b$. They again coincide with the cells in $D_i$. Although $D_b$ is already contained in $D_i$ as points and lines, the role of the points and line segments is completely different. This is indicated by the orientations in \figref{fig:dualcellcomplex2D}. The outer orientation along the boundary -- a 1D cell complex -- differs from the outer orientation of these same geometric objects considered as elements from $D_i$ embedded in 2D.
\begin{figure}[htb]
\centering
\includegraphics[width=0.6\textwidth]{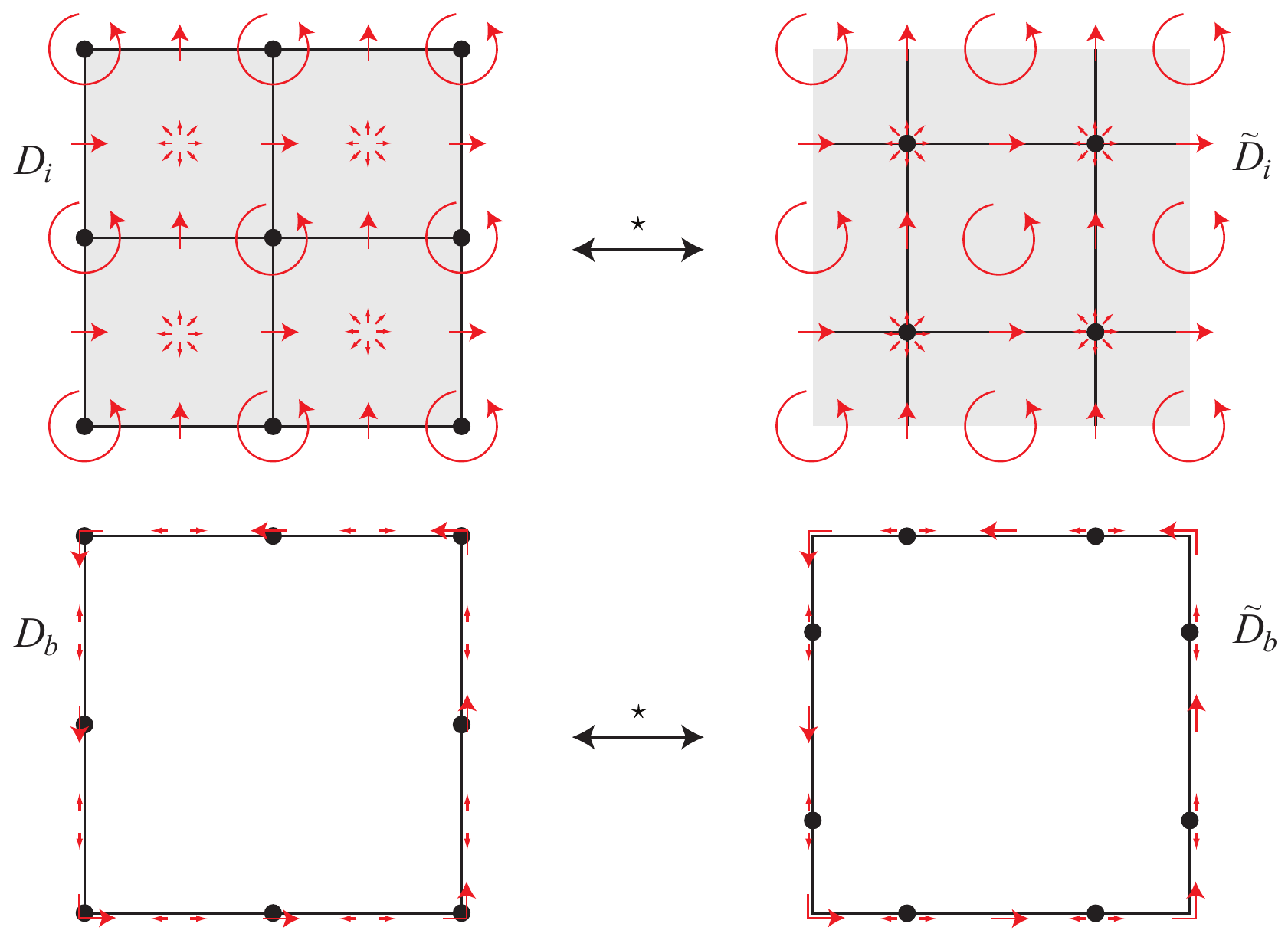}
\caption{Interior and boundary parts of the primal and dual cell complexes, $D=D_i\cup D_b$ and $\tilde{D}=\tilde{D}_i\cup\tilde{D}_b$, respectively. Orientation of all cells are included.}
\label{fig:dualcellcomplex2D}
\end{figure}
\end{example}

\begin{remark}
The boundaries $D_b$ and $\tilde{D}_b$ are boundaryless cell complexes.
\end{remark}

The construction of the boundary of a cell complex and its dual given above, indicates that the boundary of a cell complex is disjoint from the cell complex itself. In order to formally 'glue' the boundary to its cell complex, we introduce the chain inclusions $\imath_\sharp\,:\, C_k\left (\partial D \right ) \rightarrow C_k \left ( D \right )$ and $\tilde{\imath}_\sharp\,:\, \partial C_k(\tilde{D})\rightarrow C_k(\tilde{D})$, such that, for $k=0,\dots, n-1,$
\begin{equation}
\imath_\sharp \left ( \tau_{(k)} \in C_k \left ( \partial D \right ) \right ) = \tau_{(k)} \in C_k \left ( D \right )\;,
\quad\mathrm{and}\quad
\tilde{\imath}_\sharp \left ( \tilde{\tau}_{(k)} \in C_k ( \partial \tilde{D} )  \right ) = \tilde{\tau}_{(k)} \in C_k (\tilde{D})\;.
\end{equation}
Chain inclusions $\imath_\sharp$ and $\tilde{\imath}_\sharp$ are the discrete analogues of the continuous inclusion map defined in Definition~\ref{def:inclusion_map}. Using duality pairing between chains and cochains, we can define the associated cochain maps $\imath^\sharp\,:\, C^k\left ( D \right )\rightarrow C^k \left (\partial D \right )$ and $\tilde{\imath}^{\,\sharp}: C^k(\tilde{D})\rightarrow C^k( \partial \tilde{D})$ of the inclusion maps.
\begin{definition}[\textbf{Discrete trace operator}]
For all $c^{(k)}\in C^k( D)$, there must exist $b^{(k)} \in C^k(\partial D)$, such that
\begin{equation}
\left \langle b^{(k)}, c_{(k)} \right \rangle = \left \langle c^{(k)},\imath_\sharp ( c_{(k)} ) \right \rangle \;,\;\;\; \forall c_{(k)} \in C_k\left ( \partial D \right ) \;.
\end{equation}
The map $\imath^\sharp$ is then defined by
\begin{equation}
\imath^\sharp \left ( c^{(k)} \right ) = b^{(k)} \,.
\end{equation}
The cochain map $\tilde{\imath}^{\,\sharp}$ is defined similarly.
The cochain maps $\imath^\sharp$ and $\tilde{\imath}^{\,\sharp}$ are the discrete analogues of the of the pullback of the inclusion map given in Definition~\ref{def::diffGeometry_trace_operator}. We therefore will write $\tr$ and $\tilde{\tr}$ instead of $\imath^\sharp$ and $\tilde{\imath}^{\,\sharp}$.
\end{definition}

\begin{remark}
Note that the inclusion chain maps can be defined more generally: Let $K$ be a sub-cell complex in $D$, then we can define the map $\imath_\sharp : C_k(K) \rightarrow C_k(D)$ and its associated dual operation on cochains. Although useful in some applications, in this paper we restrict ourselves to $K=\partial D$.
\end{remark}

If we associate the primal complex $D$ with a cell complex endowed with outer orientation and the dual complex $\tilde{D}$ with a cell complex endowed with inner orientation, we have

\begin{definition} \textbf{(Tangent $k$-cochains)}
A $k$-cochain $\kcochain{c}{k} \in \kcochainspacedomain{k}{D}$ is called {\em parallel} or {\em tangent} to the cell complex $D$, if $\tr\left ( \kcochain{c}{k} \right ) = \kcochain{0}{k}$. We will denote the set of all tangent $k$-cochains on a given cell complex $D$ by $C_t^k(D)$.
\end{definition}

\begin{definition} \textbf{(Normal $k$-cochains)}
A $k$-cochain $\tilde{c}^{(k)} \in \kcochainspacedomain{k}{\tilde{D}}$ is called {\em perpendicular} or {\em normal} to the cell complex $\tilde{D}$, if $\tr \left ( \tilde{c}^{(k)} \right ) = \tilde{0}^{(k)}$. We will denote the set of all normal $k$-cochains on the cell complex $\tilde{D}$ by $C_n^k(\tilde{D})$.
\end{definition}

If the dual cells are labeled with the same number as the associated primal cells, and if the orientations of primal and dual cells are in agreement with the orientation of the embedding space, then the incidence matrix $\tilde{\mathsf{E}}_{(p-1,p)}$ on the dual complex is related to the incidence matrices on the primal complex by
\begin{equation}
\tilde{\mathsf{E}}_{(p-1,p)}=\mathsf{E}^T_{(n-p,n-p+1)},\quad p=1,\hdots,n.
\label{Relation_incidence_on_dual_and_primal}
\end{equation}
\begin{remark}
This seemingly uninteresting relation has quite some consequences for numerical operators to be developed. The incidence matrices encode the boundary operator, $\partial$, but by Stokes equation, also the coboundary operator, $\delta$. The coboundary operator is the discrete analogue of the exterior derivative. Take for instance $p=1$ in \eqref{Relation_incidence_on_dual_and_primal}, then $\tilde{\mathsf{E}}^{(1,0)}=\tilde{\mathsf{E}}^T_{(0,1)}$ is the discrete analogue of the exterior derivative acting on 0-forms. We identified this operator with the gradient operator.
\begin{equation}
\mathsf{E}^{(n,n-1)}=\left(\mathsf{E}_{(n-1,n)}\right)^T\stackrel{\eqref{Relation_incidence_on_dual_and_primal}}{=}\tilde{\mathsf{E}}_{(0,1)}=\left(\tilde{\mathsf{E}}^{(1,0)}\right)^T.
\label{symmetry_vector_operations}
\end{equation}
The coboundary operator acting on $(n-1)$-cochains discretely represents the exterior derivative acting on $(n-1)$-forms, i.e. the divergence operator. Therefore, \eqref{symmetry_vector_operations}, states that the discrete divergence operator (on the primal complex) is the transpose of the discrete gradient (on the dual complex). In vector calculus it says $\mathrm{div}=-\mathrm{grad}^T$. The minus sign is a consequence of the fact that orientation is not taken consistently into account in vector calculus. Furthermore, vector calculus does not reveal that these two operators act on spaces with a different type of orientation (inner and outer). The notion of inner- and outer-oriented complexes will lead to staggered grids.
\end{remark}

\begin{remark}
Consider in 2D the discrete gradient operator, applied to the dual cell-complex. Then the incidence matrix $\tilde{\mathsf{E}}^{(1,0)}$ relates the 0- and 1-cells as shown in \figref{fig:gradientdual2D}. Note that the discrete gradient does not give rise to 1-cells between the 0-cells on the boundary.
\begin{figure}[htb]
\centering
\includegraphics[width=0.2\textwidth]{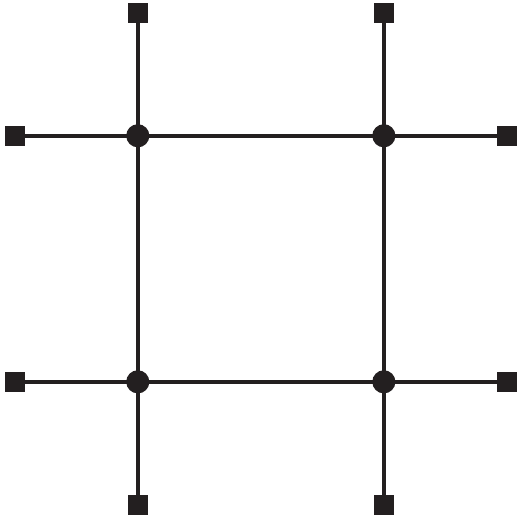}
\caption{A sub-cell-complex showing all 0- and 1-cells that are involved in the discrete gradient on the dual cell-complex in 2D. It consists of four internal 0-cells (\textbullet), eight boundary 0-cells (\tiny$\blacksquare$\normalsize) and twelve line segments.}
\label{fig:gradientdual2D}
\end{figure}
\end{remark}

\section{Mimetic operators}
\label{mimeticoperators}
The general philosophy in this section is to develop projections and coprojections, denoted by $\pi$ and $\pi^\star$, respectively, onto a finite dimensional space of differential forms, $\Lambda_h^k(\Omega;C_k)$, such that an operation $T$ at the finite dimensional space $\Lambda_h^k(\Omega;C_k)$, behaves in the same way as in the continuous level, $\Lambda^k(\Omega)$. By `behaves in the same way' we mean that either $\pi \circ T = T \circ \pi$ or $\pi^\star \circ T = T \circ \pi^\star$, i.e. when one of the following diagrams commute
\[
\begin{CD}
\Lambda^k(\Omega) @>T>> \Lambda^{l}(\Omega)\\
@V\pi VV @V\pi VV\\
\Lambda_h^k(\Omega;C_k) @>T>> \Lambda_h^{l}(\Omega;C_l)
\end{CD}
\quad\quad\quad\quad\quad\quad\quad
\begin{CD}
\Lambda^k(\Omega) @>T>> \Lambda^{l}(\Omega)\\
@V\pi^\star VV @V\pi^\star VV\\
\Lambda_h^k(\Omega;C_k) @>T>> \Lambda_h^{l}(\Omega;C_l).
\end{CD}
\]
In the previous sections, we denoted a manifold with either $\manifold{M}$ or $\manifold{N}$. In computational engineering it is customary to denote the manifold -- or computational domain -- by $\Omega$. Therefore, from now on we will refer to the manifold as $\Omega$.

\subsection{Reduction, reconstruction and projection operator}
A discretization involves a projection operator, $\pi$, from the complete space $\Lambda^k(\Omega)$ to a subspace $\Lambda^k_h(\Omega;C_k)\subset\Lambda^k(\Omega)$. In this subspace we are able to express differential forms in terms of $k$-cochains defined on $k$-chains, and corresponding $k$-cochain interpolation forms (often called basis functions or basis forms). There is a clear relation between the continuous representation using differential geometry and the discrete expressions in algebraic topology. There exists a reduction operator, $\reduction$, that integrates the $k$-forms on $k$-chains to get $k$-cochains. On the other hand, there is a reconstruction operator, $\reconstruction$, to reconstruct $k$-forms from $k$-cochains using appropriate basis forms. These mimetic operators were already introduced before in the work of Hyman and Scovel \cite{HymanScovel90} and in Bochev and Hyman \cite{bochev2006principles}. A composition of the two operators gives the projection operator $\pi=\reconstruction\circ\reduction$ as illustrated below.
\begin{diagram}
\Lambda^k(\Omega)&\rTo^{\quad\pi\quad}& \Lambda^k_h(\Omega;C_k)\\
\dTo^{\reduction} & \ruTo_{\reconstruction} & \\
C^k(D) & &
\end{diagram}
These three operators together set up the mimetic framework. The coprojections cannot be written directly in terms of reduction and reconstruction. However coprojections are expressed in terms of Hodge-$\star$ operators and projections. The latter are written in terms of reduction and reconstruction.
In this section we focus on the discretization of the differential forms, involving the reduction, reconstruction and projection operators.

\begin{definition}
\label{def:reduction}
The \emph{reduction operator} $\reduction:\kformspacedomain{k}{\Omega}\rightarrow \kcochainspacedomain{k}{D}$ is a homomorphism that maps differential forms to cochains. This linear map is also called the {\em De Rham map} and is defined by integration as
\begin{equation}
\duality{\reduction \kdifform{a}{k}}{\tau_{(k)}}:=\int_{\tau_{(k)}}\kdifform{a}{k},\quad \forall\tau_{(k)}\in C_k(D).
\label{reduction}
\end{equation}
Then for all $\kchain{c}{k}\in C_k(D)$, the reduction of the $k$-form, $\kdifform{a}{k}\in\Lambda^{k}(\Omega)$, to the $k$-cochain, $\kcochain{a}{k}\in C^k(D)$, is given by
\begin{equation}
\kcochain{a}{k}(\kchain{c}{k}):=\duality{\reduction\kdifform{a}{k}}{\kchain{c}{k}}\stackrel{\rm Def.~\ref{def::kchain}}{=}\sum_ic^i\duality{\reduction \kdifform{a}{k}}{\tau_{(k),i}}\stackrel{\eqref{reduction}}{=}\sum_ic^i\int_{\tau_{(k),i}}\kdifform{a}{k}=\int_{\kchain{c}{k}}\kdifform{a}{k}.
\end{equation}
\end{definition}
\noindent
It is the integration of a $k$-form over all $k$-cells in a $k$-chain that results in a $k$-cochain. Note that since $\tau_{(k),i}$ is compact and $C^k(D)$ is a finite dimensional free Abelian group, the reduction operator $\mathcal{R}$ is well-defined for $a^{(k)}=a(x_1,\hdots,x_n)\ederiv x^{i_1}\wedge\hdots\wedge\ederiv x^{i_k}$ with $a(x_1,\hdots,x_n)\in L_{\mathrm{loc}}(\Omega)$. 
\begin{definition}[\textbf{Integration over the domain $\Omega$}]
Let $\kdifform{a}{n}$ be any $n$-form, then the integral
\[
\int_\Omega\kdifform{a}{n}:=\duality{\reduction\kdifform{a}{n}}{\boldsymbol\omega_{(n)}}\;,
\]
where the chain $\boldsymbol\omega_{(n)}=\sum_i\tau_{(n),i}$ (so all $c^i=+1$) covers the entire computational domain $\Omega$.
\end{definition}
We can list the following properties of the reduction map. First of all, the reduction operator is a non-injective and surjective map.
\begin{definition}
\label{def:equivalence}
Consider two $k$-forms $\kdifform{a}{k}_1,\kdifform{a}{k}_2\in\kformspace{k}(\Omega)$ and $\kdifform{a}{k}_1\neq\kdifform{a}{k}_2$, they are in the same equivalence class if
\begin{equation}
\reduction\kdifform{a}{k}_1=\reduction\kdifform{a}{k}_2\quad\Leftrightarrow\quad\reduction\left(\kdifform{a}{k}_1-\kdifform{a}{k}_2\right)=0.
\end{equation}
\end{definition}
\noindent
Secondly, the reduction map commutes with respect to differentiation.
\begin{lemma}
The reduction map has a commuting property with respect to continuous and discrete differentiation,
\begin{equation}
\reduction\ederiv=\dederiv\reduction\quad\mathrm{on}\ \Lambda^k(\Omega).
\label{cdp1}
\end{equation}
This commutation can be illustrated as
\[\begin{CD}
\Lambda^k @>\ederiv>> \Lambda^{k+1}\\
@VV\reduction V @VV\reduction V  \\
C^k @>\delta>> C^{k+1}
\end{CD}\]
\begin{proof}
This property can be proven using Stokes' Theorem \eqref{eq::difGeom_stokes_theorem} and the duality property \eqref{eq::algTop_codifferential_dual},
\begin{equation}
\langle\reduction\ederiv \kdifform{a}{k},\kchain{c}{k}\rangle\stackrel{\eqref{reduction}}{=}\int_{\kchain{c}{k}}\ederiv \kdifform{a}{k}\stackrel{\eqref{eq::difGeom_stokes_theorem}}{=}\int_{\partial\kchain{c}{k}}\kdifform{a}{k}\stackrel{\eqref{reduction}}{=}\langle\reduction \kdifform{a}{k},\partial\kchain{c}{k}\rangle\stackrel{\eqref{eq::algTop_codifferential_dual}}{=}\langle\dederiv\reduction \kdifform{a}{k},\kchain{c}{k}\rangle.
\end{equation}
\end{proof}
\end{lemma}
\noindent
Thirdly, the reduction map commutes with the pullback, $\Phi^\star$, and the cochain map, $\Phi^\sharp$.
\begin{lemma}
Let $\Phi:\Omega_\mathcal{M}\rightarrow\Omega_\mathcal{N}$ be a continuous map between two manifolds, let $\Phi_\sharp:C_k(D_\mathcal{M})\rightarrow C_k(D_\mathcal{N})$ be the associated chain map of $k$-chains between two cell complexes and let $D_\mathcal{M}$ be a covering of the manifold $\Omega_\mathcal{M}$ and $D_\mathcal{N}$ the covering of $\Omega_\mathcal{N}$, according to \defref{cellcomplex}. Then the reduction map commutes with the continuous and discrete pullback,
\begin{equation}\label{commutation_reduction_map}
\reduction \pullback = \Phi^\sharp \reduction \quad \mathrm{on}\ \Lambda^k(\Omega_\mathcal{N}).
\end{equation}
This commutation is illustrated by
\[\begin{CD}
\Lambda^k(\Omega_\mathcal{N}) @>\pullback>> \Lambda^k(\Omega_\mathcal{M})\\
@VV\reduction V @VV\reduction V  \\
\kcochainspacedomain{k}{D_\mathcal{N}} @>\Phi^\sharp>> \kcochainspacedomain{k}{D_\mathcal{M}}
\end{CD}\]
where $\pullback$ is defined by Definition~\ref{def:pullback} and $\Phi^\sharp$ is defined in Definition~\ref{Induced_cochain_map}.
\begin{proof} For all $\kchain{c}{k} \in C_k(D_\mathcal{M})$ and for all $\kdifform{a}{k}\in\Lambda^k(\Omega_\mathcal{N})$ we have
\[
\langle \reduction \pullback \kdifform{a}{k},\kchain{c}{k}\rangle \stackrel{\eqref{reduction}}{=} \int_{\kchain{c}{k}} \pullback a^{(k)}\stackrel{\eqref{eq::diffGeom_pullback_integral_manifolds}}{=} \int_{\Phi_\sharp \kchain{c}{k}} \kdifform{a}{k} \stackrel{\eqref{reduction}}{=} \langle \reduction \kdifform{a}{k}, \Phi_\sharp \kchain{c}{k} \rangle \stackrel{\mathrm{Def.}\;\ref{Induced_cochain_map}}{=} \langle \Phi^\sharp \reduction \kdifform{a}{k},  \kchain{c}{k} \rangle.
\]
\end{proof}
\end{lemma}
The operator acting in opposite direction to the reduction operator is the reconstruction operator, $\reconstruction$, which maps $k$-cochains onto finite dimensional $k$-forms, and is defined as follows.
\begin{definition} \label{reconstructionoperator}
The \emph{reconstruction operator} $\reconstruction:\kcochainspacedomain{k}{D}\rightarrow\kformspace{k}_h(\Omega;C_k)$, also called the \emph{Whitney map}, is an isomorphism 
that maps cochains back to differential forms. The reconstructed differential forms belong to the space $\Lambda^k_h(\Omega;C_k)$, which is a proper subset of the complete $k$-form space $\Lambda^k(\Omega)$. While the reduction step is clearly defined in \defref{def:reduction}, in the choice of interpolation forms there exists some freedom. Although the choice of a reconstruction method allows for some freedom, $\reconstruction$ must satisfy the following properties:
\begin{itemize}
\item Reconstruction $\reconstruction$ must be the right inverse of $\reduction$, so it returns identity ({\it consistency property}),
\begin{equation}
\reduction\reconstruction=Id\quad\mathrm{on}\ C^k(D).
\label{consistency}
\end{equation}
\item Like $\reduction$, also the reconstruction operator $\reconstruction$ has to possess a commuting property with respect to differentiation. A properly chosen reconstruction operator $\reconstruction$ must satisfy a commuting property with respect to the exterior derivative and coboundary operator,
\begin{equation}
\ederiv\reconstruction=\mathcal{I}\dederiv\quad\mathrm{on}\ C^k(D).
\label{cdp2}
\end{equation}
This commutation can be illustrated as
\[\begin{CD}
\kformspaceh{k} @>\ederiv>> \kformspaceh{k+1}\\
@AA\reconstruction A @AA\reconstruction A  \\
\kcochainspace{k} @>\delta>> \kcochainspace{k+1}
\end{CD}\]
\item The reduction should commute with continuous and discrete pullback, i.e.
\begin{equation} \label{commutation_reconstruction_map}
\reconstruction \Phi^\sharp=\pullback\reconstruction\quad\mathrm{on}\ C^k(D_\mathcal{N}).
\end{equation}
This commutation relation can be represented by
\[\begin{CD}
\kformspacedomainh{k}{\Omega_\mathcal{N};C_k} @>\Phi^\star >> \kformspacedomainh{k}{\Omega_\mathcal{M};C_k}\\
@AA\reconstruction A @AA\reconstruction A  \\
\kcochainspacedomain{k}{D_\mathcal{N}} @>\Phi^\sharp>> \kcochainspacedomain{k}{D_\mathcal{M}}
\end{CD}\]
\item Let $\tilde{D}_i=*D_i$ be as defined in Definition~\ref{def:Boundary_of_cell_complexes}. For all $\kcochain{\tilde{c}}{n-k}\in C^{n-k}(\tilde{D}_i)$ there should exist a $\kcochain{c}{k}\in C^k(D_i)$, such that
\begin{equation}
\tilde{\reconstruction}\kcochain{\tilde{c}}{n-k}=\reconstruction\kcochain{c}{k},
\label{eq:tildeIc_n-k=Ick}
\end{equation}
where $\tilde{\reconstruction}$ reconstructs cochains on $\tilde{D}_i$ and $\reconstruction$ reconstructs cochains on the cell complex $D_i$. The same must hold in the opposite direction.

Both should also hold for the boundary part of the primal and dual cell complexes, i.e. for $D_b$ and $\tilde{D}_b$. Let $\tilde{D}_b=*D_b$ be as defined \defref{def:Boundary_of_cell_complexes}. For all $\kcochain{\tilde{c}}{n-1-k}\in C^{n-1-k}(\tilde{D}_b)$, there should exist a $\kcochain{c}{k}\in C^k(D_b)$, such that
\[
\tilde{\reconstruction}\kcochain{\tilde{c}}{n-1-k}=\reconstruction\kcochain{c}{k},
\]
where $\tilde{\reconstruction}$ reconstructs cochains on $\tilde{D}_b$ and $\reconstruction$ reconstructs cochains on $D_b$. Again the same must hold in opposite direction.
\end{itemize}
\end{definition}
\begin{remark}
Moreover, we want $\reconstruction$ to be an approximate left inverse of $\reduction$, so the result is close to identity ({\it approximation property}),
\begin{equation}
\reconstruction\reduction=Id+\mathcal{O}\left(h^p\right)\quad \mathrm{in}\ \Lambda^k(\Omega),
\label{approximation}
\end{equation}
where $\mathcal{O}(h^p)$ indicates a truncation error in terms of a measure of the grid size, $h$, and a polynomial order $p$.
\end{remark}

The composition in the last remark gives rise to the definition of a new operator, $\pi_h$, which is a projection operator.
\begin{definition}\label{def:projection}
Define the operator $\pi_h:\Lambda^k(\Omega)\rightarrow\Lambda^k_h(\Omega;C_k)$ as the composition $\reconstruction\circ\reduction$. It allows for an approximate continuous representation of a $k$-form $\kdifform{a}{k}$,
\begin{equation}
\kdifformh{a}{k}=\projection\kdifform{a}{k}=\reconstruction\reduction\kdifform{a}{k}, \quad\projection\kdifform{a}{k}\in\kformspace{k}_h(\Omega)\subset\kformspace{k}(\Omega).
\label{projection}
\end{equation}
where $\reconstruction\reduction \kdifform{a}{k}$ is expressed as a combination of $k$-cochains and interpolating $k$-forms.
\end{definition}
\begin{proposition}\label{prop:projection}
The operator $\pi_h:\Lambda^k(\Omega)\rightarrow\Lambda^k_h(\Omega;C_k)$ is a projection operator.
\begin{proof}
For $\projection$ to be a projection, it must be a homomorphism, which is true since both the reduction and the reconstruction operators are homomorphisms, therefore for all $a^{(k)},b^{(k)}\in\Lambda^k(\Omega)$ it holds
\begin{equation}
\projection(\kdifform{a}{k}+\kdifform{b}{k})=\projection \kdifform{a}{k}+\projection \kdifform{b}{k},
\label{linearityprojection}
\end{equation}
and the projection operator must be idempotent. Let $\kdifformh{a}{k}\in\Lambda^k_h(\Omega;C^k)$, then $\projection\kdifformh{a}{k}=\kdifformh{a}{k}$, and so $\projection\equiv Id$ on $\Lambda^k_h(\Omega;C_k)$. For all $\kdifformh{a}{k}\in\Lambda^k_h(\Omega;C_k)$, there exists $\kdifform{a}{k}\in\Lambda^k(\Omega)$ such that $\kdifformh{a}{k}=\projection\kdifform{a}{k}=\reconstruction\reduction\kdifform{a}{k}$, so
\begin{equation}
\reconstruction\reduction\kdifformh{a}{k}=\reconstruction\reduction(\reconstruction\reduction\kdifform{a}{k})\stackrel{\eqref{consistency}}{=}\reconstruction\reduction\kdifform{a}{k}=\kdifformh{a}{k}.
\end{equation}
\end{proof}
\end{proposition}
\begin{definition}[\textbf{Bounded projection}]\label{def:boundedprojection}
For a projection operator $\pi_h$ to be a useful operator, we require it to be a bounded projection operator, i.e. for $C_1<\infty$,
\begin{equation}
\Vert\pi_h\Vert_{\mathcal{L}(\Lambda^k,\Lambda^k_h)}:=\sup_{\kdifform{a}{k}\in\Lambda^k}\frac{\Vert\projection \kdifform{a}{k}\Vert_{\Lambda^k}}{\Vert \kdifform{a}{k}\Vert_{\Lambda^k}}\leq C_1,
\end{equation}
where $\| \cdot \|_{\Lambda^k}$ is some norm defined on $\Lambda^k(\Omega)$.
\end{definition}
\begin{corollary}
Alternatively, one can write for the bounded projection operator
\begin{equation}
\Vert\projection \kdifform{a}{k}\Vert_{\Lambda^k_h}\leq C_1\Vert \kdifform{a}{k} \Vert_{\Lambda^k}.
\end{equation}
From the triangle inequality applied to $\Vert \kdifform{a}{k}\Vert_{\Lambda^k}$ it follows that there exists $C_2=C_1+1<\infty$, such that
\begin{equation}
\Vert(Id-\projection)\kdifform{a}{k}\Vert_{\Lambda^k_h}\leq C_2\Vert \kdifform{a}{k}\Vert_{\Lambda^k}.
\end{equation}
\end{corollary}

\begin{proposition}\label{prop:projdecomposition}
Using the projection operator, the differential form space $\Lambda^k(\Omega)$ can be decomposed into a projected space and its complement, so $\Lambda^k=\Lambda^k_h\oplus\Lambda^{k,c}_h$, and so any $k$-form can be uniquely decomposed as
\begin{equation}
\kdifform{a}{k}=\kdifformh{a}{k}+a^{(k),c}_h=\projection\kdifform{a}{k}+(Id-\projection)\kdifform{a}{k},\quad\forall\kdifform{a}{k}\in\Lambda^k(\Omega).
\end{equation}
Note that $(Id-\projection)$ is also a projection, but now onto the unresolved part. As a consequence of the direct sum decomposition and \propref{prop:projection}, $\projection\circ(Id-\projection)=(Id-\projection)\circ\projection=0$.
\end{proposition}
The projection is not an orthogonal projection and therefore the complement space is not orthogonal to the projected space, so $\Lambda^{k,c}_h\neq\Lambda^{k,\perp}_h$. This will be demonstrated in \exampleref{galerkinprojection} in the next section. As a consequence the projection, $\pi_h$, is not self-adjoint,
\begin{equation}
\inner{\projection\kdifform{a}{k}}{\kdifform{b}{k}}\neq\inner{\kdifform{a}{k}}{\projection\kdifform{b}{k}},\quad\forall\kdifform{a}{k},\kdifform{b}{k}\in\Lambda^k(\Omega).
\end{equation}

Both $\kdifform{a}{k}\in\Lambda^k(\Omega)$ and its projected part $\pi_h\kdifform{a}{k}\in\Lambda^k_h(\Omega;C_k)$ are in the same equivalence class, as defined in Definition \ref{def:equivalence}, so
\begin{equation}
\label{reductionprojection}
\reduction\projection\kdifform{a}{k}=\reduction\kdifform{a}{k}\quad \mathrm{and}\quad \reduction(Id-\projection)\kdifform{a}{k}=0,
\end{equation}
and for the special case of integration of a volume form over the whole domain, this integral preserving property of the projection gives
\[
\int_\Omega\projection\kdifform{a}{n}=\int_\Omega\kdifform{a}{n}.
\]
As for the reduction operator, also the projection operator is non-injective and surjective. Then because of Definitions \ref{def:equivalence} and \ref{def:projection}, there also exists an equivalence class for the projection operator. The projection of $\kdifform{a}{k}$ is not unique, but depends among others on the underlying cell complex, as is illustrated in the following example:
\begin{example}
Consider a uniform and a Gauss-Lobatto grid, then $\pi_{h}^{\rm uni}\kdifform{a}{k},\pi_{h}^{\rm gl}\kdifform{a}{k}\in\Lambda^k_h([-1,1])$, but $\pi_{h}^{\rm uni}\kdifform{a}{k}\neq\pi_{h}^{\rm gl}\kdifform{a}{k}$, see Figure \ref{figure:projection} for an example of a 0-form and a 1-form. The difference between the two projections reduces with grid refinement, since
\begin{align*}
|\projection^{\rm uni}\kdifform{a}{k}-\projection^{\rm gl}\kdifform{a}{k}|
&=|(I-\projection^{\rm gl})\kdifform{a}{k}-(I-\projection^{\rm uni})\kdifform{a}{k}|\\
&\leq|(I-\projection^{\rm gl})\kdifform{a}{k}|+|(I-\projection^{\rm uni})\kdifform{a}{k}|\stackrel{\eqref{approximation}}{=}\mathcal{O}(h^p).
\end{align*}
Moreover note that, according to \propref{prop:projection}, it can be observed that $\pi_h^{\rm uni}\pi_h^{\rm gl}\kdifform{a}{k}=\pi_h^{\rm gl}\kdifform{a}{k}$, $\pi_h^{\rm gl}\pi_h^{\rm uni}\kdifform{a}{k}=\pi_h^{\rm uni}\kdifform{a}{k}$ and so $\pi_h^{\rm uni}\pi_h^{\rm gl}\kdifform{a}{k}\neq\pi_h^{\rm gl}\pi_h^{\rm uni}\kdifform{a}{k}$.
\begin{figure}[htb]
\centering
\begin{tabular}{cc}
\subfigure[zero-form]{\includegraphics[width=.45\textwidth]{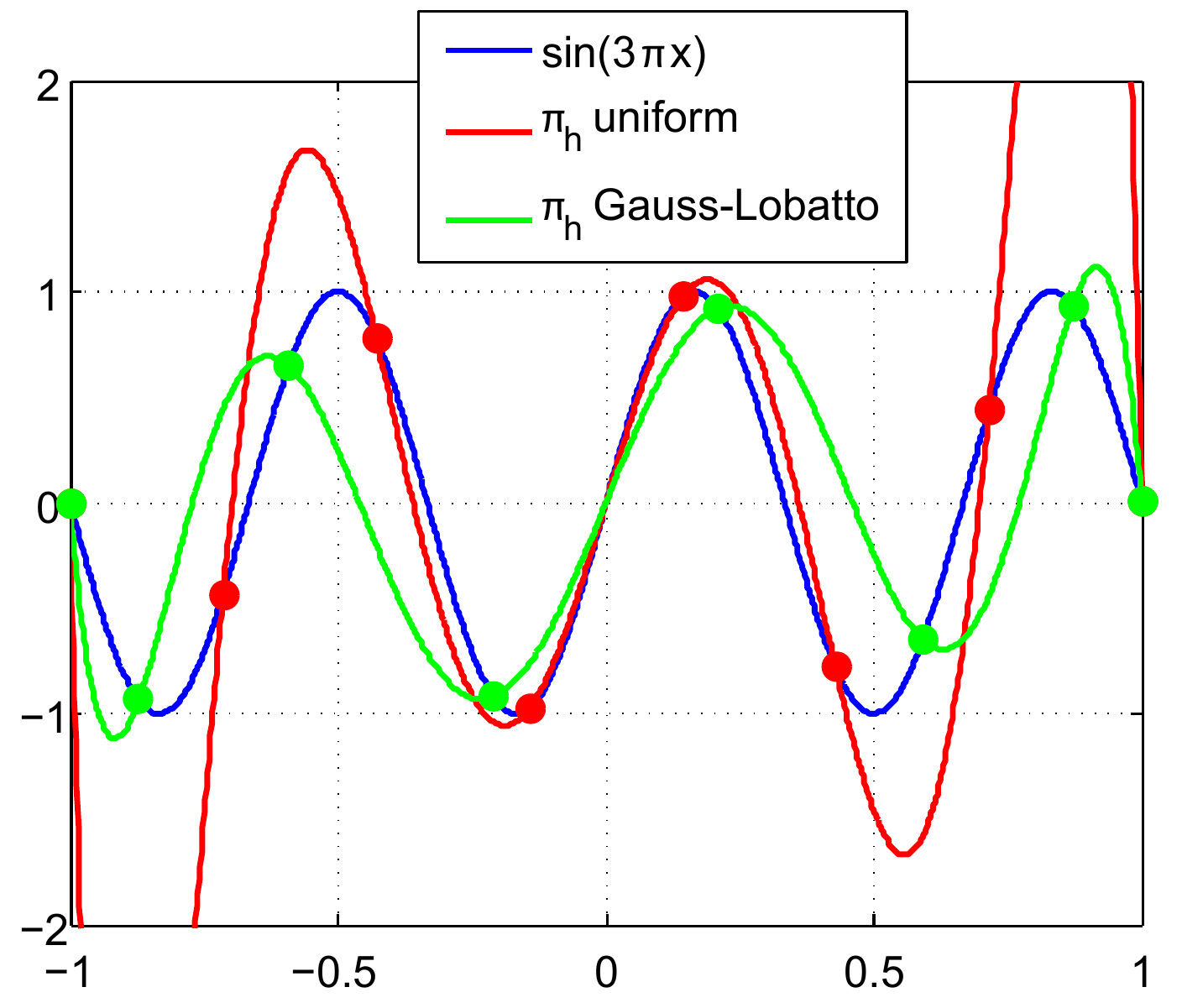}} &
\subfigure[one-form]{\includegraphics[width=.45\textwidth]{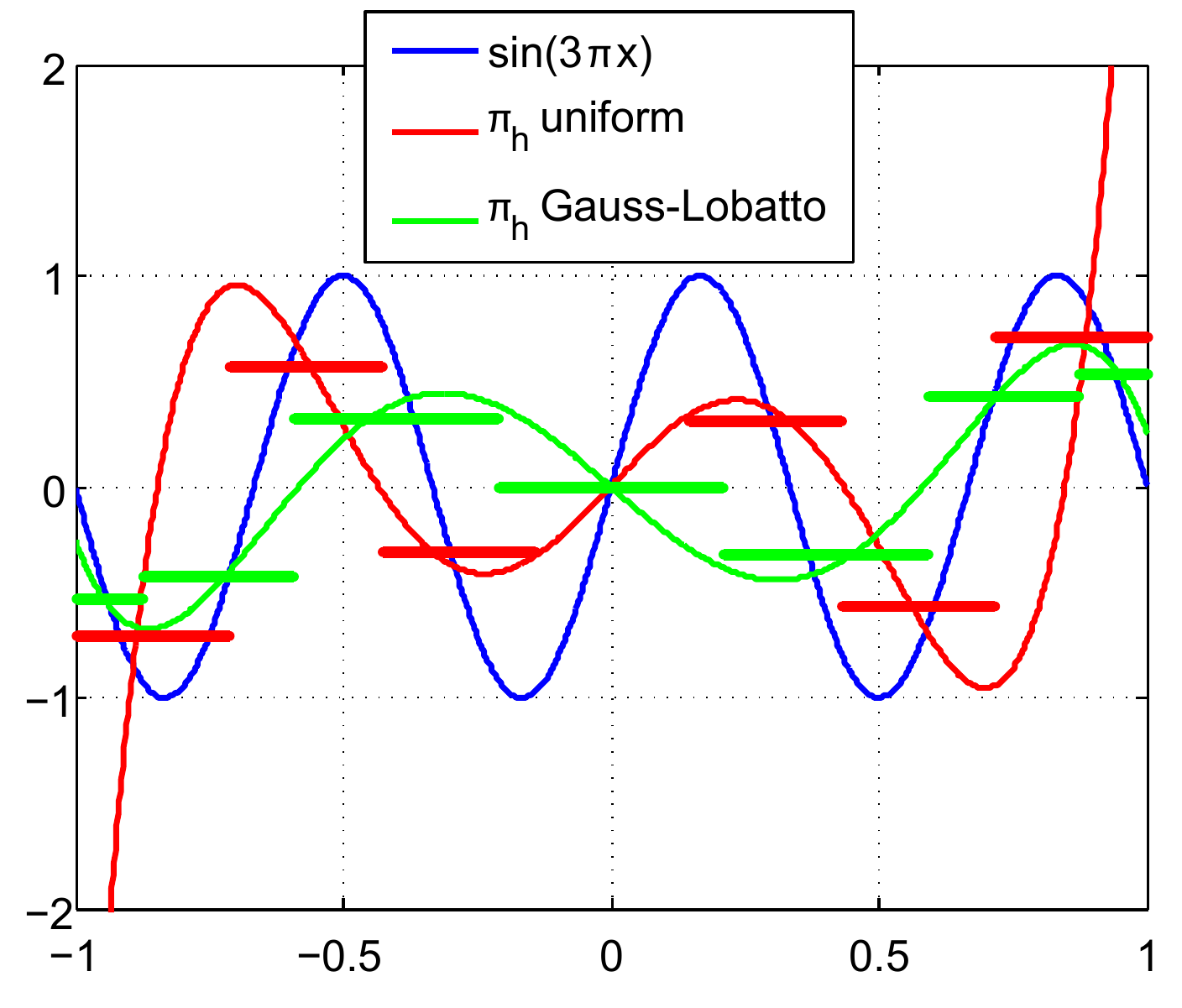}}
\end{tabular}
\caption{These figures illustrate the non-uniqueness of the projection operator for zero-forms (a) and one-forms (b). In (a) the dots indicate the 0-cochains, in (b) the line segments are the 1-cochains. The resolutions are intentionally kept low to show the differences.}
\label{figure:projection}
\end{figure}
\end{example}

In Section~\ref{sec:AlgebraicTopology} staggered cell complexes were introduced. The two complexes are dual with respect to each other as defined in \defref{def:dualcellcomplex} and illustrated in Figures \ref{fig:dualcellcomplex1D} and \ref{fig:dualcellcomplex2D}. This allows us to define two projection operators, one projecting onto the subspace $\Lambda^k(\Omega;C_k)$ and one projecting onto the subspace $\Lambda^k(\Omega;\tilde{C}_k)$, where $C_k=C_k(D)$ and $\tilde{C}_k=C_k(\tilde{D})$.
\begin{definition}[\textbf{Canonical projection operators}]\label{def:canonical_projection_operators}
Given the outer-oriented cell complex $D$ and its dual inner-oriented cell complex $\tilde{D}$, we define two projections:
\begin{itemize}
\item $\projection=\reconstruction\reduction$, where the reduction, $\reduction$, and the reconstruction, $\reconstruction$, are performed on the interior part of the primal cell complex,  i.e. $D_i$. The corresponding subspace $\Lambda^k_h(\Omega;C_k)$ is the space of finite dimensional $k$-forms, with $k$-cochains associated with the \emph{outer}-oriented cells in $D_i$, such that
\[
\Lambda^k_h(\Omega;C_k)\define\projection\Lambda^k(\Omega)=\reconstruction\reduction\Lambda^k(\Omega).
\]
\item $\dualprojection=\tilde{\reconstruction}\tilde{\reduction}$, where the reduction, $\tilde{\reduction}$, and the reconstruction, $\tilde{\reconstruction}$, are performed on the interior part of the dual cell complex, i.e. $\tilde{D}_i$. The space $\Lambda^k_h(\Omega;\tilde{C}_k)$ is the space of discrete $k$-forms, with $k$-cochains associated with the \emph{inner}-oriented cells in $\tilde{D}_i$, such that
\[
\Lambda^k_h(\Omega;\tilde{C}_k)\define\tilde{\pi}_h\Lambda^k(\Omega)=\tilde{\reconstruction}\tilde{\reduction}\Lambda^k(\Omega).
\]
\end{itemize}
\end{definition}
\noindent
To every projection we can define a corresponding coprojection.
\begin{definition}[\textbf{Canonical coprojection operators}]\label{def:canonical_coprojection_operators}
Given projection $\pi_h$ using the outer-oriented cell complex $D_i$ and projection $\tilde{\pi}_h$ using the inner-oriented cells in $\tilde{D}_i$, we define two coprojections:
\begin{itemize}
\item $\coprojection$ is defined as $\coprojection = (-1)^{k(n-k)} \star \dualprojection \star$. This coprojection is defined in terms of a projection of $(n-k)$-forms, $\star \kdifform{a}{k}$, on the inner-oriented $(n-k)$-chains in $\tilde{D}_i$ and then represented in terms of the $k$-form basis functions on the outer-oriented cell complex $D_i$. The corresponding subspace is $\pi^\star_h\Lambda^k(\Omega)$.
\item $\dualcoprojection$ is defined as $\dualcoprojection = (-1)^{k(n-k)} \star \projection \star$. This coprojection is defined in terms of a projection of $(n-k)$-forms, $\star \kdifform{a}{k}$, on the outer-oriented $(n-k)$-chains in $D_i$ and then represented in terms of the $k$-form basis functions on the interior part of the inner-oriented dual cell complex, i.e. $\tilde{D}_i$. The corresponding subspace is $\tilde{\pi}^\star_h\Lambda^k(\Omega)$.
\end{itemize}
\end{definition}
\begin{remark}
A coprojection is also a projection according to \propref{prop:projection}.
\end{remark}
\begin{remark}
Although projections and coprojections in Definitions \ref{def:canonical_projection_operators} and \ref{def:canonical_coprojection_operators} were defined with the interior parts of cell complexes $D$ and $\tilde{D}$. The same projections and coprojections can be defined for the boundary parts of the cell complexes $D$ and $\tilde{D}$, i.e. $D_b$ and $\tilde{D}_b$.
\end{remark}

The projection and coprojection operators possess commutation relations with the operators to be defined in the next subsection. The projected differential forms differ, because the forms are reduced on different chains as illustrated in Example~\ref{figure:projection}. See also \figref{fig:projections} in the next section.

\begin{proposition}[\textbf{Equivalence spaces of differentials forms $\projection \Lambda^k$ and $\coprojection \Lambda^k$}]\label{prop:pispace=pi_starspace}
Define the spaces
\[ \projection \Lambda^k := \left \{ \kdifformh{a}{k}\,| \, \exists \kdifform{a}{k} \in \Lambda^k(\Omega) \mbox{ s.t. } \kdifformh{a}{k} = \projection \kdifform{a}{k}\,\right \} \;,\]
and
\[ \coprojection \Lambda^k := \left \{ \kdifformh{a}{k}\,| \, \exists \kdifform{a}{k} \in \Lambda^k(\Omega) \mbox{ s.t. } \kdifformh{a}{k} = \coprojection \kdifform{a}{k}\,\right \} \;,\]
then $\Lambda^k_h=\projection \Lambda^k \equiv \coprojection \Lambda^k$, but in general $\projection \kdifform{a}{k} \neq \coprojection \kdifform{a}{k}$. 
\begin{proof}
The reduction of $\star \kdifform{a}{k}$ on $(n-k)$-chains in $\tilde{D}_i$ yields a $(n-k)$-cochain in $\tilde{D}_i$ and according to (\ref{eq:tildeIc_n-k=Ick}) in Definition~\ref{reconstructionoperator}, there exists a $k$-cochain on the cell complex $D$ such that $\tilde{\reconstruction}\kcochain{\tilde{c}}{n-k} = \reconstruction \kcochain{c}{k}$, therefore $\coprojection \Lambda^k \subset \projection \Lambda^k$. By a similar argument one can show that $\projection \Lambda^k \subset \coprojection \Lambda^k$, so we have $\projection \Lambda^k \equiv \coprojection \Lambda^k$. That $\projection \kdifform{a}{k} \neq \coprojection \kdifform{a}{k}$ will be shown in Figure~\ref{fig:projections} in Section~\ref{sec:MSEM}.
\end{proof}
\end{proposition}

\begin{proposition}[\textbf{Equivalence projections on subspaces}]\label{prop:equiv_projections}
If $\kdifform{a}{k} \in \Lambda_h^k(\Omega;C_k)$ then $\coprojection \kdifform{a}{k} = \projection \kdifform{a}{k}$ and if $\kdifform{a}{k} \in \Lambda_h^k(\Omega;\tilde{C}_k)$ then $\dualcoprojection \kdifform{a}{k} = \dualprojection \kdifform{a}{k}$.
\begin{proof}
\[ \mbox{If } \kdifformh{a}{k} \in \Lambda_h^k\;\; \stackrel{\mbox{\scriptsize{Prop.~\ref{prop:pispace=pi_starspace}}}}{\Longrightarrow}\;\; \left \{ \begin{array}{ll}
\kdifformh{a}{k} = \projection \kdifformh{a}{k}\quad \quad & \mbox{because } \projection = Id \mbox{ on } \Lambda_h^k = \projection \Lambda^k \\
 & \\
\kdifformh{a}{k} = \coprojection \kdifformh{a}{k}\quad \quad & \mbox{because } \coprojection = Id \mbox{ on } \Lambda_h^k = \coprojection \Lambda^k
\end{array}
\right . \]
Therefore $\kdifformh{a}{k} = \projection \kdifformh{a}{k} = \coprojection \kdifformh{a}{k}$. The proof on the dual complex is the same.
\end{proof}
\end{proposition}
\begin{remark}
Although $\coprojection \kdifform{a}{k} = \projection \kdifform{a}{k}$ holds in the finite dimensional subspace $\Lambda_h^k(\Omega;C_k)$, it does not holds on the entire space $\Lambda^k(\Omega)$. The difference between these two projections could give rise to {\em natural} and {\em derived} operators, \cite{bochev2006principles}. However, from our point of view projections $\projection$ and coprojections $\coprojection$ are essentially different operators.
\end{remark}

\begin{definition}\label{def:general_pi}
Whenever we refer to the projection $\pi$ it can be any of the projections from Definition~\ref{def:canonical_projection_operators}.
\end{definition}


Although $\pi_h=Id$ on $\Lambda^k_h$, (Proposition~\ref{prop:projection}), it will be shown that it can be a very useful operation when changing the expression of a finite dimensional $k$-form. A finite dimensional $k$-form $a^{(k)}_h$ is usually expressed in terms of a $k$-cochain $\kcochain{a}{k}\in C^k(D)$ and interpolating $k$-forms. If this is not the case, but the $k$-form $a_h^{(k)}$ is expressed in terms of an $l$-cochain corresponding to a chain in $\hat{C}_l:=C_l(\hat{D})$, where $\hat{D}$ is either the primal cell complex $D$ or the dual cell complex $\tilde{D}$, then a projection is used to express the finite dimensional $k$-form in terms of its corresponding $k$-cochains. For this special case we introduce a separate projection operator, $\pi_M$.
\begin{definition}
\label{def:piM}
Define a special, bijective projection $\pi_M:\Lambda^k_h(\Omega;\hat{C}_{l})\rightarrow\Lambda^k_h(\Omega;C_k)$, such that
\begin{equation}
\pi_M=Id\quad on\ \Lambda^k_h(\Omega;\hat{C}_l).
\end{equation}
So $\pi_Ma_h=a_h$, but its expression changes in terms of cochains, from an $l$-cochain in $C^l(\hat{D})$ to a $k$-cochain in $C^k(D)$, and its expression changes with respect to the basis functions. In terms of reduction and reconstruction this is
\begin{equation}
a_h=\hat{\reconstruction}\hat{\reduction}a=\reconstruction\reduction(\hat{\reconstruction}\hat{\reduction}a)=\pi_Ma_h.
\end{equation}
\end{definition}
\noindent
The seemingly redundant projection $\pi_M$ will reappear almost everywhere, but we will not explicitly mention this in this section. It will be more instructive to show its action with concrete reconstruction operators in Section~\ref{sec:MSEM}.

\subsection{Discrete operators}\label{discreteoperators}

In this section we discuss operations of the operators discussed in \secref{sec:DifferentialGeometry}, restricted to the set of finite dimensional subspaces $\Lambda^k_h$, i.e. $\ederiv$, $\star$, $\coderiv$, $\pullback$, $\wedge_h$ and $(\cdot,\cdot)_h$.


\subsubsection{The exterior derivative}
First consider the exterior derivative. With \eqref{cdp1} and \eqref{cdp2} a commuting property of the projection with respect to the exterior derivative can be shown.
\begin{lemma}\label{Lem:projectionextder}
The projections $\projection= \reconstruction \reduction$ and $\dualprojection = \tilde{\reconstruction}\tilde{\reduction}$ commute with the exterior derivative,
\begin{equation}
\label{projectionextder}
\ederiv\projection=\projection\ederiv\quad\mathrm{and}\quad \ederiv\dualprojection=\dualprojection\ederiv \quad \mathrm{on}\ \Lambda^k(\Omega).
\end{equation}
This can be illustrated as
\[
\begin{CD}
\Lambda^k @>\ederiv>> \Lambda^{k+1}\\
@VV\projection V @VV\projection V\\
\Lambda_h^k @>\ederiv>> \Lambda_h^{k+1}.
\end{CD} \quad\quad \quad 
\begin{CD}
\Lambda^k @>\ederiv>> \Lambda^{k+1}\\
@VV\dualprojection V @VV\dualprojection V\\
\tilde{\Lambda}_h^k @>\ederiv>> \tilde{\Lambda}_h^{k+1}.
\end{CD}
\]
\begin{proof}
Express the projection in terms of the reduction and reconstruction operator, then for all $a\in\Lambda^k(\Omega)$,
\[
\ederiv\projection\difform{a}=\ederiv\reconstruction\reduction\difform{a}\stackrel{\eqref{cdp2}}{=}\reconstruction\delta\reduction\difform{a}\stackrel{\eqref{cdp1}}{=}\reconstruction\reduction\ederiv\difform{a}=\projection\ederiv\difform{a}.
\]
The proof for $\dualprojection = \tilde{\reconstruction}\tilde{\reduction}$ is the same.
\end{proof}
\end{lemma}



\subsubsection{The Hodge-$\star$ operator}
The projection and coprojection possess a commuting diagram property with the Hodge-$\star$.
\begin{proposition}\label{prop:commutation_Hodge_projections}
We have for all $\kdifform{a}{k} \in \Lambda^k(\Omega)$ the following commutation relations,
\[ \coprojection \star = \star \dualprojection \;\;\; \mbox{and}\;\;\; \star \coprojection  = \dualprojection \star. \]
\end{proposition}
\begin{proof} From Definition~\ref{def:canonical_projection_operators} we have
\[ \coprojection = (-1)^{k(n-k)} \star \dualprojection \star \quad \Longleftrightarrow \quad  \coprojection\star =  \star\dualprojection \;\;\; \mbox{and}\;\;\; \star \coprojection  = \dualprojection \star  .\]
So the Hodge-$\star $ operator at the continuous level commutes with the Hodge-$\star$ in the finite dimensional setting with respect to the projections $\dualprojection$ and $\coprojection$.
\[\begin{CD}
\kformspace{k} @>\star>> \kformspace{n-k}\\
@VV\dualprojection V @VV\coprojection V  \\
\tilde{\Lambda}_h^k @>\star>> \kformspaceh{n-k}
\end{CD}\;\;\;\quad \quad \quad \quad
\begin{CD}
\kformspace{k} @>\star>> \kformspace{n-k}\\
@VV\coprojection V @VV\dualprojection V  \\
\kformspaceh{k} @>\star>> \tilde{\Lambda}_h^{n-k}
\end{CD}\]
\end{proof}
\begin{remark}
Similar commutation relations can be set up on the dual cell complex and this gives: for all $\kdifform{a}{k} \in \Lambda^k(\Omega)$
\[ \dualcoprojection \star = \star \projection \;\;\; \mbox{and}\;\;\; \star \dualcoprojection  = \projection \star \;,\]
with the associated commutating diagrams.
\end{remark}
\begin{corollary}\label{rem:projectionhodge}
Proposition~\ref{prop:commutation_Hodge_projections} combined with Proposition~\ref{prop:equiv_projections} gives that for $\kdifform{a}{k} \in \Lambda_h^k$
\[ \dualprojection \star \stackrel{\rm Prop.~\ref{prop:commutation_Hodge_projections}}{=} \star \coprojection \stackrel{\rm Prop.~\ref{prop:equiv_projections}}{=} \star \projection \;.\]
\end{corollary}
In Definition \ref{def::diffGeom_hodge} the Hodge-$\star$ was defined as a mapping from $k$-forms to $(n-k)$-forms. As a consequence, not only the dimension of the chain over which they are integrated changes, but also the orientation changes from inner to outer orientation or vice versa. This corresponds to the vertical relations in \figref{fig:diffGeom_orientation}. On finite dimensional subspaces, the Hodge-$\star$ operator is a mapping from $\Lambda^k_h(\Omega;C_k)$ to $\Lambda^{n-k}_h(\Omega;\tilde{C}_{n-k})$, as defined in \defref{def:canonical_projection_operators}.

With the Hodge-$\star$ and exterior derivative, the finite dimensional double De Rham complex can be set up similar to \eqref{eq::diffGeom_deRham_Complex} as
\[
\begin{matrix}
\mathbb{R} \longrightarrow&\Lambda^0_h(\Omega;C_0)
&\stackrel{\ederiv}{\longrightarrow}& \Lambda^1_h(\Omega;C_1)
&\stackrel{\ederiv}{\longrightarrow}& \hdots \;
&\stackrel{\ederiv}{\longrightarrow}\; &\Lambda^n_h(\Omega;C_n) \;
&\stackrel{\ederiv}{\longrightarrow}\; &0  \\
&\star\updownarrow & & \star\updownarrow &&  &
&\star\updownarrow & &   \\
0 \stackrel{\ederiv}{\longleftarrow}&\Lambda^n_h(\Omega;\tilde{C}_{n})
&\stackrel{\ederiv}{\longleftarrow}& \Lambda^{n-1}_h(\Omega;\tilde{C}_{n-1})
&\stackrel{\ederiv}{\longleftarrow}& \hdots \;
&\stackrel{\ederiv}{\longleftarrow}\; &\Lambda^0_h(\Omega;\tilde{C}_{0}) \;
&\stackrel{}{\longleftarrow}\; &\mathbb{R}.
\end{matrix}
\]

\subsubsection{The codifferential}
With \defref{def:canonical_projection_operators} a commuting property of the coprojection with respect to the codifferential can be shown.
\begin{proposition}[\textbf{Commutation coprojections with codifferential}]\label{Commutation_coproj_codiff} For all $\kdifform{a}{k} \in \Lambda^k(\Omega)$ we have
\[ \coprojection \coderiv = \coderiv \coprojection \quad \mbox{and} \quad \dualcoprojection \coderiv = \coderiv \dualcoprojection\;.\]
\begin{proof}
\begin{gather*}
\coprojection \coderiv \stackrel{{\rm Def.}~\ref{def:canonical_coprojection_operators},\ {\rm Prop.}~\ref{prop::diffgeom_codifferential_adjoint}}{=} (-1)^{k(n-k)}\star \dualprojection \star (-1)^{n(k+1)+1} \star \ederiv \star \stackrel{\eqref{eq::diffGeom_double_hodge}}{=}  (-1)^{n+k+1} \star \dualprojection \ederiv \star \quad\quad\quad\quad\quad  \\
\quad\quad\quad\quad\stackrel{\rm Lem.~\ref{Lem:projectionextder}}{=} (-1)^{n+k+1} \star  \ederiv \dualprojection \star \stackrel{\eqref{eq::diffGeom_double_hodge}}{=} (-1)^{n(k+1)+1} \star \ederiv \star (-1)^{k(n-k)} \star \dualprojection \star \stackrel{{\rm Def.}~\ref{def:canonical_coprojection_operators},\ {\rm Prop.}~\ref{prop::diffgeom_codifferential_adjoint}}{=} \coderiv \coprojection \;.
\end{gather*}
The proof for $\dualcoprojection$ is the same.
\end{proof}
\end{proposition}
This means that the codifferential is exact with respect to the coprojection, i.e. the following diagram commutes
\[\begin{CD}
\kformspace{k+1} @>\coderiv>> \kformspace{k}\\
@VV\coprojection V @VV\coprojection V  \\
\kformspaceh{k+1} @>\coderiv>> \kformspaceh{k}
\end{CD} \quad \quad \quad \quad
\begin{CD}
\kformspace{k+1} @>\coderiv>> \kformspace{k}\\
@VV\dualcoprojection V @VV\dualcoprojection V  \\
\kformspaceh{k+1} @>\coderiv>> \kformspaceh{k}
\end{CD}
\]

\begin{corollary}
Combining Proposition~\ref{Commutation_coproj_codiff} and Proposition~\ref{prop:equiv_projections} shows that for $\kdifformh{a}{k} \in \Lambda_h^k(\Omega)$ we have
\[ \projection \coderiv \stackrel{\rm Prop.~\ref{prop:equiv_projections}}{=} \coprojection \coderiv \stackrel{\rm Prop.~\ref{Commutation_coproj_codiff}}{=} \coderiv \coprojection \stackrel{\rm Prop.~\ref{prop:equiv_projections}}{=} \coderiv \projection \;.\]
\end{corollary}

\subsubsection{The pullback} The projection commutes with the pullback as follows.
\begin{lemma}
Let $\pi$ be either $\pi_h$ or $\tilde{\pi}_h$ as in Definition~\ref{def:general_pi}. For all $\kdifform{a}{k}\in\Lambda^k(\Omega_\manifold{N})$ and for $\Phi:\Omega_\manifold{M}\rightarrow\Omega_\manifold{N}$, there exists a commuting property between the projection operator $\pi$ and the pullback $\pullback$, such that
\begin{equation}
\label{pullbackprojection}
\pullback\pi=\pi\pullback\quad\mathrm{on}\ \Lambda^k(\Omega_\manifold{N}).
\end{equation}
This commutation can be illustrated as
\[
\begin{CD}
\Lambda^k(\Omega_\manifold{M}) @>\pullback>> \Lambda^k(\Omega_\manifold{N})\\
@VV\pi V @VV\pi V\\
\Lambda^k_h(\Omega_\manifold{M},\kchain{c}{k}) @>\pullback>> \Lambda^k_h(\Omega_\manifold{N},\kchain{c}{k})
\end{CD}
\]
\begin{proof} The proof is based on the commuting properties of the discrete pullback, $\Phi^\sharp$:
\[ \reconstruction \reduction \pullback a^{(k)} \stackrel{\eqref{commutation_reduction_map}}{=} \reconstruction \Phi^\sharp \reduction a^{(k)} \stackrel{\eqref{commutation_reconstruction_map}}{=} \pullback \reconstruction \reduction a^{(k)}. \]
An alternative proof can be given, based on \eqref{eq::diffGeom_pullback_integral_manifolds} and \eqref{reductionprojection}. Integrate over a $k$-cell $\tau_{(k)}$, then
\[
\int_{\tau_{(k)}}\pullback\pi\kdifform{a}{k}\stackrel{\eqref{eq::diffGeom_pullback_integral_manifolds}}{=}\int_{\Phi(\tau_{(k)})}\pi\kdifform{a}{k}\stackrel{\eqref{reductionprojection}}{=}\int_{\Phi(\tau_{(k)})}\kdifform{a}{k}\stackrel{\eqref{eq::diffGeom_pullback_integral_manifolds}}{=}\int_{\tau_{(k)}}\pullback\kdifform{a}{k}\stackrel{\eqref{reductionprojection}}{=}\int_{\tau_{(k)}}\pi\pullback\kdifform{a}{k}.
\]
\end{proof}
\end{lemma}
\noindent
This commutation does not hold for the coprojections.
Combining the commutation relations \eqref{projectionextder} and \eqref{pullbackprojection} gives
\[
\pi\pullback\ederiv=\pullback\pi\ederiv=\pullback\ederiv\pi=\ederiv\pullback\pi=\ederiv\pi\pullback=\pi\ederiv\pullback.
\]

\subsubsection{The wedge product}
The next operator to be considered in the subspaces $\Lambda^k_h\subset\Lambda^k$, $0\leq k\leq n$, is the wedge product. A product of a $k$- and $l$-form from subspaces $\Lambda^k_h$ and $\Lambda^l_h$ gives a $(k+l)$-form that is not in the subspace $\Lambda^{k+l}_h$. It therefore requires again a projection step.
\begin{definition}\label{discretewedge}
A discrete wedge product is introduced such that $\wedge_h:\Lambda^k_h\times\Lambda^l_h\rightarrow\Lambda^{k+l}_h$, given by
\begin{equation}
\kdifformh{a}{k}\wedge_h \kdifformh{b}{l}\define\pi\left(\kdifformh{a}{k}\wedge\kdifformh{b}{l}\right),
\end{equation}
where $\pi$ is either $\pi_h$ or $\tilde{\pi}_h$, as defined in Definition~\ref{def:general_pi}.
\end{definition}
As a consequence the discrete wedge product, $\wedge_h$, approximates the wedge product, $\wedge$, because
\[
\kdifformh{a}{k}\wedge \kdifformh{b}{l}-\kdifformh{a}{k}\wedge_h \kdifformh{b}{l}=(Id-\pi)(\kdifformh{a}{k}\wedge \kdifformh{b}{l})\stackrel{\eqref{approximation}}{=}\mathcal{O}(h^p).
\]
Let us verify that the discrete wedge product satisfies the same properties as the original wedge product. Let $\kdifformh{a}{k},\kdifformh{c}{k}\in\Lambda^k_h,\ \kdifformh{b}{l}\in\Lambda^l_h$, with $2k+l\leq n$. Using the linearity of the projection it is straightforward to show that
\[
\left(\kdifformh{a}{k}+\kdifformh{c}{k}\right)\wedge_h \kdifformh{b}{l}=\kdifformh{a}{k}\wedge_h \kdifformh{b}{l}+\kdifformh{c}{k}\wedge_h \kdifformh{b}{l}.
\]
Also the skew-symmetry follows from the linearity of the projection,
\[
\kdifformh{a}{k}\wedge_h \kdifformh{b}{l}=\pi\left(\kdifformh{a}{k}\wedge \kdifformh{b}{l}\right)=(-1)^{kl}\pi\left(\kdifformh{b}{l}\wedge \kdifformh{a}{k}\right)=(-1)^{kl}\kdifformh{b}{l}\wedge_h \kdifformh{a}{k}.
\]
The third property of the wedge product is associativity \eqref{wedgeassociativity}. As stated already in \cite{desbrun2005discrete}, the associativity property is in general not satisfied. Now let $a\in\Lambda^k_h,\ b\in\Lambda^l_h,\ c\in\Lambda^m_h$, with $k+l+m\leq n$, then\footnote{Both sub- and superscripts are intentionally suppressed for readability.}
\begin{align*}
&(a\wedge_h b)\wedge_h c-a\wedge_h(b\wedge_h c)\\
&\ \ =\pi\left(\pi(a\wedge b)\wedge c\right)-\pi\left(a\wedge\pi(b\wedge c)\right)\\
&\ \ =\pi\left[(a\wedge b)\wedge c-(I-\pi)(a\wedge b)\wedge c - a\wedge (b\wedge c) + a\wedge(I-\pi)(b\wedge c)\right]\\
&\ \ =\pi\Big[ a\wedge\underbrace{(I-\pi)(b\wedge c)}_{\mathcal{O}(h^p)} - \underbrace{(I-\pi)(a\wedge b)}_{\mathcal{O}(h^p)}\wedge c \Big]=\mathcal{O}(h^p).
\end{align*}
When considering bilinear products, the normal wedge product and discrete wedge product are in the same equivalence class,
\begin{equation}
\label{reductionwedge}
\reduction\left(\kdifformh{a}{k}\wedge_h \kdifformh{b}{l}\right)=\reduction(\kdifformh{a}{k}\wedge \kdifformh{b}{l}).
\end{equation}
In case $k+l=n$, we get
\[
\int_\Omega \kdifformh{a}{k}\wedge_h \kdifformh{b}{l}=\int_\Omega \kdifformh{a}{k}\wedge \kdifformh{b}{l}.
\]
The discrete wedge product in combination with the exterior derivative gives the Leibniz rule \eqref{eq::dif_and_wedge},
\begin{equation}
\begin{aligned}
\ederiv(\kdifformh{a}{k}\wedge_h\kdifformh{b}{l})&=\ederiv\pi(\kdifformh{a}{k}\wedge \kdifformh{b}{l})=\pi\ederiv(\kdifformh{a}{k}\wedge \kdifformh{b}{l})\\
&=\pi\left(\ederiv \kdifformh{a}{k}\wedge \kdifformh{b}{l}+(-1)^k\kdifformh{a}{k}\wedge\ederiv \kdifformh{b}{l}\right)\\
&=\ederiv \kdifformh{a}{k}\wedge_h \kdifformh{b}{l}+(-1)^k\kdifformh{a}{k}\wedge_h\ederiv \kdifformh{b}{l}.
\end{aligned}
\end{equation}
Since the pullback operator $\Phi^\star$ commutes with the projection, the discrete algebra homomorphism is satisfied
\[
\pullback(\kdifformh{a}{k}\wedge_h \kdifformh{b}{l})=\pullback\pi(\kdifformh{a}{k}\wedge \kdifformh{b}{l})=\pi\pullback(\kdifformh{a}{k}\wedge \kdifformh{b}{l})=\pi(\pullback \kdifformh{a}{k}\wedge\pullback \kdifformh{b}{l})=\pullback \kdifformh{a}{k}\wedge_h\pullback \kdifformh{b}{l}.
\]

\subsubsection{The discrete inner product}
Last operator to be defined is the inner product restricted to the finite dimensional subspace $\Lambda^k_h$.
\begin{definition}
Define a discrete inner product $\inner{\cdot}{\cdot}_h:\Lambda^k_h\times\Lambda^k_h\rightarrow\Lambda^n_h$ as
\begin{equation}
\innerspace{\kdifformh{a}{k}}{\kdifformh{b}{k}}{h}\define\pi\left\{\inner{\kdifformh{a}{k}}{\kdifformh{b}{k}}\omega^{(n)}\right\},
\label{discreteinnerproduct}
\end{equation}
with either $\pi=\pi_h$ or $\pi=\tilde{\pi}_h$. This discrete inner product is bilinear, symmetric, positive definite.
\end{definition}
\begin{corollary}[\textbf{Discrete $L^2$ inner product}]
The corresponding $L_2$ inner product $\innerspace{\cdot}{\cdot}{L^2\Omega,h}:\Lambda^k_h\times\Lambda^k_h\rightarrow\mathbb{R}$ is essentially the same as \eqref{eq::diffGeom_L2_inner}, because of \eqref{reductionprojection} and that $\Lambda^k_h\subset\Lambda^k$. Let $\kdifformh{a}{k},\kdifformh{b}{k}\in\Lambda^k_h$, then
\begin{equation}
\label{L2discreteinnerproduct}
\begin{gathered}
\innerspace{\kdifformh{a}{k}}{\kdifformh{b}{k}}{L^2\Omega,h}:=\int_\Omega\innerspace{\kdifformh{a}{k}}{\kdifformh{b}{k}}{h}\stackrel{\eqref{discreteinnerproduct}}{=}\int_\Omega\pi\left\{\left(\kdifformh{a}{k},\kdifformh{b}{k}\right)\omega^{(n)}\right\}\\
\stackrel{\eqref{reductionprojection}}{=}\int_\Omega\inner{\kdifformh{a}{k}}{\kdifformh{b}{k}}\kdifform{\omega}{n}=\innerspace{\kdifformh{a}{k}}{\kdifformh{b}{k}}{L^2\Omega}.
\end{gathered}
\end{equation}
\end{corollary}
The Hodge-$\star$ was defined in Definition \ref{def::diffGeom_hodge} as a combination of the inner product and the wedge product. Now when using the discrete wedge product and discrete inner product, the Hodge-$\star$ remains unchanged, and therefore
\begin{equation}
\inner{\kdifformh{a}{k}}{\kdifformh{b}{k}}_h=\kdifformh{a}{k}\wedge_h\star \kdifformh{b}{k}.
\label{discretehodge2}
\end{equation}
In weak formulations, integrals over $\Omega$ are considered. In that case \eqref{discretehodge2} reduces, according to \eqref{reductionwedge} and \eqref{L2discreteinnerproduct}, to
\begin{equation}
\innerspace{\kdifformh{a}{k}}{\kdifformh{b}{k}}{L^2\Omega}=\int_\Omega \kdifformh{a}{k}\wedge\star \kdifformh{b}{k}.
\end{equation}
This result shows that in a weak formulation, the original inner product, wedge product and Hodge-$\star$ operator can be used.

\subsection{Discrete Hodge decomposition}
Let $\Lambda^k_h\subset\Lambda^k$ be the space of finite dimensional differerential forms, and let $(\Lambda_h,\ederiv)$ be a finite-dimensional subcomplex with $\ederiv\Lambda^k_h\subset\Lambda^{k+1}_h$. Since the exterior derivative commutes with the projection operator, it follows that $\mathcal{B}(\ederiv,\Lambda^k_h)\subset\mathcal{B}(\ederiv,\Lambda^k)$, $\mathcal{Z}(\ederiv,\Lambda^k_h)\subset\mathcal{Z}(\ederiv,\Lambda^k)$ and that $\mathcal{B}(\ederiv,\Lambda^k_h)\subseteq\mathcal{Z}(\ederiv,\Lambda^k_h)$. We can make the following decomposition of the space of discrete differential forms,
\[
\Lambda^k_h=\mathcal{B}(\ederiv,\Lambda^{k-1}_h)\oplus\mathcal{B}^{c}(\ederiv,\Lambda^{k-1}_h).
\]
The space of discrete harmonic forms is defined as $\mathcal{H}_h^k=\mathcal{N}(\ederiv,\Lambda^k_h)\cap\mathcal{B}^c(\ederiv,\Lambda^{k-1}_h)$. This gives the following decomposition, $\mathcal{B}^c(\ederiv,\Lambda^{k-1}_h)=\mathcal{H}^k_h\oplus\mathcal{Z}^c(\ederiv,\Lambda^{k}_h)$. Although $\mathcal{B}^c(\ederiv,\Lambda^{k}_h)\subset\mathcal{B}^c(\ederiv,\Lambda^{k})$, in general, $\mathcal{H}^k_h\not\subset\mathcal{H}^k$ and $\mathcal{Z}^c(\ederiv,\Lambda^{k}_h)\not\subset\mathcal{Z}^c(\ederiv,\Lambda^{k})$. They both depend on $\coderiv$, since $\mathcal{Z}^c(\ederiv,\Lambda^{k})=\mathcal{B}(\coderiv,\Lambda^{k+1})$, of which we know that it does not commute with the projection $\pi_h$ on $\Lambda^k$, see \remarkref{rem:projectionhodge}. Both harmonic spaces $\mathcal{H}^k$ and $\mathcal{H}^k_h$ are finite dimensional spaces and their dimension depends on the topology of the domain $\Omega$. Because the dimension only depends on the topology (its Bettie number), we have $\mathrm{dim}\;\mathcal{H}^k_h=\mathrm{dim}\;\mathcal{H}^k$. The gap between $\mathcal{H}^k_h$ and $\mathcal{H}^k$ vanishes as $h\rightarrow0$. See Theorem 3.5 in  \cite{arnold2010finite} for the definition and details about the gap between the two harmonic form spaces. Substituting the decomposition into the previous decomposition gives the discrete Hodge decompostion,
\[
\Lambda^k_h=\mathcal{B}(\ederiv,\Lambda^{k-1}_h)\oplus\mathcal{H}^k_h\oplus\mathcal{B}(\coderiv,\Lambda^{k+1}_h).
\]

\begin{example}\label{ex:discretehodgedecomposition}
Consider the Hodge decomposition, Corollary~\ref{cor:hodgedecomposition}, of a $k$-form, $\kdifform{a}{k}\in\Lambda^k$, in terms of $\kdifform{b}{k-1}\in\Lambda^{k-1}$, $\kdifform{h}{k}\in\mathcal{H}^k$ and $\kdifform{c}{k+1}\in\Lambda^{k+1}$,
\[
\kdifform{a}{k}=\ederiv \kdifform{b}{k-1}+\kdifform{h}{k}+\coderiv \kdifform{c}{k+1}.
\]
Then apply the projection operator, $\pi_h$,
\[
\pi_h\kdifform{a}{k}
=\ederiv\pi_h\kdifform{b}{k-1}+\pi_h(\kdifform{h}{k}+\coderiv \kdifform{c}{k+1})=\underbrace{\ederiv \kdifformh{b}{k-1}}_{\in\mathcal{B}(\ederiv,\Lambda^k_h)}+\underbrace{(\kdifformh{h}{k}+\coderiv \kdifformh{e}{k+1})}_{\in\mathcal{B}^c(\ederiv,\Lambda^{k}_h)},
\]
where $\kdifformh{h}{k}\in\mathcal{H}^k_h\not\subset\mathcal{H}^k$ and $\kdifformh{e}{k+1}\in\Lambda^{k+1}_h$, but $\kdifformh{e}{k+1}\neq\pi_h\kdifform{c}{k+1}$.

If we apply the coprojection operator, $\coprojection$, to $\kdifform{a}{k}$, we obtain
\[
\coprojection \kdifform{a}{k}
=\coprojection\ederiv \kdifform{b}{k-1}+\coprojection \kdifform{h}{k}+\coderiv\coprojection \kdifform{c}{k+1}
=\underbrace{\ederiv \kdifformh{r}{k-1}+h_{h^*}^{(k)}}_{\in\mathcal{B}^c(\coderiv,\Lambda^k_h)}+\underbrace{\coderiv \kdifformh{c}{k+1}}_{\in\mathcal{B}(\coderiv,\Lambda^{k}_h)},
\]
where $h_{h^*}^{(k)}\in\mathcal{H}^k_{h^*}\not\subset\mathcal{H}^k$ and $\kdifformh{r}{k-1}\in\Lambda^{k-1}_h$, but $\kdifformh{r}{k-1}\neq\coprojection \kdifform{b}{k-1}$.
\end{example}

The continuous Hodge decomposition, discrete Hodge decomposition and the cochain space decompostion, are related to each other by means of the reduction and reconstructino operators.
\begin{proposition}\label{reductionreconstructionhodgedecomposition}
From the definition of the reduction operator $\reduction$, the Hodge decomposition for $k$-forms and the cochain space decomposition, it follows that
\begin{equation}
\mathcal{B}(\ederiv;\Lambda^{k-1})\stackrel{\reduction}{\longrightarrow}B^k,\quad
\mathcal{H}^k\stackrel{\reduction}{\longrightarrow}H^k,\quad
\mathcal{Z}^c(\ederiv;\Lambda^{k})\stackrel{\reduction}{\longrightarrow}(Z^{k})^c.
\end{equation}
From the definition of the reconstruction operator $\reconstruction$, the Hodge decomposition for finite dimensional $k$-forms and the cochain space decomposition, it follows that
\begin{equation}
B^k\stackrel{\reconstruction}{\longrightarrow}\mathcal{B}(\ederiv;\Lambda^{k-1}_h),\quad
H^k\stackrel{\reconstruction}{\longrightarrow}\mathcal{H}^k_h,\quad
(Z^{k})^c\stackrel{\reconstruction}{\longrightarrow}\mathcal{Z}^c(\ederiv;\Lambda^{k}_h).
\end{equation}
\end{proposition}

A consequence of this proposition is discrete well-posedness.
\begin{theorem}[\textbf{Discrete well-posedness}]
For $f^{(k+1)}\in\mathcal{B}(\ederiv;\Lambda^k)$ and $a^{(k)}\in\mathcal{Z}^c(\ederiv;\Lambda^k)$, the equation $\ederiv a^{(k)}=f^{(k+1)}$ is well-posed in the sense that there \emph{exists} a solution and the solution is \emph{unique}. Then the discrete equation $\ederiv\kdifformh{a}{k}=\kdifformh{f}{k+1}$ is also well-posed if and only if the projection $\pi$ is either $\pi_h$ or $\tilde{\pi}_h$.
\end{theorem}
\begin{proof}
For $\kdifform{f}{k+1}\in\mathcal{B}(\ederiv;\Lambda^k)$ and $\kdifform{a}{k}\in\mathcal{Z}^c(\ederiv;\Lambda^k)$ there exists a unique solution by the Hodge decomposition \eqref{eq:generic_Hodge_decomposition}. The discrete equation follows by projection,
\[
0=\pi(\ederiv\kdifform{a}{k}-\kdifform{f}{k+1})=\ederiv\pi\kdifform{a}{k}-\pi\kdifform{f}{k+1}=\ederiv\kdifformh{a}{k}-\kdifformh{f}{k+1}.
\]
Moreover we can write
\[
\pi\ederiv\kdifform{a}{k}-\pi\kdifform{f}{k+1}=\reconstruction(\delta\reduction\kdifform{a}{k}-\reduction\kdifform{f}{k+1})=\reconstruction(\delta\kcochain{a}{k}-\kcochain{f}{k+1}).
\]
By \propref{reductionreconstructionhodgedecomposition}, $\reduction\kdifform{a}{k}=\kcochain{a}{k}\in(Z^k)^c$ and $\reduction\kdifform{f}{k+1}=\kcochain{f}{k+1}\in B^{k+1}$. Because $\reconstruction$ is bijective, the discrete problem is also well-posed, since the cochain relation $\delta\kcochain{a}{k}=\kcochain{f}{k+1}$ is well-posed, see \propref{uniquecochain}. Furthermore, it follows that $\kdifformh{a}{k}$ has a unique solution, because $\kcochain{a}{k}$ is a unique cochain of the problem $\delta\kcochain{a}{k}=\kcochain{f}{k+1}$.
\end{proof}

Example~\ref{ex:discretehodgedecomposition} showed the decomposition of the finite dimensional $k$-form $\kdifformh{a}{k}\in\Lambda^k_h$. The following example explains how to extract the harmonic part from this $k$-form.
\begin{example}[\textbf{Discrete harmonic forms}]
Following \exampleref{ex:discretehodgedecomposition}, let $\pi_h\kdifform{a}{k}\in\Lambda^k_h$ and $h_h^{(k)}\in\mathcal{H}^k_h$ its harmonic part by the discrete Hodge decomposition.
In order to find the harmonic form $h^{(k)}_h$, we first determine the harmonic chains, $\kchain{h}{k}$ in the cell complex, as described in Example~\ref{ex:cell_complex_with_hole}, and the  harmonic cochains, $\kcochain{h}{k}$, as described in Example~\ref{ex:domain_with_hole_continued}. Both the harmonic chains and cochains can be obtained by purely topological considerations and the calculations are performed through the incidence matrices. For $dim(\mathcal{H}^k_h)=1$, we set $\kdifformh{h}{k}=\reconstruction\,(\alpha\kcochain{h}{k})$, where $\alpha$ is obtained from the definition of reduction,
\begin{equation}
\left \langle \reduction \kdifform{a}{k}, \kchain{h}{k} \right \rangle = \left \langle \alpha\kcochain{h}{k},\kchain{h}{k} \right \rangle \quad \Longrightarrow \quad \alpha = \frac{\left \langle \reduction \kdifform{a}{k}, \kchain{h}{k} \right \rangle}{\left \langle \kcochain{h}{k},\kchain{h}{k} \right \rangle}\;.
\end{equation}
When dim$({\mathcal H}_h^k)>1$, we repeat this process for all harmonic chain-cochain pairs and set
\begin{equation}
h_h^{(k)} = \sum_{j=1}^{\mathrm{dim}({\mathcal H}_h^k)} \alpha_j\, \reconstruction\,\kcochain{h}{k}_i\;.
\end{equation}
The harmonic form $h_{h^*}^{(k)}$ can be found in a similar way as in the case with $\projection$. Find the $(n-k)$-harmonic chain-cochain pairs, $(\tilde{\mathbf{h}}_{(n-k),j},\tilde{\mathbf{h}}^{(n-k)}_j)$, $j=1,2,\hdots,dim(\mathcal{H}_h^k)$, on the dual grid $\tilde{D}_i$. Then the harmonic form $h_{h^*}^{(k)}$ becomes
\begin{equation}
h_{h^*}^{(k)} = \sum_{j=1}^{\mathrm{dim}(\mathcal{H}^k_h)} \star \tilde{\reconstruction}\,(\tilde{\alpha}_j\kcochain{\tilde{h}}{n-k}_j)\;,\quad\mathrm{with}\quad
\tilde{\alpha}_j = \frac{\left \langle \tilde{\reduction} (\star \kdifform{a}{k}),\tilde{\mathbf{h}}_{(n-k),j} \right \rangle}{\left \langle \tilde{\mathbf{h}}^{(n-k)}_j,\tilde{\mathbf{h}}_{(n-k),j} \right \rangle}\;.
\end{equation}
%
Note that, by construction, $\ederiv h_h^{(k)} \equiv 0$ and $\coderiv h_{h^*}^{(k)} \equiv 0$.
\end{example}

\subsection{Discussion}
Although the reduction operator $\reduction$ and the reconstruction operator $\reconstruction$ are also introduced in \cite{bochev2006principles}, we have chosen to work in $\Lambda^k_h=\reconstruction\reduction\Lambda^k$ instead of $C^k=\reduction\Lambda^k$. The reason is that metric concepts like an inner product, the wedge product and the Hodge operator are metric concepts which cannot be modeled in topology. Furthermore, these metric concepts are independent of the reduction operator and depend explicitly on the reconstruction operator $\reconstruction$.

\section{The Mimetic Spectral Element Method}\label{sec:MSEM}
Now that a mimetic framework is formulated using differential geometry, algebraic topology and the relations between those, using the mimetic operators, we seek a projection that satisfies the properties of the mimetic framework. We restrict ourselves to piecewise polynomial reconstructions, as is common in most finite element/spectral element methods.

First the polynomials with mimetic properties are described. Next the discretization is explained using different kinds of sample problems.

\subsection{Domain partitioning.}
Instead of solving in the infinite-dimensional space $\Lambda^k(\Omega)$ we restrict our search to a proper subspace $\Lambda^k_h(\Omega;C_k)\subset\Lambda^k(\Omega)$. Although there already exists several methods for the subspace $\Lambda^k_h(\Omega;C_k)$, (see for example \cite{arnold2010finite,BuffaSangalliRivasVazquez2011,hiptmair2001,Rapetti2009}), here we focus on a spectral element based method for curvilinear quadrilaterals. In spectral element methods a domain decomposition of $\Omega$ is performed into $M$ non-overlapping curvilinear quadrilateral closed sub-domains $\Omega_m$:
\[
\Omega=\bigcup_{m=1}^M\Omega_m,\quad \Omega_m\cap\Omega_l=\partial\Omega_m\cap\partial\Omega_l,\ m\neq l,
\]
where in each sub-domain a Gauss-Lobatto grid is constructed, see \figref{fig:curvedmesh}. The collection of Gauss-Lobatto grids in all elements $\Omega_m$ constitutes the cell complex $D$.
For each spectral element there exists a sub cell complex, $D_m$, i.e. the subcell complex in the spectral element $\Omega_m$. Note that $D_m\cap D_l,\ m\neq l$, is not an empty set in case they are neighboring elements, but contains all $k$-cells, $k<n$, of the common boundary, see \defref{cellcomplex}.

Each sub-domain $\Omega_m$ is obtained from the map $\Phi_m:\Omega_{\rm ref}\rightarrow\Omega_m$, where the reference domain $\Omega_{\rm ref}$ is a unit cube $\Omega_{\rm ref}=[-1,1]^n$, with $n ={\rm dim}(\Omega)$. All differential forms defined on $\Omega_m$ are pulled back onto this reference element using the following pullback operation, $\pullback_m:\Lambda^k(\Omega_m)\rightarrow\Lambda^k(\Omega_{\rm ref})$. In three dimensions the reference element is given by
\begin{equation}\label{referencedomainR3}
\Omega_{\rm ref}\define\{(\xi,\eta,\zeta)\;|\;-1\leq\xi,\eta,\zeta\leq1\}.
\end{equation}
The solution within each sub-domain is expanded with respect to a polynomial basis, corresponding to the chains in that element.

\begin{figure}[h]
\centering
\includegraphics[width=0.35\textwidth]{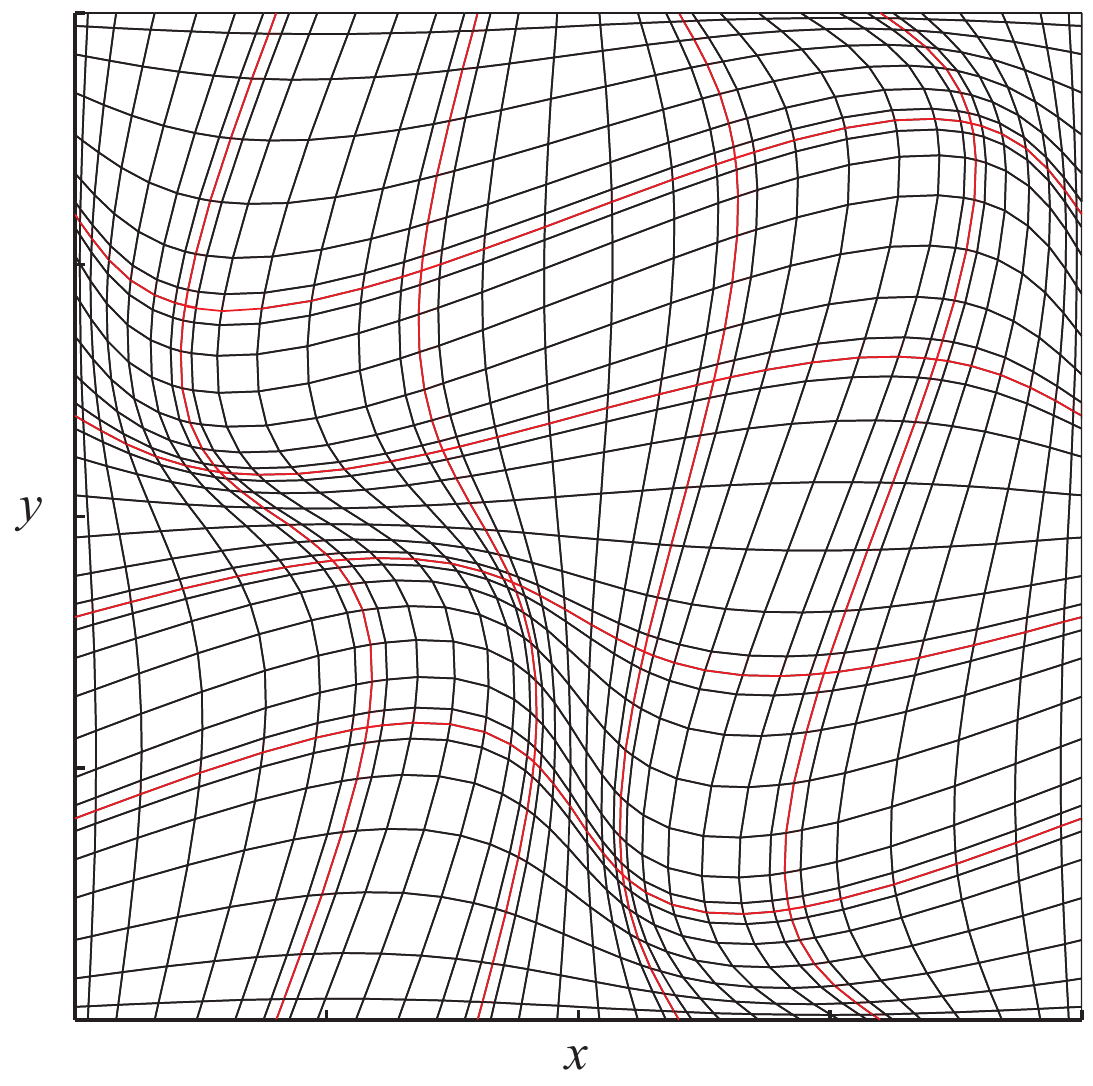}
\caption{A $5\times5$ curvilinear multi-element Gauss-Lobatto grid.}
\label{fig:curvedmesh}
\end{figure}

\begin{remark}
The map $\Phi_m =\varphi_m^{-1}$ is essentially the inverse of a chart from the computational manifold $\Omega$ to the unit cube in $\mathbb{R}^n$. The map from the cells formed by the Gauss-Lobatto grid in $[-1,1]^n$ to the cells in $\Omega$ are the singular $n$-cubes.
\end{remark}

\subsection{Reduction, mimetic basis-functions and projections}
In the previous section four projections -- two direct projection and two coprojections -- were introduced.
In this section we derive mimetic basis-functions, that are cardinal basis-functions of arbitrary polynomial order, that are capable of reconstructing $k$-cochains according to \defref{reconstructionoperator}. Because we consider only tensor product based mimetic spectral elements for the reconstruction, it is sufficient to do the derivation and analysis in one dimension only.

\subsubsection{Projection $\projection$ using $D$}

In spectral element methods the differential forms on $\Omega_m$ are approximated using piecewise polynomial expansions. 
The domain is given by $\Omega=\Omega_{\rm ref}\define\xi\in[-1,1]$. 
We start with the projection operator $\projection$, which is formed by the reconstruction of the reduction of a $k$-form on the interior cell complex $D_i$, see Section~\ref{sec:AT_boundary_cell_complex}.
On $\Omega$ a cell complex $D$ is defined according to \defref{cellcomplex}, that consists of $N+1$ nodes (0-cells), $\xi_i$, where $-1\leq \xi_0<\hdots<\xi_N\leq 1$, and $N$ line segments (1-cells), $[\xi_{i-1},\xi_i]$, of which the nodes constitute the boundary, see \figref{fig:dualcellcomplex1D}. Consider a 0-form $a^{(0)}=a(\xi)\in\Lambda^0(\Omega)$. Corresponding to this set of nodes (0-chain) there exists a projection, $\pi_h$, using $N^{\rm th}$ order \emph{Lagrange polynomials}, $l_i(\xi)$, to approximate a $0$-form, as
\begin{equation}
\projection\kdifform{a}{0}(\xi)=\sum_{i=0}^N a_il_i(\xi),\quad\mathrm{where}\quad a_i=\bar{\psi}_i\left(\kcochain{a}{0}\right)=\bar{\psi}_i\left(\reduction(a^{(0)})\right)=a(\xi_i),
\label{nodalapprox}
\end{equation}
where $\bar{\psi}$ is the isomorphism which associates a $k$-cochain to its coefficients in the expansion in terms of canonical basis $k$-cochains, see Definition~\ref{def:coeff_vector_for_cochains}. The $i$th coefficient is referred to as $\bar{\psi}_i$.
Lagrange polynomials have the property that they interpolate nodal values and are therefore suitable to reconstruct the cochain $\kcochain{a}{0}=\reduction\kdifform{a}{0}(\xi)$ containing the set $a_i=a(\xi_i)$ for $i=0,\hdots,N$. So these polynomials can be used to reconstruct a 0-form from a 0-cochain. Lagrange polynomials are in fact 0-forms themselves, $\kdifform{l_i}{0}(\xi)\in\Lambda^0_h(\Omega_{\rm ref};C_0)$.
They are constructed such that its value is one in the corresponding point and zero in all other grid points,
\begin{equation}
 \reduction l^{(0)}_i(\xi)=l^{(0)}_i(\xi_p)=\left\{
\begin{aligned}
&1&{\rm if}\ i=p\\ &0&{\rm if}\ i\neq p
\end{aligned}
\right..
 \label{nodalproperty}
\end{equation}
\begin{figure}[h]
\centering
\includegraphics[width=0.6\textwidth]{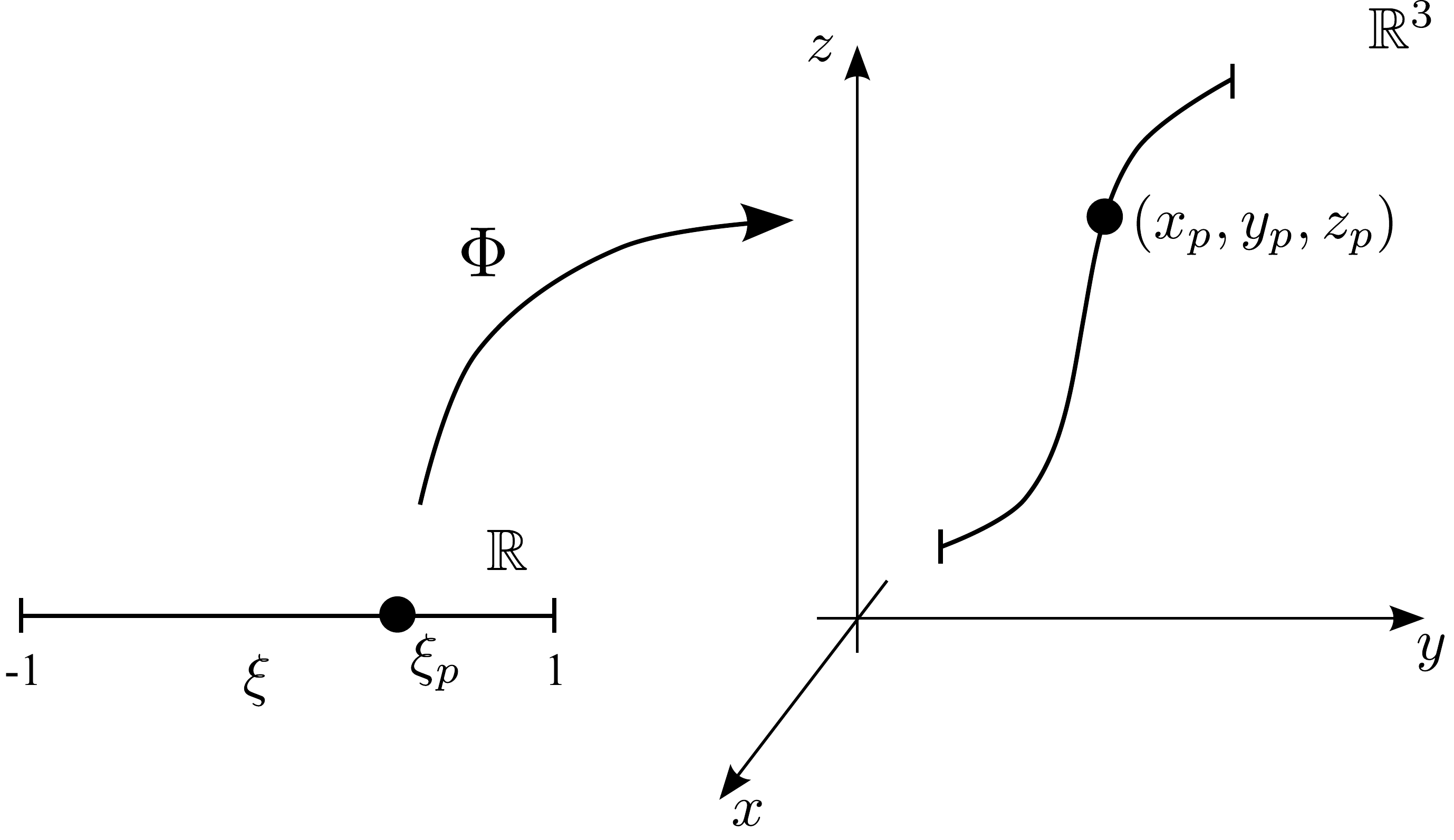}
\caption{The map from $\xi \in [-1,1]$ into a $1$-dimensional sub-manifold in $\mathbb{R}^2$.}
\label{fig:1D-manifold_map}
\end{figure}
This satisfies \eqref{consistency}, where in this case $\reconstruction=l^{(0)}_i(\xi)$. It is straightforward to show that Lagrange polynomials are invariant under a coordinate transformation. If $\bar{l}_i(x,y,z)$ is a Lagrange polynomial defined on a curvilinear 1-manifold embedded in a three-dimensional domain, then on that manifold there exist 0-cells $\tau_{(0),p}(x,y,z)$, associated to each node $p$, $(x_{p}, y_{p}, z_{p})$ , of the mesh of that manifold. In this case $(x,y,z)=\Phi(\xi)$ and $(x_{p},y_{p},z_{p})=\Phi(\xi_{p})$, see Figure~\ref{fig:1D-manifold_map} for a similar map in 2D. As a consequence $\pullback\bar{l}_i(x,y,z)=l_i(\xi)$, and so:
\[
\left.\bar{l}_i(x,y,z)\right|_{(x,y,z)=(x_{p},y_{p},z_{p})}=\left.\bar{l}_i(\Phi(\xi))\right|_{\xi=\xi_p}=\left.(\pullback\bar{l}_i)(\xi)\right|_{\xi=\xi_p}=\left.l_i(\xi)\right|_{\xi=\xi_p}=\delta_{i,p}.
\]

Other useful properties of Lagrange polynomials are
\begin{equation}
\sum_{i=0}^Nl_i(\xi)=1,\quad\quad \sum_{i=0}^N\ederiv l_i(\xi)=0.
\label{lagrangeproperties}
\end{equation}
Gerritsma \cite{gerritsma::edge_basis} and Robidoux \cite{robidoux-polynomial} independently derived the same projection for 1-forms, consisting of $1$-cochains and $1$-form polynomials, that is called the \textit{edge polynomial}, $e_i^{(1)}(\xi)$.
\begin{lemma}
\label{lemma:edge}
Following Definitions \ref{def:reduction} and \ref{reconstructionoperator}, application of the exterior derivative $d$ to $\kdifformh{a}{0}(\xi)$, gives the 1-form $\kdifformh{u}{1}(\xi)=\ederiv\kdifformh{a}{0}(\xi)$ given by
\begin{equation}
\projection\kdifform{u}{1}(\xi)=\kdifformh{u}{1}(\xi)=\sum_{i=1}^Nu_i\kdifform{e}{1}_i(\xi),
\end{equation}
with 1-cochain $\kcochain{u}{1}$, where
\begin{equation}
\label{ucochain}
\begin{aligned}
u_i&=\bar{\psi}_i\left(\kcochain{u}{1}\right)=\bar{\psi}_i\left(\reduction\kdifform{u}{1}\right)\\
&=\int_{\tau_{(1),i}}\kdifform{u}{1}(\xi)=\int_{\tau_{(1),i}}\ederiv a^{(0)}(\xi)\stackrel{\rm Th.~\ref{theor::difGeom_stokes_theorem}}{=}\int_{\partial\tau_{(1),i}}a^{(0)}(\xi)\\
&=\kdifform{a}{0}(\xi_i)-\kdifform{a}{0}(\xi_{i-1})=a_i-a_{i-1},
\end{aligned}
\end{equation}
and with the edge polynomial defined as
\begin{align}
\label{edge}
\kdifform{e}{1}_i(\xi)&=-\sum_{k=0}^{i-1}\ederiv\kdifform{l_k}{0}(\xi)=\sum_{k=i}^{N}\ederiv\kdifform{l_k}{0}(\xi)=\tfrac{1}{2}\sum_{k=i}^{N}\ederiv\kdifform{l_k}{0}(\xi)-\tfrac{1}{2}\sum_{k=0}^{i-1}\ederiv\kdifform{l_k}{0}(\xi).
\end{align}

\begin{proof} First take the exterior derivative of \eqref{nodalapprox}:
\[
\ederiv\kdifformh{a}{0}(\xi)=\sum_{i=0}^Na_i\ederiv l^{(0)}_i(\xi).
\]
Now the $1$-form $\ederiv \kdifformh{a}{0}$ is expanded in terms of the $0$-cochain $\kcochain{a}{0}$. Then a change of basis is applied to rewrite this expansion in terms of the $1$-cochain $\delta \kcochain{a}{0}$ using the projection $\pi_M$, see Definition~\ref{def:piM}. The trick is to subtract $a_k\sum_{i=0}^N\ederiv l^{(0)}_i(\xi)$, being equal to zero, and rewrite it such that we retrieve the coboundary operator and an edge polynomial,
\begin{displaymath}
\begin{aligned}
\kdifformh{u}{1}(\xi)&=\ederiv\kdifformh{a}{0}(\xi)=\pi_M\ederiv\kdifformh{a}{0}(\xi)\stackrel{\eqref{lagrangeproperties}}{=}\sum_{i=0}^N(a_i-a_k)\ederiv l^{(0)}_i(\xi)\\
&=\sum_{i=0}^N\left[-\sum_{j=i+1}^k(a_j-a_{j-1})+\sum_{j=k+1}^i(a_j-a_{j-1})\right]\ederiv l^{(0)}_i(\xi)\\
&\!\!\!\stackrel{\eqref{ucochain}}{=}-\sum_{i=0}^{k-1}\ederiv l^{(0)}_i(\xi)\sum_{j=i+1}^ku_j+\sum_{i=k+1}^N\ederiv l^{(0)}_i(\xi)\sum_{j=k+1}^iu_j\\
&=-\sum_{j=1}^k\left(\sum_{i=0}^{j-1}\ederiv l^{(0)}_i(\xi)\right)u_j+\sum_{j=k+1}^N\left(\sum_{i=j}^N\ederiv l^{(0)}_i(\xi)\right)u_j\\
&\!\!\!\stackrel{\eqref{lagrangeproperties}}{=}\sum_{i=1}^Nu_i\left(-\sum_{k=0}^{i-1}\ederiv l_k^{(0)}(\xi)\right)\stackrel{\eqref{edge}}{=}\sum_{i=1}^Nu_ie_i^{(1)}(\xi).
\end{aligned}
\end{displaymath}
\end{proof}
\end{lemma}
The cochain corresponding to line segment (1-cell) $\tau_{(1),i}$ is given by $u_i=a_i-a_{i-1}$ and so $\kcochain{u}{1}=\delta\kcochain{a}{0}$ is the discrete derivative operator in 1D. This operation is purely topological; no metric is involved. It satisfies \eqref{cdp2}, since $\ederiv\reconstruction\kcochain{a}{0}=\reconstruction\delta\kcochain{a}{0}$. Note that according to \eqref{eq:double_dif}, we have $\ederiv e_i(\xi)=\sum\ederiv\!\circ\!\ederiv  l^{(0)}_i(\xi)=0$.
The polynomial 1-form can be decomposed into a polynomial and the canonical basis for differential forms (see \propref{diffGeometry::basis_elements}),
\begin{displaymath}
e^{(1)}_i(\xi)=\ve_i(\xi)\ederiv \xi,\quad\mathrm{with}\quad\ve_i(\xi)=-\sum_{k=0}^{i-1}\frac{\ederiv l_k}{\ederiv \xi}.
\end{displaymath}
Similar to \eqref{nodalproperty}, the edge functions are constructed such that when integrating $e^{(1)}_i(\xi)$ over a line segment it gives one for the corresponding element and zero for any other line segment, so
\begin{equation}
 \reduction e_i^{(1)}(\xi)=\int_{\xi_{p-1}}^{\xi_p}e^{(1)}_i(\xi)=\left\{
\begin{aligned}
&1&{\rm if}\ i=p\\ &0&{\rm if}\ i\neq p
\end{aligned}
\right..
 \label{intedge}
\end{equation}
This also satisfies \eqref{consistency}, where in this case $\reconstruction=e^{(1)}_i(\xi)$. The last property to verify is invariance under transformations. If $\bar{e}_i(x,y,z)$ is an edge function defined on a curvilinear 1-manifold embedded in a three dimensional domain, then on that manifold there exist 1-cells $\tau_{(1),p}(x,y,z)$ , associated to each edge, $p$, of the mesh of that manifold. In this case $(x,y,z)=\Phi(\xi)$, Figure~\ref{fig:1D-manifold_map}, and $\tau_{(1),p}(x,y,z)=\Phi(\tau_{(1)}^p(\xi))$. As a consequence $\pullback\bar{e}_i(x,y,z)=e_i(\xi)$ and so:
\[
\int_{\tau_{(1)}^p(x,y,z)}\bar{e}_i(x,y,z)=\int_{\Phi(\tau_{(1)}^p(\xi))}(\Phi^{-1})^{\star}(e_i(\xi))
=\int_{(\Phi^{-1}\circ\Phi)(\tau_{(1)}^p(\xi))}e_i(\xi)
=\int_{\tau_{(1)}^p(\xi)}e_i(\xi).
\]

This is what could be expected, since the generalized Stokes Theorem \eqref{eq::difGeom_stokes_theorem} is purely topological and does not depend on the particular coordinate system or polynomial representation.

The following example shows the commutation property between projection and exterior derivative.
\begin{example}
In Figure~\ref{fig:comm_reduction_exterior_der} a graphical representation of the commutating diagram for the projection and the exterior derivative is given. This figure illustrates Lemma~\ref{Lem:projectionextder}.
\begin{figure}[h]
\centering
\includegraphics[width=0.7\textwidth]{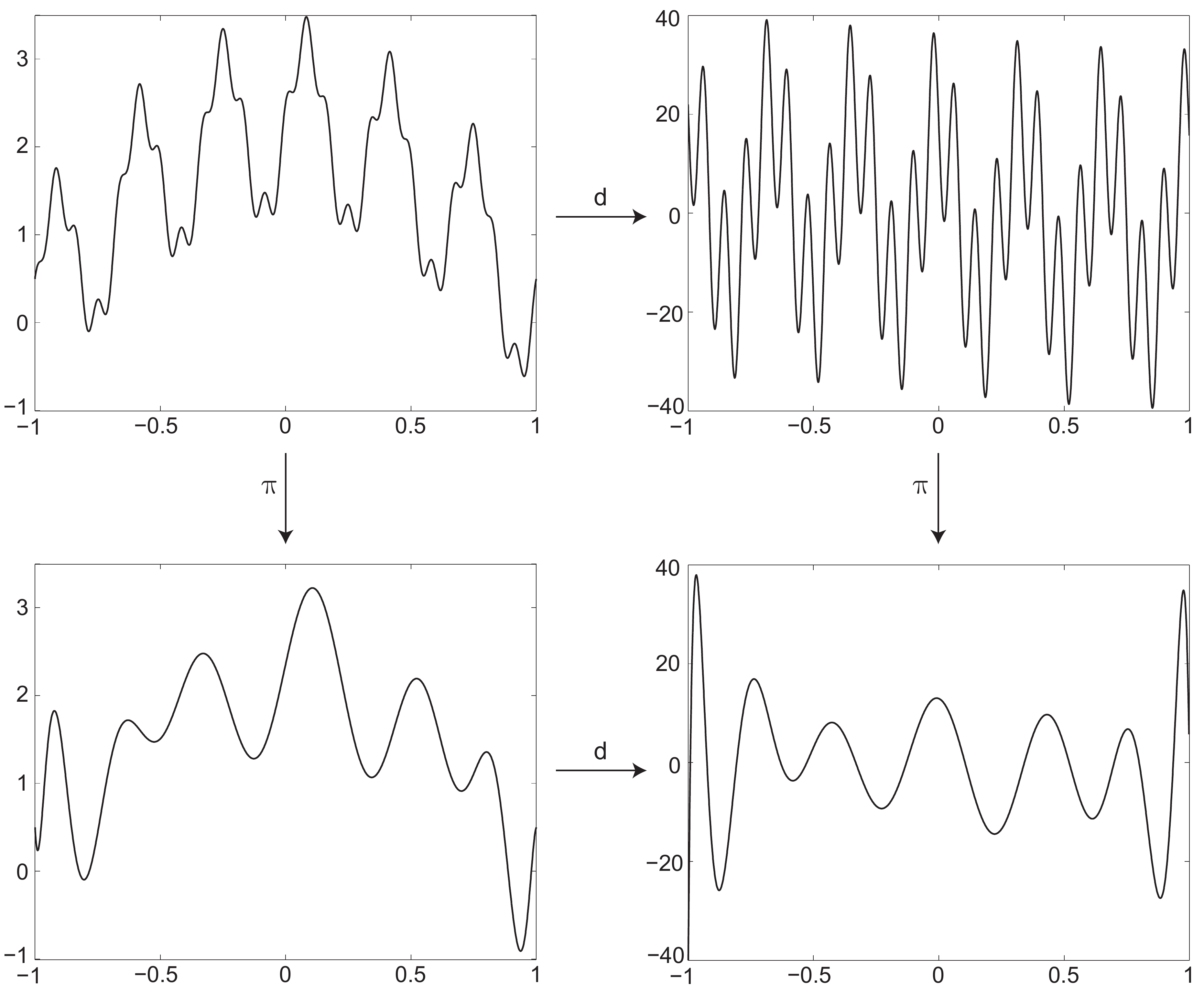}
\caption{An oscillatory $0$-form (Top left) is projected onto a polynomial $0$-form (Bottom left) using Lagrange functions. The exterior derivative of the oscillatory $0$-form (Top right) is projected onto polynomials using edge functions (Bottom right). This diagram commutes.}
\label{fig:comm_reduction_exterior_der}
\end{figure}
\end{example}

Now that the mimetic basis-functions are defined it can be proven by example that the projection, $\pi_h$, is in general not a Galerkin projection.
\begin{example}\label{galerkinprojection}
Let $\Pi_h$ be a Galerkin projection, and $\Pi_h\kdifform{a}{1}$ be expanded using edge basis-functions as $\Pi_h\kdifform{a}{1}=\sum_{i=1}^N\bar{a}_ie_i(\xi),$ then the coefficients $\bar{a}_i$ are determined by
\[
\left(\Pi_h\kdifform{a}{1},e_i(\xi)\right)=\left(\kdifform{a}{1},e_i(\xi)\right).
\]
In general, $[\bar{a}_1\ \hdots\ \bar{a}_N]^T\neq\reduction\kdifform{a}{1}$, and therefore $[\bar{a}_1\ \hdots\ \bar{a}_N]^T$ is not a cochain, and $\Pi_h$ is not a cochain projection. As an example, let $\kdifform{a}{1}=x^3\ederiv x$, then for $N=1$, $[\bar{a}_1\ \bar{a}_2]^T=[-\tfrac{3}{10}\ \tfrac{3}{10}]^T\neq\reduction\kdifform{a}{1}=[-\tfrac{1}{4}\ \tfrac{1}{4}]^T$.
\end{example}

\begin{figure}[t]
      \begin{minipage}[t]{0.49\linewidth}
            \centering\includegraphics[width=0.9\linewidth]{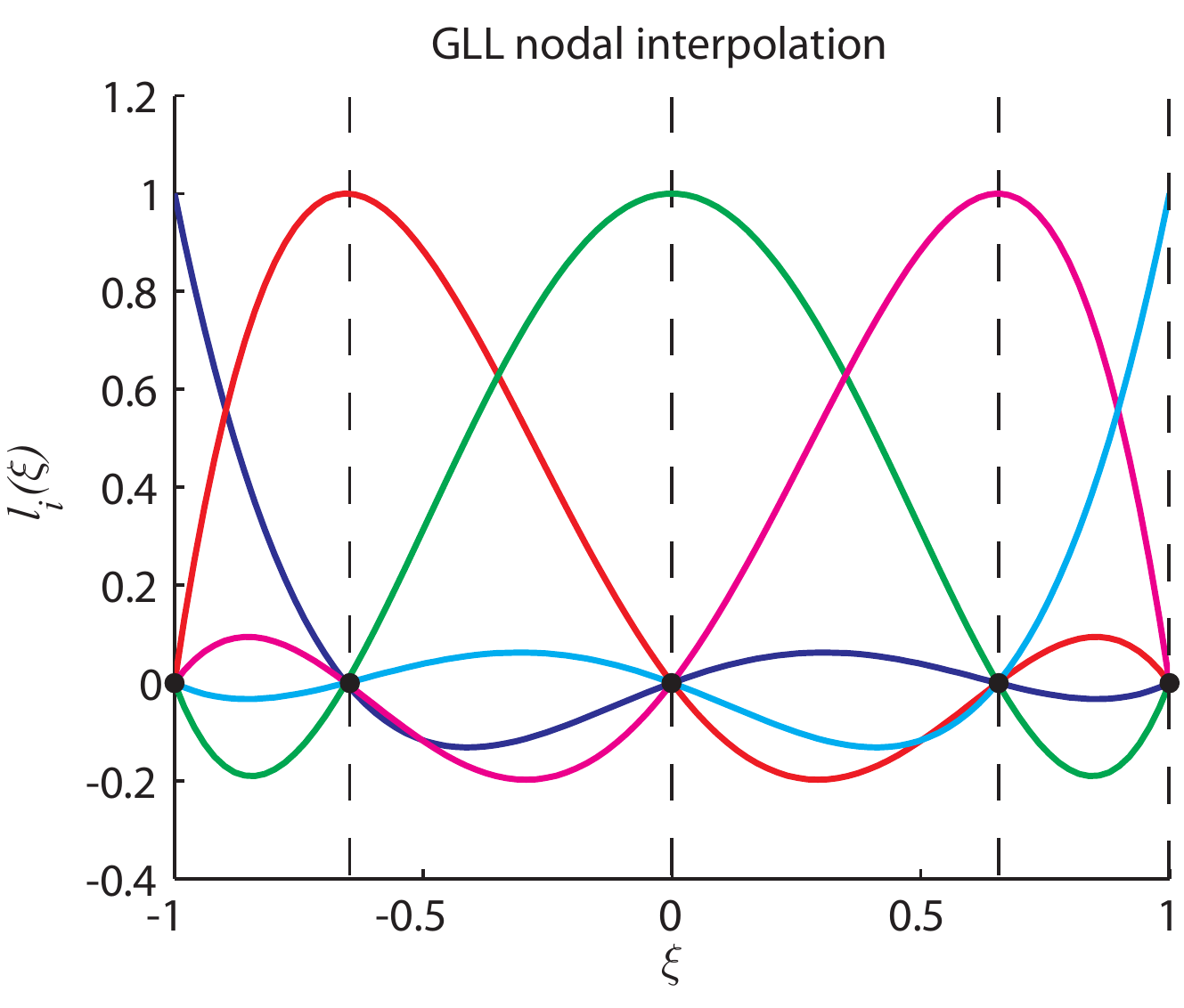}
	\caption{Lagrange polynomials on Gauss-Lobatto-Legendre grid.}
	\label{fig:lagrange}
    \end{minipage}\hfill
    \begin{minipage}[t]{0.49\linewidth}
            \centering\includegraphics[width=0.9\linewidth]{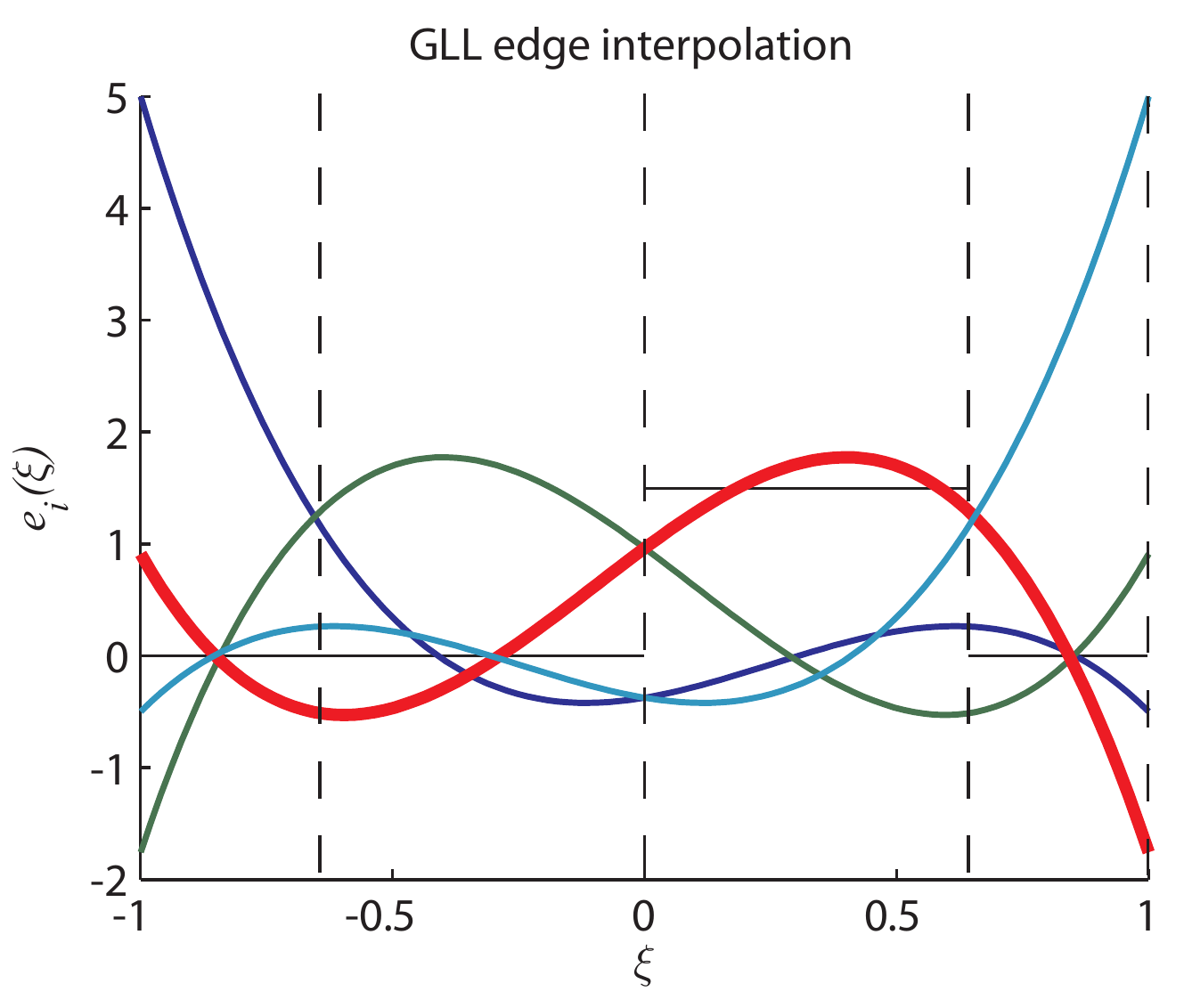}
	\caption{Edge polynomials on Gauss-Lobatto-Legendre grid.}
	\label{fig:edgepoly}
    \end{minipage}
\end{figure}

Again, let $N$ be the number of line segments in a spectral element. Then Lagrange basis interpolating $N+1$ 0-cells with fourth-order polynomials, corresponding to a Gauss-Lobatto grid, is shown in \figref{fig:lagrange}. The edge basis interpolating $N$ line segments with third-order polynomials, corresponding to a Gauss-Lobatto grid, is shown in \figref{fig:edgepoly}.

\subsubsection{Bounded linear projections} As mentioned in \defref{def:boundedprojection}, boundedness of the projection is a requirement, and is therefore shown for the projections introduced above.

The mimetic framework uses Lagrange, $l_i(\xi)\in H\Lambda^0(\Omega_{\rm ref})$, and edge functions, $e_i(\xi)\in L^2\Lambda^1(\Omega_{\rm ref})$, for the reconstruction, $\mathcal{I}$, where the latter is constructed using the former; i.e., from the finite dimensional 0-form $\pi_ha^{(0)}=\sum_{i=0}^N a_il_i(\xi)\in\Lambda^0_h(\Omega_{\rm ref};C_0)$, we define $\pi_hb^{(1)}\in\Lambda^1_h(\Omega_{\rm ref};C_1)$, such that
\[
\pi_hb^{(1)}=\pi_h\ederiv a^{(0)}=\sum_{i=1}^Nv_ie_i(\xi).
\]
Because we consider tensor products to construct higher-dimensional interpolation, it is sufficient to show that the projection operator is bounded in one dimension. A similar approach was used in \cite{BuffaSangalliRivasVazquez2011}. Due to the way the edge functions are constructed, there exists a commuting diagram property between projection and exterior derivative,
\[
\begin{CD}
\mathbb{R} @>>> H\Lambda^0 @>\ederiv>> L^2\Lambda^1 @>>>0 \\
@.  @VV \pi_h V @VV\pi_hV  @. \\
\mathbb{R} @>>> \Lambda_h^0 @>\ederiv>> \Lambda_h^1 @>>>0,
\end{CD}
\]
which gives, for $a^{(0)}\in H\Lambda^0(\Omega_{\rm ref})$, the $1$-form
\begin{equation}
\label{commutation}
\ederiv\pi_ha^{(0)}=\pi_h\ederiv a^{(0)}, \quad \mathrm{in}\ \Lambda^1_h(\Omega_{\rm ref}).
\end{equation}
Lagrange interpolation by itself does not guarantee a convergent approximation, \cite{Fejer}, but it requires a suitably chosen set of points, $-1\leq\xi_0<\xi_1<\hdots<\xi_N\leq1$. Here, (extended)-Gauss and Gauss-Lobatto distributions are proposed, because of their superior convergence behaviour. The a priori error estimate for these kind of interpolants in the $H\Lambda^0$-norm is given by, \cite{canuto1}
\begin{equation}
\label{hpconvergenceestimate}
\Vert a^{(0)}-\pi_ha^{(0)}\Vert_{H\Lambda^0}\leq C\frac{h^{l-1}}{p^{m-1}}|a^{(0)}|_{H^m\Lambda^0},\quad l=\mathrm{min}(p+1,m).
\end{equation}
Equation \eqref{hpconvergenceestimate} also implies that the projection of zero-forms is bounded in the $H\Lambda^0(\Omega_{\rm ref})$, as is shown in the following proposition.

\begin{proposition}\label{boundedh}
For $a^{(0)}\in H\Lambda^0(\Omega_{\rm ref})$ and the projection $\pi_h:H\Lambda^0\rightarrow \Lambda^0_h$, there exists the following two stability estimates in $H\Lambda^0$-norm and $H\Lambda^0$-semi-norm:
\begin{align}
\Vert\pi_ha^{(0)}\Vert_{H\Lambda^0}&\leq C\Vert a^{(0)}\Vert_{H\Lambda^0},\label{H1normstability}\\
|\pi_ha^{(0)}|_{H\Lambda^0}&\leq C|a^{(0)}|_{H\Lambda^0}.\label{H1seminormstability}
\end{align}
\begin{proof}
The $H\Lambda^0$-norm of $\pi_ha^{(0)}$ can be bounded from above using triangle inequality,
\[
\Vert\pi_ha^{(0)}\Vert_{H\Lambda^0}\leq\Vert a^{(0)}\Vert_{H\Lambda^0}+\Vert a^{(0)}-\pi_ha^{(0)}\Vert_{H\Lambda^0}.
\]
For $a^{(0)}\in H\Lambda^0$, we use Poincar\'e inequality, \lemmaref{poincareinequality}, to bound the first term,
\[
\Vert a^{(0)} \Vert_{H\Lambda^0}\leq c_P \Vert \ederiv a^{(0)} \Vert_{L^2\Lambda^0}= c_P\vert a^{(0)}\vert_{H\Lambda^0}.
\]
Substituting this inequality and the interpolation estimate \eqref{hpconvergenceestimate} into the triangle inequality gives the following result,
\[
\Vert\pi_ha^{(0)}\Vert_{H\Lambda^0}\leq C|a^{(0)}|_{H\Lambda^0}.
\]
Then \eqref{H1normstability} and \eqref{H1seminormstability} follow directly
\[
|\pi_ha^{(0)}|_{H\Lambda^0}\leq\Vert\pi_ha^{(0)}\Vert_{H\Lambda^0}\leq C|a^{(0)}|_{H\Lambda^0}\leq C\Vert a^{(0)}\Vert_{H\Lambda^0}.
\]
\end{proof}
\end{proposition}
Now that we have a bounded linear projection\footnote{also referred to as \emph{bounded cochain projection}, \cite{arnold2006finite,arnold2010finite}.} of $0$-forms in one dimension, we can also proof boundedness of the projection for $1$-forms.
\begin{proposition}\label{boundede}
Let $a^{(0)}\in H\Lambda^0$ and $b^{(1)}\in L^2\Lambda^1$, then the projection $\pi_h:L^2\Lambda^1\rightarrow\Lambda^1_h$ given by Lemma~\ref{lemma:edge} is bounded
\begin{equation}
\Vert\pi_hb^{(1)}\Vert_{L^2\Lambda^1}\leq C\Vert b^{(1)}\Vert_{L^2\Lambda^1}.
\end{equation}
\begin{proof}
Because $L^2\Lambda^1=\mathcal{R}(\ederiv;H\Lambda^0)$, we can write $b^{(1)}=\ederiv a^{(0)}$. Then proof follows from the result of the previous proposition and the commutation between projection and derivative, \lemmaref{Lem:projectionextder},
\[
\begin{aligned}
\Vert\pi_hb^{(1)}\Vert_{L^2\Lambda^1}&=|\pi_h\ederiv a^{(0)}|_{L^2\Lambda^1}=|\ederiv\pi_ha^{(0)}|_{L^2\Lambda^1}=|\pi_ha^{(0)}|_{H\Lambda^0}\\
&\leq C|a^{(0)}|_{H\Lambda^0}=C|\ederiv a^{(0)}|_{L^2\Lambda^1}=C\Vert b^{(1)}\Vert_{L^2\Lambda^1}.
\end{aligned}
\]
\end{proof}
\end{proposition}
Propositions \ref{boundedh} and \ref{boundede} show that the projection $\pi_h$ is a \emph{bounded linear projection}, based on Lagrange functions and edge functions. Just like for $0$-forms using Lagrange interpolation, we can also give an estimate for the interpolation error of $1$-forms, interpolated using edge functions.
\begin{proposition}\label{convergenceedge}
Let $a^{(0)}\in H\Lambda^0$ and $b^{(1)}=\ederiv a^{(0)}\in L^2\Lambda^1$, the interpolation error $b^{(1)}-\pi_hb^{(1)}\in L^2\Lambda^1$ is given by
\begin{equation}
\label{edgeinterpolationerror}
\Vert b^{(1)}-\pi_hb^{(1)}\Vert_{L^2\Lambda^1}\leq C\frac{h^{l-1}}{p^{m-1}}|b^{(1)}|_{H^{m-1}\Lambda^1},
\end{equation}
with $l=\mathrm{min}(p+1,m)$.
\begin{proof}
\[
\Vert b^{(1)}-\pi_hb^{(1)}\Vert_{L^2\Lambda^1}=\Vert\ederiv(a^{(0)}-\pi_ha^{(0)})\Vert_{L^2\Lambda^1}=|a^{(0)}-\pi_ha^{(0)}|_{H\Lambda^0}
\]
and
\[
|a^{(0)}-\pi_ha^{(0)}|_{H\Lambda^0}\leq\Vert a^{(0)}-\pi_ha^{(0)}\Vert_{H\Lambda^0}.
\]
Then \eqref{edgeinterpolationerror} follows from \eqref{hpconvergenceestimate}, with the semi-norm rewritten as
\[
|a^{(0)}|_{H^m\Lambda^0}=|\ederiv a^{(0)}|_{H^{m-1}\Lambda^1}=|b^{(1)}|_{H^{m-1}\Lambda^1}.
\]
\end{proof}
\end{proposition}

\begin{corollary}
Although Propositions \ref{boundedh}, \ref{boundede} and \ref{convergenceedge} were derived for the 0- and 1-forms on the reference domain $\Omega_{\rm ref}$. Due to the commuting property between the pullback and exterior derivative, \propref{prop:commutation_pullback_ederiv}, these propositions also hold on any curvilinear domain $\Omega$, where $\Phi:\Omega_{\rm ref}\rightarrow\Omega$.
\end{corollary}

\subsubsection{Projection $\dualprojection$ using $\tilde{D}$}
Given a cell complex $D_i$, a corresponding dual grid $\tilde{D}_i$ was defined in \secref{sec:dualcomplex}.
For simplicity, consider a one-dimensional manifold, on which the two cell complexes are defined (e.g. \figref{fig:dualcellcomplex1D}) . Let the Gauss-Lobatto grid defined above be the primal cell complex, $D_i$,
consisting of $N+1$ points and $N$ line segments. The dual of $D_i$, being $\tilde{D}_i$, consists of $N$ points and $N+1$ line segments, according to \defref{def:dualcellcomplex}. Having a Gauss-Lobatto primal cell complex, the points for the dual part could be the Gauss points, $-1<\tilde{\xi}_1<\hdots<\tilde{\xi}_N<1$, \cite{canuto1}. We indicate the Gauss-Lagrange interpolant corresponding to $\tilde{D}_i$ by $\tilde{l}^\mathrm{g}_i(\xi)\in\Lambda^0_h(\Omega_{\rm ref};\tilde{C}_0)$, with $\tilde{C}_0(\tilde{D}_i)$. The projection of a $0$-form, $\kdifform{a}{0}\in\Lambda^0(\Omega_{\rm ref})$, using the $0$-cells in $\tilde{D}_i$, is a polynomial of degree $N-1$, and is given by
\begin{equation}
\tilde{\pi}_h\kdifform{a}{0}(\xi)=\sum_{i=1}^Na_i\tilde{l}^\mathrm{g}_i(\xi).
\end{equation}
The dual complex $\tilde{D}$ is the union of $\tilde{D}_i$ and $\tilde{D}_b$, where the two additional boundary points are given by $\tilde{\xi}_0=-1$ and $\tilde{\xi}_{N+1}=1$. The $0$-cells of the dual cell complex $\tilde{D}$ are interpolated using \emph{extended-Gauss-Lagrange} polynomials, $\tilde{l}^{\mathrm{eg}}_i(\xi)\in\Lambda^0_h(\Omega_{\rm ref};\tilde{C}_0)$, with $\tilde{C}_0=C_0(\tilde{D})$. The projection of a $0$-form, $\kdifform{a}{0}\in\Lambda^0(\Omega)$, using the $0$-cells in cell complex $\tilde{D}$, is a polynomial of degree $N+1$, and is given by
\begin{equation}
\tilde{\pi}_h\kdifform{a}{0}(\xi)=\sum_{i=0}^{N+1}a_i\tilde{l}^\mathrm{eg}_i(\xi).
\end{equation}
The \emph{extended Gauss-edge} functions, $\tilde{e}^\mathrm{eg}_i\in\Lambda^1_h(\Omega_{\rm ref};\tilde{C}_1)$ with $\tilde{C}_1\in\tilde{D}$, are found similarly to \lemmaref{lemma:edge}. The projection of a one-form $\kdifform{b}{1}\in\Lambda^1(\Omega)$, using the 1-cells in $\tilde{D}$, is a polynomial of degree $N$, and is given by
\begin{equation}
\tilde{\pi}_h\kdifform{b}{1}(\xi)=\sum_{i=1}^{N+1}b_i\tilde{e}^{\mathrm{eg}}_i(\xi),\quad\mathrm{with}\quad \tilde{e}^\mathrm{eg}_i(\xi)=-\sum_{k=0}^{i-1}\ederiv \tilde{l}^\mathrm{eg}_i(\xi).
\end{equation}
The extended-Gauss polynomials in the context of mimetic discretization were first preliminarily discussed in \cite{gerritsma:lssc2009,bouman::icosahom2009}.

\begin{remark}
The boundary cells in $\tilde{D}_b$ can be seen as connectivity cells. Either they connect the computational domain to the outside world in terms of boundary conditions, or they connect adjacent spectral elements within the computational domain, thus providing $C^0$-continuity for $0$-forms.
\end{remark}
\begin{remark}
Although the boundary cells are part of the solution and thus should be solved for, the integral over the domain $[-1,1]$ of the boundary polynomials is zero, i.e.
\[ \int_{-1}^1 \tilde{l}^\mathrm{eg}_0(\xi)\,d\xi = \int_{-1}^1 \tilde{l}^\mathrm{eg}_{N+1}(\xi)\,d\xi \equiv 0 \;.\]
Therefore they do not contribute to the domain integrals over $\Omega=[-1,1]$. There is a strong analogy with integration by parts. Let $q_h^{(0)}\in\Lambda^0_h(\Omega;C_0)$ be represented using the primal cell complex and $\tilde{a}^{(0)}_h\in\Lambda^0_h(\Omega;\tilde{C}_0)$ be expressed using the dual cell complex. Then integration by parts gives
\[
\int_\Omega\ederiv q^{(0)}_h\wedge\tilde{a}^{(0)}_h=-\int_\Omega q^{(0)}_h\wedge\ederiv\tilde{a}^{(0)}_h+\int_{\partial\Omega}q^{(0)}_h\wedge\tilde{a}^{(0)},\quad\forall q^{(0)}_h\in\Lambda^0_h(\Omega;C_0).
\]
This shows that the derivative on the primal cell complex $D$ is related to the derivative on the interior part of the dual complex, $\tilde{D}_i$, supplemented with boundary part, $\tilde{D}_b$.
\end{remark}

\begin{figure}[t]
      \begin{minipage}[t]{0.49\linewidth}
            \centering\includegraphics[width=0.9\linewidth]{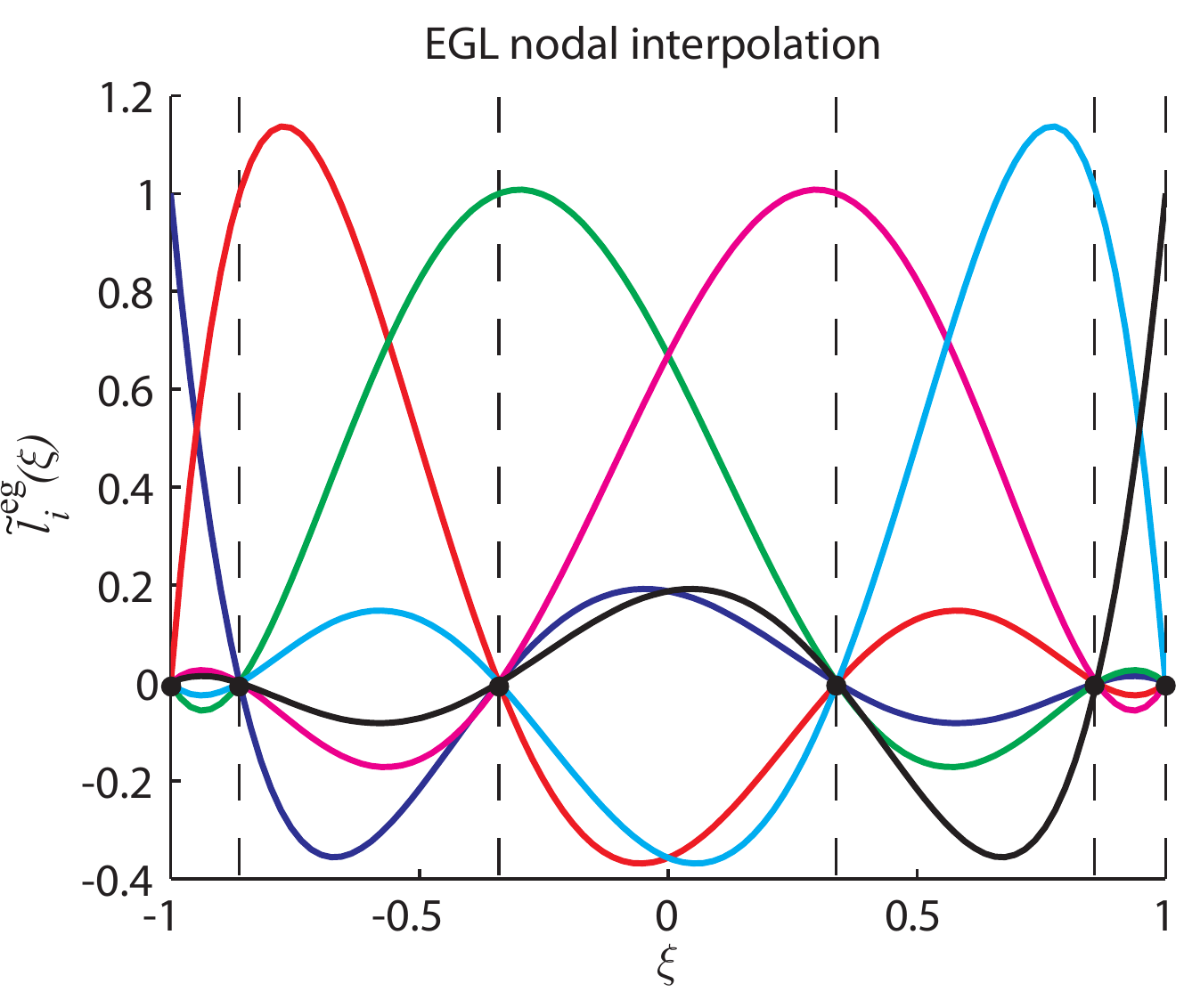}
	\caption{Lagrange polynomials on extended Gauss-Legendre grid.}
	\label{fig:nodal_dual_grid}
    \end{minipage}\hfill
    \begin{minipage}[t]{0.49\linewidth}
            \centering\includegraphics[width=0.9\linewidth]{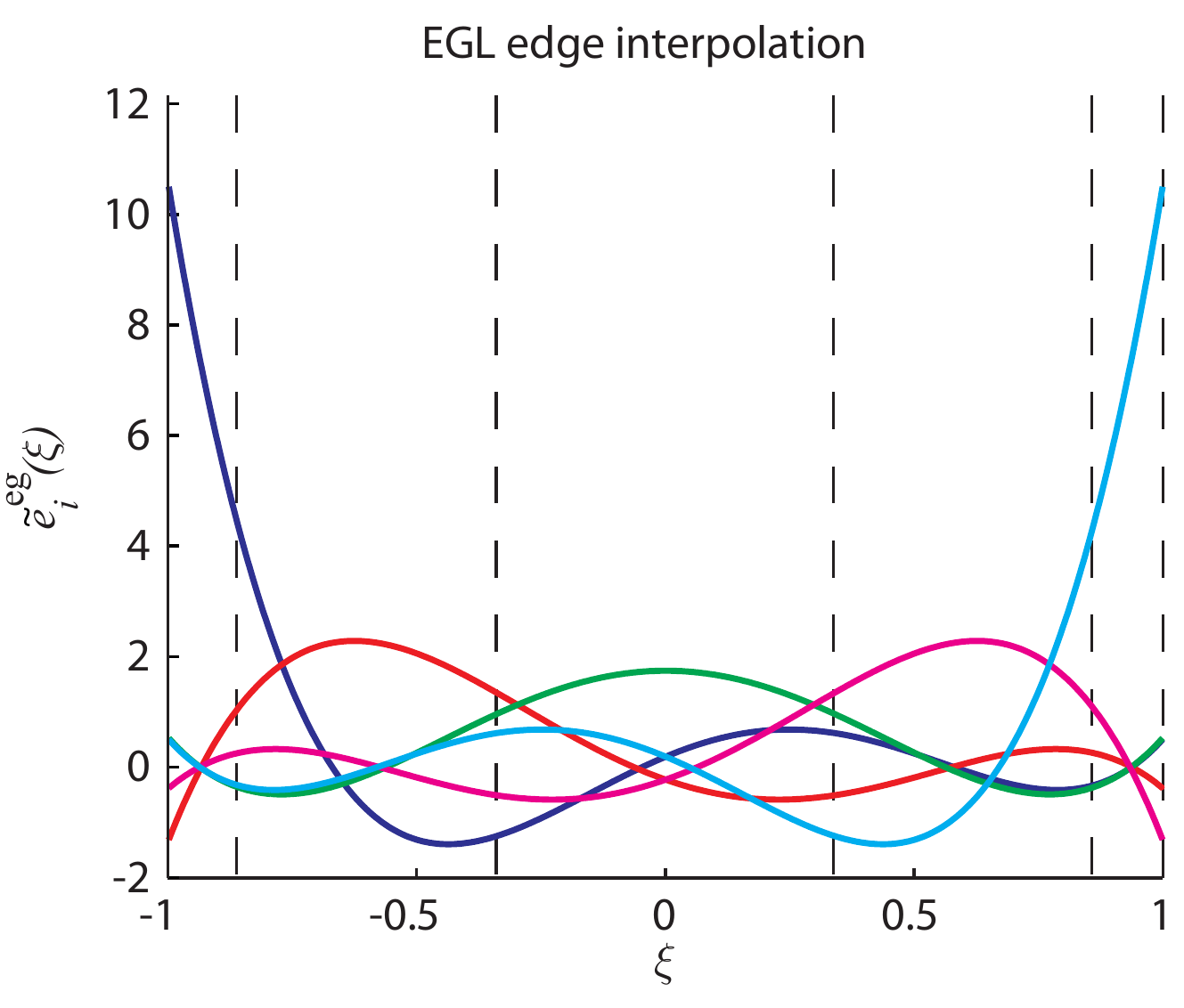}
	\caption{Edge polynomials on extended Gauss-Legendre grid.}
	\label{fig:edge_dual}
    \end{minipage}
\end{figure}

Figure~\ref{fig:nodal_dual_grid} shows the nodal basis functions on the dual grid, while Figure~\ref{fig:edge_dual} shows the edge functions on the dual grid. Together with the basis functions depicted in Figures~\ref{fig:lagrange} and \ref{fig:edgepoly}, these four sets of basis functions make up the entire set of reconstruction operators discussed in Section~\ref{mimeticoperators}.

\subsubsection{The coprojections $\coprojection$ and $\dualcoprojection$}
Now that the reconstruction polynomials on the primal and the dual complex have been introduced, the coprojections can be directly expressed in terms of these polynomials. For $\coprojection$, we take the Hodge-$\star$ of a $k$-form $\kdifform{a}{k}$ at the continuous level, which yields a $(n-k)$-form. This form is projected by $\dualprojection$ with respect to the dual grid. Then the Hodge-$\star$ operator is applied to the projected differential form and multiplied by $(-1)^{k(n-k)}$. This result is then expanded in terms of the $k$-form basis functions on the primal grid. This is possible thanks to property (\ref{eq:tildeIc_n-k=Ick}).

A similar route is followed by the coprojection $\dualcoprojection$, where in this case the projection is with respect to the primal grid and the final result is expanded in terms of the $k$-form basis functions on the dual complex. In \figref{fig:projections}, the four projection operators are applied to a 0-form, $a^{(0)} = \mathrm{sin}(3\pi x+0.08)$, and the resulting projected $0$-forms are plotted. 

\begin{figure}[h]
\centering
\includegraphics[width=0.5\textwidth]{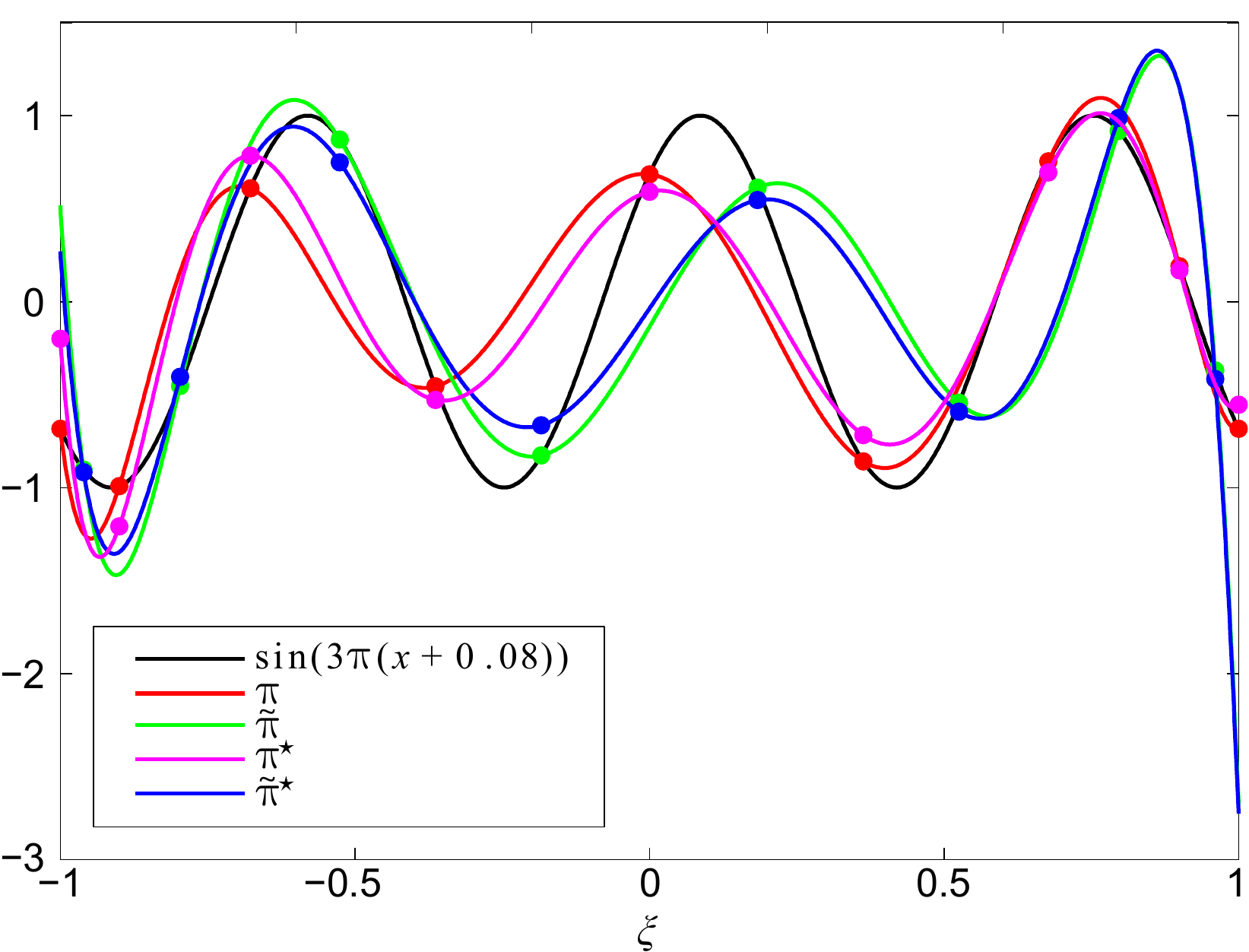}
\caption{Comparison between the four projections, $\pi$, $\tilde{\pi}$, $\pi^{\star}$ and $\tilde{\pi}^{\star}$, for a one-dimensional mesh of order $N=8$.}
\label{fig:projections}
\end{figure}

\subsection{Applications of discrete operators}
In this section we will show some examples of the use of mimetic basis-functions and the action of the operators $\ederiv,\ \star$ and $\coderiv$ described in Section \ref{discreteoperators}. Now that Lagrange polynomials and edge polynomials are defined, we can construct interpolants for all $k$-cochains in $\Omega\subset\mathbb{R}^n$. Because we consider quadrilateral elements only and employ tensor products to form the spectral element basis, the higher-dimensional basis functions are formed naturally by applying tensor products. For instance, in $\Omega\subset\mathbb{R}^3$, the surface element is the tensor product of two edge polynomials and one Lagrange polynomial, whereas the volume basis function is the tensor product of three edge polynomials.\\
 \\
We start with the analog between the exterior derivative and the coboundary operator in 3D, as illustrated in the diagram below.
\begin{diagram}
\Lambda^{0}_h(\Omega;C_0) & \rTo^{\quad\ederiv\quad}_{\rm grad} & \Lambda^{1}_h(\Omega;C_1) & \rTo^{\quad\ederiv\quad}_{\rm curl} & \Lambda^{2}_h(\Omega;C_2) & \rTo^{\quad\ederiv_h\quad}_{\rm div} & \Lambda^{3}_h(\Omega;C_3)\\
\dTo^{\reduction} \uTo_{\reconstruction} &  & \dTo^{\reduction} \uTo_{\reconstruction} & & \dTo^{\reduction} \uTo_{\reconstruction} &  & \dTo^{\reduction} \uTo_{\reconstruction}\\
C^{0}(D) & \rTo^{\delta} & C^{1}(D) & \rTo^{\delta} & C^{2}(D) & \rTo^{\delta} & C^{3}(D)
\end{diagram}
We recognize the continuous and discrete gradient, curl and divergence operator from vector calculus. Again for clarity, we restrict ourselves to $\Omega=\Omega_{\rm ref}$, see \eqref{referencedomainR3}. A similar De Rham complex can be set up for the dual complex using the nodal and edge functions given in the previous subsection.  Note that in the finite dimensional setting there is no need to distinguish between $\projection$ and $\coprojection$, due to Proposition~\ref{prop:pispace=pi_starspace}.

\begin{example}[\textbf{Gradient operator}]\label{ex:discrete_gradient}
Consider $\kdifformh{u}{1}=\ederiv\kdifformh{p}{0}$, where $\kdifformh{p}{0}$ is expanded in standard coordinates $(\xi,\eta,\zeta)$ as
\begin{equation}
\kdifformh{p}{0}(\xi,\eta,\zeta)=\sum_{i=0}^N\sum_{j=0}^N\sum_{k=0}^Np_{i,j,k}l_i(\xi)l_j(\eta)l_k(\zeta).
\end{equation}
Apply the exterior derivative in the same way as in \lemmaref{lemma:edge}, it gives
\begin{equation}
\label{uh1}
\begin{aligned}
\kdifformh{u}{1}=&\sum_{i=1}^N\sum_{j=0}^N\sum_{k=0}^Nu^\xi_{i,j,k}e_i(\xi)l_j(\eta)l_k(\zeta)\\
+&\sum_{i=0}^N\sum_{j=1}^N\sum_{k=0}^Nu^\eta_{i,j,k}l_i(\xi)e_j(\eta)l_k(\zeta)\\
+&\sum_{i=0}^N\sum_{j=0}^N\sum_{k=1}^Nu^\zeta_{i,j,k}l_i(\xi)l_j(\eta)e_k(\zeta),
\end{aligned}
\end{equation}
where
\begin{equation}
u^\xi_{i,j,k}=p_{i,j,k}-p_{i-1,j,k},\ u^\eta_{i,j,k}=p_{i,j,k}-p_{i,j-1,k}\ \mathrm{and}\ u^\zeta_{i,j,k}=p_{i,j,k}-p_{i,j,k-1},
\end{equation}
can compactly be written as $\kcochain{u}{1}=\delta\kcochain{p}{0}$, or in matrix notation as $\bar{\psi}(\kcochain{u}{1})=\incidencederivative{1}{0}\bar{\psi}(\kcochain{p}{0})$. This relation is exact and invariant under transformations. Note that in terms of function spaces the vector proxies of $p_h^{(0)}$ are in $H^1$ and that of $u_h^{(1)}$ are in $H(\mathrm{curl})$.
\end{example}

\begin{example}[\textbf{Curl operator}]\label{ex:discrete_curl_equation}
Let $\kdifformh{u}{1}$ be defined as in \eqref{uh1}, then $\kdifformh{w}{2}=\ederiv\kdifformh{u}{1}$. Apply the exterior derivative as in \lemmaref{lemma:edge} and consider the wedge product property \eqref{awedgea}, it gives
\begin{equation}
\label{wh2}
\begin{aligned}
\kdifformh{w}{2}=&\sum_{i=0}^N\sum_{j=1}^N\sum_{k=1}^Nw^\xi_{i,j,k}l_i(\xi)e_j(\eta)e_k(\zeta)\\
+&\sum_{i=1}^N\sum_{j=0}^N\sum_{k=1}^Nw^\eta_{i,j,k}e_i(\xi)l_j(\eta)e_k(\zeta)\\
+&\sum_{i=1}^N\sum_{j=1}^N\sum_{k=0}^Nw^\zeta_{i,j,k}e_i(\xi)e_j(\eta)l_k(\zeta),
\end{aligned}
\end{equation}
where
\begin{equation}
\begin{aligned}
w^\xi_{i,j,k}&=u^\zeta_{i,j,k}-u^\zeta_{i,j-1,k}-u^\eta_{i,j,k}+u^\eta_{i,j,k-1},\\
w^\eta_{i,j,k}&=u^\zeta_{i,j,k}-u^\zeta_{i-1,j,k}-u^\xi_{i,j,k}+u^\xi_{i,j,k-1}-,\\
w^\zeta_{i,j,k}&=u^\eta_{i,j,k}-u^\eta_{i-1,j,k}-u^\xi_{i,j,k}+u^\xi_{i,j-1,k},
\end{aligned}
\end{equation}
can compactly be written as $\kcochain{w}{2}=\delta\kcochain{u}{1}$, or in matrix notation as $\bar{\psi}(\kcochain{w}{2})=\incidencederivative{2}{1}\bar{\psi}(\kcochain{u}{1})$. This relation is exact and invariant under transformations. Note that in terms of function spaces the vector proxies of $w_h^{(2)}$ are in $H(\mathrm{div})$. If $\kdifformh{u}{1}$ is the gradient of $\kdifformh{p}{0}$, then $w^\xi_{i,j,k}=w^\eta_{i,j,k}=w^\zeta_{i,j,k}=0$ and so $\kdifformh{w}{2}=0$. This is in accordance with \eqref{algtop::stokes} and \eqref{eq:double_dif}. 
\end{example}

\begin{example}[\textbf{Divergence operator}]\label{ex:discrete_divergence_equation}
Let the 2-form $\kdifformh{q}{2}$ be defined as
\begin{equation}
\label{qh2}
\begin{aligned}
\kdifformh{q}{2}=&\sum_{i=0}^N\sum_{j=1}^N\sum_{k=1}^Nq^\xi_{i,j,k}l_i(\xi)e_j(\eta)e_k(\zeta)\\
+&\sum_{i=1}^N\sum_{j=0}^N\sum_{k=1}^Nq^\eta_{i,j,k}e_i(\xi)l_j(\eta)e_k(\zeta)\\
+&\sum_{i=1}^N\sum_{j=1}^N\sum_{k=0}^Nq^\zeta_{i,j,k}e_i(\xi)e_j(\eta)l_k(\zeta),
\end{aligned}
\end{equation}
then $\kdifformh{v}{3}=\ederiv\kdifformh{q}{2}$. Apply the exterior derivative as in \lemmaref{lemma:edge} and consider the wedge product property \eqref{awedgea}, it gives
\begin{equation}
\label{vh3}
\kdifformh{v}{3}=\sum_{i=1}^N\sum_{j=1}^N\sum_{k=1}^Nv_{i,j,k}e_i(\xi)e_j(\eta)e_k(\zeta),
\end{equation}
where
\begin{equation}
v_{i,j,k}=q^\xi_{i,j,k}-q^\xi_{i-1,j,k}+q^\eta_{i,j,k}-q^\eta_{i,j-1,k}+q^\zeta_{i,j,k}-q^\zeta_{i,j,k-1},
\end{equation}
can compactly be written as $\kcochain{v}{3}=\delta\kcochain{q}{2}$, or in matrix notation as $\bar{\psi}(\kcochain{v}{3})=\incidencederivative{3}{2}\bar{\psi}(\kcochain{q}{2})$. This relation is exact and invariant under transformations. It shows that $\ederiv\kdifformh{q}{2}=\ederiv\projection\kdifform{q}{2}=\ederiv\reconstruction\reduction\kdifform{q}{2}=
\reconstruction\delta\reduction\kdifform{q}{2}=\reconstruction\reduction\ederiv\kdifform{q}{2}=\projection\ederiv\kdifform{q}{2}=\projection\kdifform{v}{3}=\kdifformh{v}{3}$. In case $\kdifformh{q}{2}=\kdifformh{w}{2}=\ederiv\kdifformh{u}{1}$, then it can be shown that $\kdifformh{v}{3}=0$. Note that in terms of function spaces the vector proxy of $v_h^{(3)}$ is in $L^2$.
\end{example}

\begin{remark}
The expansion of $0$-forms $a^{(0)}_h\in\Lambda^0_h(\Omega;C_0)$ has $\mathcal{C}^0$ continuity over the sub-domain boundaries, $\partial \Omega_m$. For $a^{(1)}_h\in\Lambda^1_h(\Omega;C_1)$, the tangential component is $\mathcal{C}^0$ continuous and for $a^{(2)}_h\in\Lambda^2_h(\Omega;C_2)$ the normal component has $\mathcal{C}^0$ continuity over the sub-domain boundaries. The expansion of volume-forms, $a^{(3)}_h\in\Lambda^3_h(\Omega;C_3)$ are discontinuous over the sub-domain boundaries. 
\end{remark}

\begin{example}[\textbf{Trace}] Without loss of generality, consider one-dimensional domain $\Omega_{\rm ref}:\xi\in[-1,1]$. We can apply the trace on a zero-form $\kdifformh{a}{0}\in\Lambda^0_h(\Omega_{\rm ref};C_0)$ and a one-form $\kdifformh{b}{1}\in\Lambda^1_h(\Omega_{\rm ref};C_1)$, then we find:
\[
\begin{aligned}
\mathrm{tr}_{\partial\Omega_L}\kdifformh{a}{0}=\mathrm{tr}_{\partial\Omega_L}\left(\sum_{i=0}^Na_il_i(\xi)\right)&=-a_0,\\ \mathrm{tr}_{\partial\Omega_R}\kdifformh{a}{0}=\mathrm{tr}_{\partial\Omega_R}\left(\sum_{i=0}^Na_il_i(\xi)\right)&=+a_N,
\end{aligned}
\]
and
\[
\begin{aligned}
\mathrm{tr}_{\partial\Omega_L}\kdifformh{b}{1}=\mathrm{tr}_{\partial\Omega_L}\left(\sum_{i=1}^Nb_ie_i(\xi)\right)&=-\sum_{i=1}^Nb_ie_i(-1),\\ \mathrm{tr}_{\partial\Omega_R}\kdifformh{b}{1}=\mathrm{tr}_{\partial\Omega_R}\left(\sum_{i=1}^Nb_ie_i(\xi)\right)&=+\sum_{i=1}^Nb_ie_i(+1).
\end{aligned}
\]
The plus and minus signs in front indicate the corresponding orientation, see left in \figref{fig:dualcellcomplex1D}. So if we evaluate the trace on the left boundary point, where the orientation is {\em negative}, we obtain the {\em positive value} $a_0$ for the $0$-form and $+\sum_{i=1}^N b_i e_i(-1)$ for the $1$-form.
Because we consider tensor products to consider higher-dimensional $k$-forms, the trace on $k$-forms in higher dimensional domains can be constructed straightforward using the above relations.
\end{example}

\begin{example}[\textbf{Hodge-$\star$}] Now consider a 1D domain $\Omega_{\rm ref}: \xi=[-1,1]$. Let $\kdifformh{a}{0}(\xi)=\star\kdifformh{u}{1}(\xi)$. This example shows how to perform the Hodge-$\star$. The action of the Hodge-$\star$ is followed by a projection $\tilde{\pi}_M:\Lambda^0(\Omega_{\rm ref};C_1)\rightarrow\Lambda^1(\Omega_{\rm ref};\tilde{C}_0)$ to write the outcome in the preferred basis. Let $\kdifformh{u}{1}(\xi)$ be expanded in terms of 1-cochains and edge-functions,
\[
\kdifformh{u}{1}(\xi)=\sum_{i=1}^Nu_ie_i(\xi),\quad\mathrm{with}\ \kdifformh{u}{1}\in\Lambda^1_h(\Omega;C_1).
\]
Then apply the Hodge-$\star$ to get $\kdifformh{a}{0}(\xi)\in\Lambda^0_h(\Omega;C_1)$, as follows
\[
\kdifformh{a}{0}=\star_h\kdifformh{u}{1}=\sum_{i=1}^Nu_i\big(\star e_i(\xi)\big)=\sum_{i=1}^Nu_i\varepsilon_i(\xi)\big(\star\ederiv\xi\big)=\sum_{i=1}^Nu_i\varepsilon_i(\xi).
\]
Rewrite the basis using the action of $\tilde{\pi}_M$
\[
\kdifformh{a}{0}=\tilde{\pi}_M\kdifformh{a}{0}=\sum_{j=1}^N\left[\sum_{i=1}^Nu_i\varepsilon_i(\tilde{\xi}_j)\right]\tilde{l}^\mathrm{g}_j(\xi)=\sum_{j=1}^Na_j\tilde{l}^\mathrm{g}_j(\xi).
\]
In this case, what is usually called the Hodge-$\star$ matrix \cite{bochev2006principles,hiptmair2001discrete}, is given by $(\mathsf{H}^{0,1})_{i,j}=\ve_i(\tilde{\xi}_j)$. The action of this Hodge-$\star$ is illustrated in Figure \ref{fig:ex_hodge}.
\begin{figure}[h]
\centering
\includegraphics[width=0.8\textwidth]{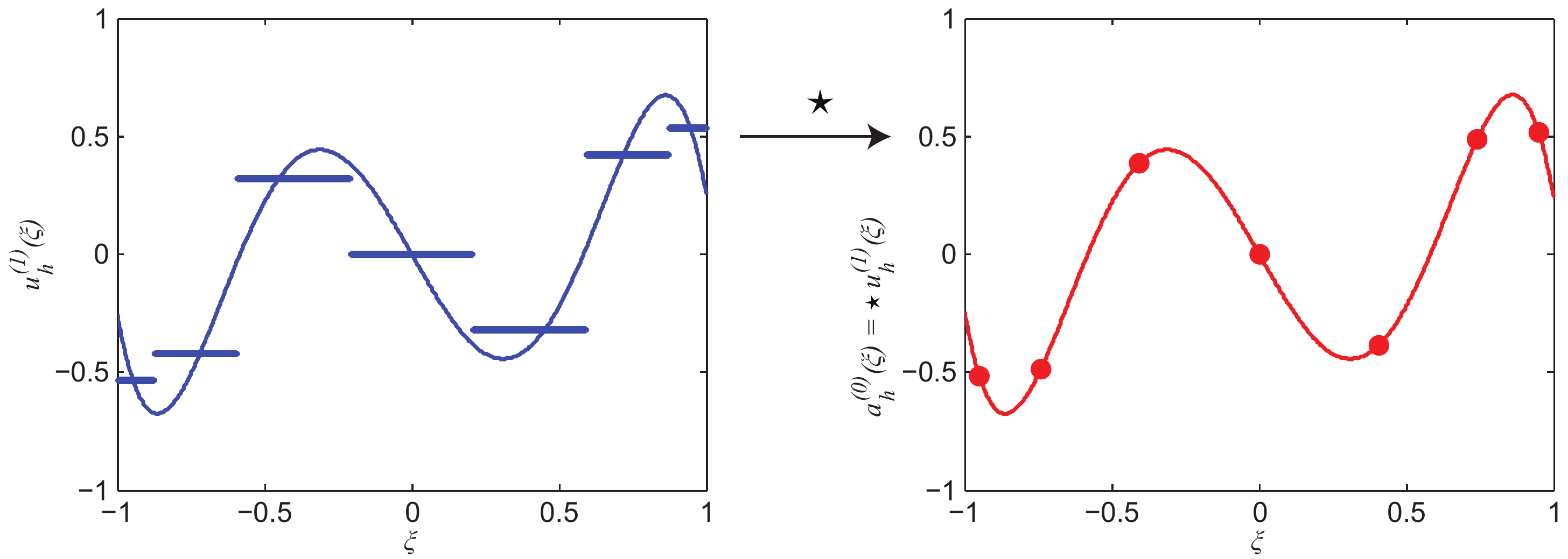}
\caption{Example of the action of the Hodge-$\star$ operator, followed by the projection $\pi_M$, $\pi_M\circ \star$. Left the function and corresponding cochain for $\kdifformh{u}{1}\in\Lambda^1(\Omega;C_1)$. Right the function and corresponding cochain for $\kdifformh{a}{0}\in\Lambda^0(\Omega;\tilde{C}_0)$.}
\label{fig:ex_hodge}
\end{figure}
\end{example}

\begin{example}[\textbf{Codifferential}] Again consider a one-dimensional domain $\Omega_{\rm ref}: \xi=[-1,1]$. We apply the codifferential on $\kdifform{u}{1}_h\in\Lambda^1_h(\Omega_{\rm ref};C_1)$ according to \eqref{eq::diffGeom_codifferential_definition}:
\[
\begin{aligned}
\coderiv\kdifformh{u}{1}&=\pi_M\coderiv\kdifformh{u}{1}=-\pi_M\star\ederiv\star\left(\sum_{i=1}^Nu_i\ve_i(\xi)\ederiv\xi\right)\\
&=-\pi_M\left(\sum_{i=1}^Nu_i\frac{\ederiv\ve_i(\xi)}{\ederiv\xi}\right)=\sum_{j=0}^N\left[\sum_{i=1}^Nu_i\left(-\sum_{k=0}^{i-1}\left.\frac{\ederiv^2l_k}{\ederiv\xi^2}\right|_{\xi=\xi_j}\right)\right]l_j(\xi).
\end{aligned}
\]
The codifferential matrix $\mathsf{D}^*$ is given by
\begin{equation}
\left(\mathsf{D}^*\right)_{i,j} = -\sum_{k=0}^{i-1}\left.\frac{\ederiv^2l_k}{\ederiv\xi^2}\right|_{\xi=\xi_j}.
\label{codifferentialmatrix}
\end{equation}
The coefficients in this codifferential matrix are independent of the location of the dual grid, as would be expected. Even when the codifferential matrix is constructed using two Hodge matrices and an incidence matrix, identically the same codifferential matrix is found. The same is true when retrieving the codifferential matrix from the formal adjoint formulation \eqref{eq::diffGeom_codiff_adjoint} when the $1$-form $kdifformh{u}{1}$ is zero at the boundary. Note that \eqref{codifferentialmatrix} requires that $l_i(\xi)$ is twice differentiable, while for the adjoint formulation it is sufficient that $l_i(\xi)$ is only one time differentiable.
\end{example}

\begin{example}[\textbf{Hodge decomposition}]
In this example we show the result of the discrete Hodge decomposition applied to a velocity field, $\kdifformh{v}{1}$, obtained from a potential flow problem on an annulus. The Hodge decomposition of the velocity field is given in terms of the gradient of the potential, $\kdifformh{\phi}{0}$, and a harmonic function, $h_{h}^{(1)}$,
\[
\kdifformh{v}{1}=\ederiv\kdifformh{\phi}{0}+h_{h}^{(1)}.
\]
For the annulus, consider the domain $(r,\theta) \in [1,R] \times [0,2\pi]$ and the cell complex $D$ which covers this domain as shown in Figure~\ref{fig:pot_flow_around_cylinder}. The topology of the cell complex is equal to the topology of the cell complex shown in Figure~\ref{fig:hole}.
\begin{figure}[h]
\centering
\includegraphics[width=0.55\textwidth]{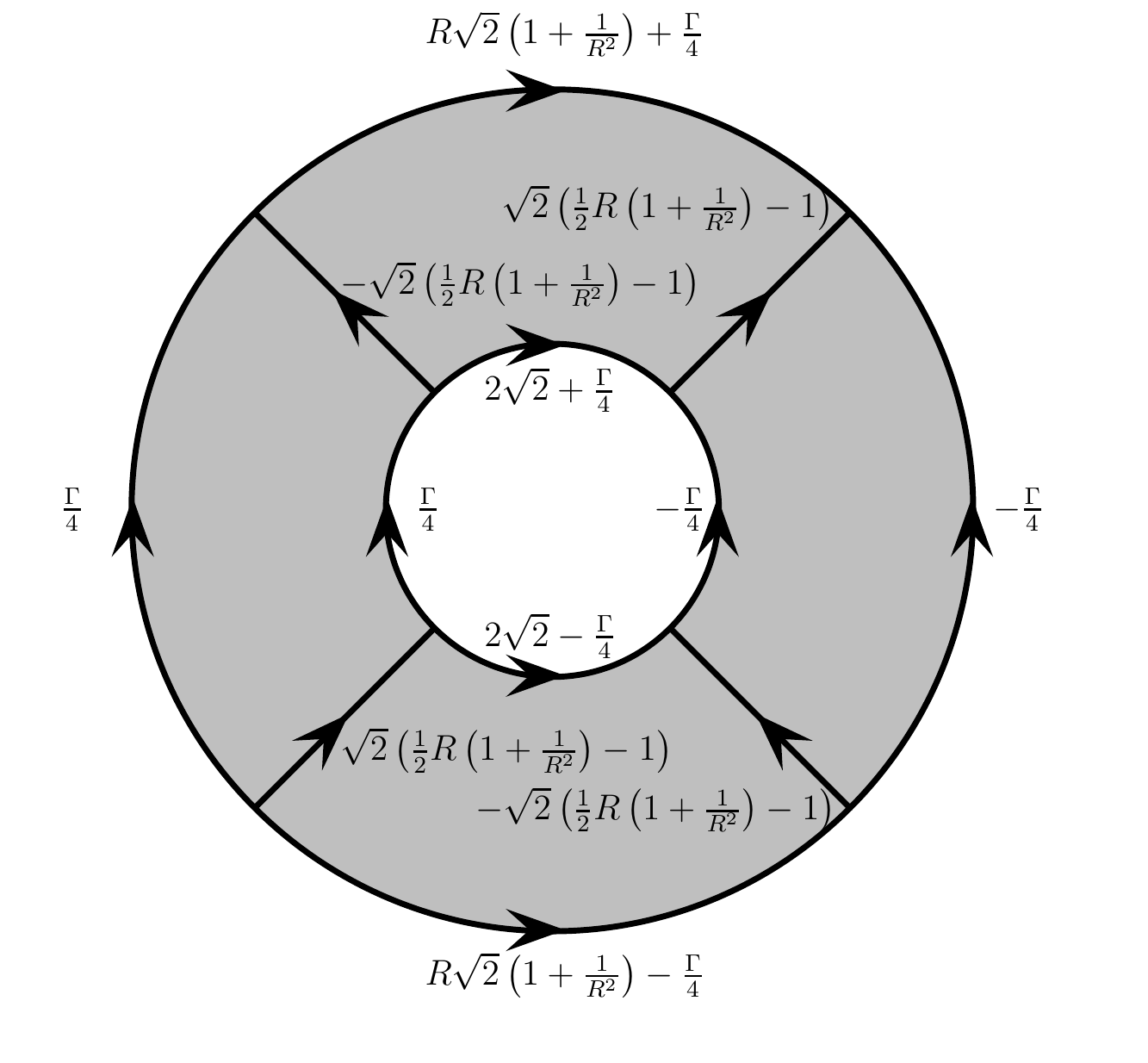}
\caption{The cell complex which covers the $(r,\theta) \in [1,R] \times [0,2\pi]$ and the $1$-cochain values associated with each $1$-cell.}
\label{fig:pot_flow_around_cylinder}
\end{figure}

The velocity 1-form on the annulus is given by
\begin{equation}
\kdifform{v}{1} = r\cos \theta \left (1 - \frac{1}{r^2} \right ) \,dr - \left [ \sin \theta \left ( 1 + \frac{1}{r^2} \right ) + \frac{\Gamma}{2\pi r} \right ] \, d\theta\;.\label{eq:pot_flow_around_cylinder}
\end{equation}
Application of the reduction map, $\reduction$, gives the $1$-cochain values associated with each $1$-chain in the cell complex. Its values are given in \figref{fig:pot_flow_around_cylinder}. If we apply the coboundary to this $1$-cochain we get the zero $2$-chain. Because the topology is the same as in Figure~\ref{fig:hole}, the harmonic $1$-chain is still given by
\[
\kchain{h}{1}=(1,-1,0,0,1,1,-1,1,-1,0,0,-1)^T\;,
\]
and therefore
\[
\left \langle \reduction \kdifform{v}{1} , \kchain{h}{1} \right \rangle = 2\Gamma,\quad\quad \left \langle \alpha\kcochain{h}{1},\kchain{h}{1} \right \rangle = 8\alpha\,,\quad\Rightarrow\quad\alpha=\frac{\Gamma}{4}.
\]
The harmonic cochain $\kcochain{h}{1}$ was determined in Example~\ref{ex:domain_with_hole_continued}. The values of the 1-cochain $\delta\kcochain{\boldsymbol\phi}{0} + \frac{\Gamma}{4}\kcochain{h}{1}$ are given in \figref{fig:pot_flow_around_cylinder}. The projection of the velocity field becomes,
\[ \projection \kdifform{v}{1} = \reconstruction\left(\delta\kcochain{\boldsymbol\phi}{0} + \frac{\Gamma}{4} \kcochain{h}{1}\right) \;.\]
\figref{fig:Phodge_decomposition_cylinder} shows the reconstruction of the velocity field components and the total velocity field, for different vortical strengths and different polynomial order.
\begin{figure}[h]
\centering
\includegraphics[width=1.\textwidth]{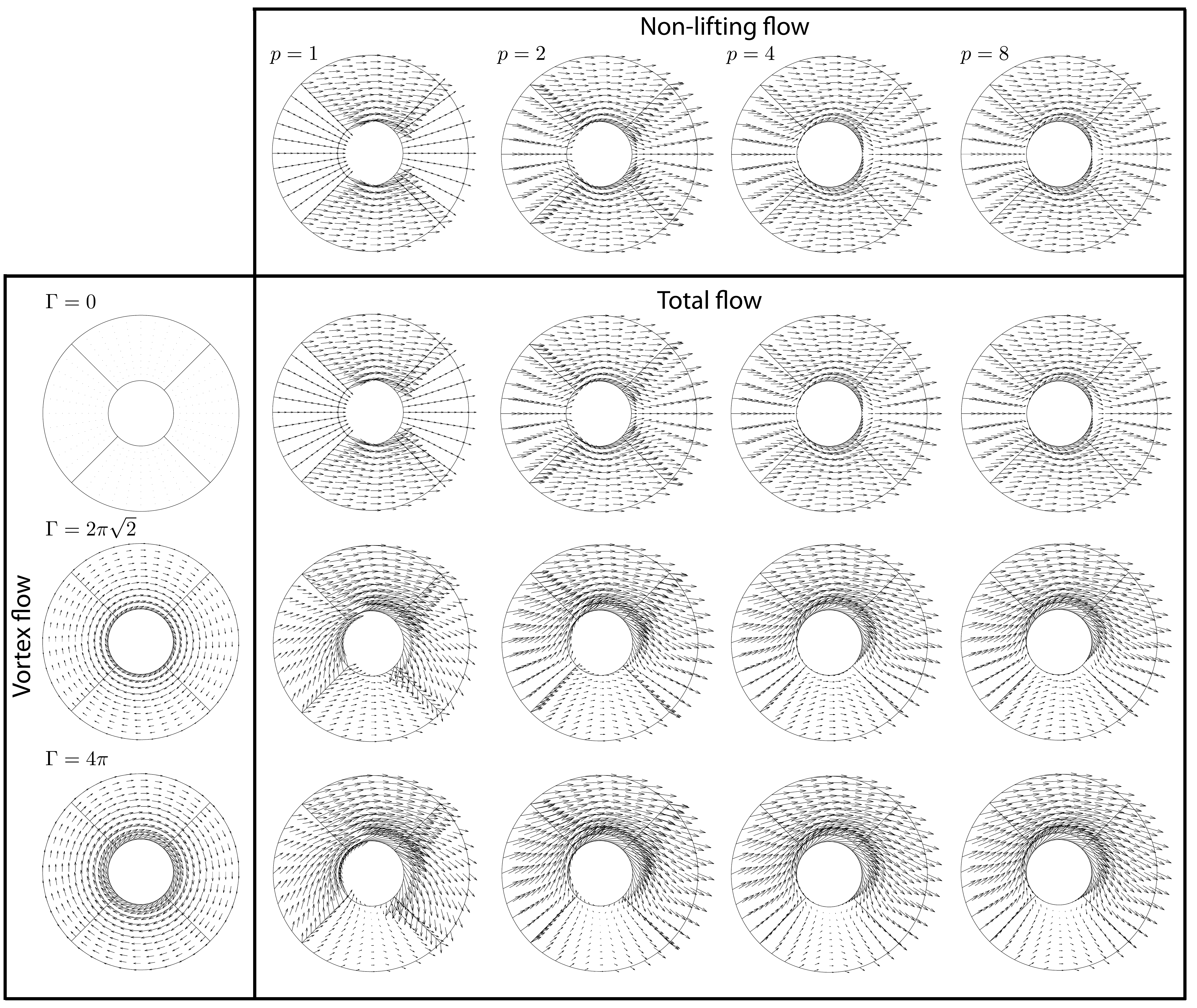}
\caption{Decomposition of velocity field. It shows the curl-free components (left), the harmonic components (top) and the total velocity field, for different vortical strengths and different polynomial order.}
\label{fig:Phodge_decomposition_cylinder}
\end{figure}

\end{example}

\subsection{Discussion}
Closely related to the presented mimetic basis-functions are (higher-order) Whitney forms, \cite{Whitney57,bossavit2002,hiptmair2001}. They share that there exists points, edges, faces and volumes (0,1,2,3-cells) both on the element boundary as well as in the element interior. These $k$-cells have physical relevance \cite{Rapetti2009}. The main difference exists in the way they are constructed.\\
On the other hand in mixed finite elements, Raviart-Thomas \cite{raviart606mixed} ($H({\rm div})$ conforming) and N\'{e}d\'{e}lec \cite{nedelec} ($H({\rm curl})$ conforming) elements are frequently used basis-functions. Main difference with our mimetic basis-functions is that these basis-functions are moment based polynomials, where each $k$-cell may have a higher-order 
moment degree, but the number of $k$-cells remains unchanged when increasing the polynomial order. A consequence is that Raviart-Thomas and N\'{e}d\'{e}lec elements allow only affine mappings, whereas the present mimetic basis-functions allow for curvilinear elements.\\
Where the mimetic basis-functions are constructed especially for quadrilateral and hexahedral elements, the Whitney and mixed finite element basis-functions also work on triangles and tetrahedrals. In case of the lowest order quadrilaterals or hexahedrals, all three types of basis-functions are the same.\\
As shown in Section~\ref{sec:AlgebraicTopology}, we need two cell-complexes to represent inner- and outer oriented variables. The use of such overlapping grids is well-known in staggered finite volume methods. A spectral staggered grid approach was presented by Kopriva and Kolias, \cite{KoprivaKolias}, and Kopriva, \cite{Kopriva1,Kopriva2}.

\section{Coda}\label{sec:Conclusions}
\textbf{Topological notions in PDE's}. In the mimetic framework we presented, building blocks are introduced, which mimic in a finite dimensional setting differential operators, which appear in PDE's as used in physics. By mimicking we mean that operators act similarly in a continuous space, $\Lambda^k$, to the operations in a finite dimensional space, $\Lambda^k_h$, or the discrete space, $C^k$. This is reflected by the many commuting diagrams presented throughout this paper. The representation of a finite dimensional solution is directly related to the underlying mesh. We made a clear distinction between topological relations that are metric free and exact, and metric relations that contain the constitutive relations and approximations. We extensively described the topology of the mesh, consisting of $k$-cells that have either inner or outer orientation, the basis of a cell complex. Topology and orientation are intrinsically related to differentiation as is indicated by the boundary and coboundary operators and their corresponding complexes. Key ingredient for exact differentiation is the generalized Stokes' Theorem. This idea is not new and is commonly referred to as \emph{finite volume methods}, where the discrete unknowns, i.e. the cochains, are integral quantities associated to geometric objects, obtained by the reduction map. In this way the action of differential operators is independent of the reconstruction map. Illustrations of this fact were given in Examples \ref{ex:discrete_gradient}, \ref{ex:discrete_curl_equation} and \ref{ex:discrete_divergence_equation}.

\textbf{Metric concepts in PDE's}. The existence of constitutive relations necessitates us to relate physical quantities associated to geometric objects of different dimension and different orientation. The one-to-one relation between physical quantities connected by constitutive relations motivates the existence of a dual grid, as is known from (staggered) finite volume methods. In dual grid methods, these metric relations are usually treated by a finite dimensional representation of the Hodge-$\star$ operator. This finite dimensional Hodge-$\star$ is not unique, but depends on the choice of reconstruction operator, in our case by arbitrary-order mimetic spectral element interpolation. In contrast, \emph{finite element methods} usually consider $L^2$-inner products as metric operator. The two are in fact intrinsically related using the wedge product according to Definition~\ref{def::diffGeom_hodge}. While in finite volume reconstruction is usually seen as interpolation, in finite elements the reconstruction functions are better known as basis-functions, ignoring the fact that these `functions' are differential forms which make the connection with the associated geometric objects.

\textbf{Single grid methods}. In terms of the framework presented in this paper, in the finite element method the weighting function `lives' on the same cell complex as the complex in which the equation is defined. The Hodge operator takes the weighting functions to the dual complex and then the wedge product is evaluated. So, although in finite element methods, one generally does not construct a dual cell complex, the dual cell complex is implicitly incorporated through the inner product. When we recognize that the weighting functions refer to the dual cell complex, it is also possible to make a more physical interpretation of integration by parts in finite element methods. Integration by parts transfers an equation from the dual complex to the primal complex (and the other way around).

\textbf{Dual boundaries}. Orientation not only introduces primal and dual cell complexes, it introduces primal and dual grids for both the domain $\Omega$ and the boundary $\partial\Omega$. This distinction is directly related to integration by parts. Especially for the dual complex, the boundary dual complex is often not recognized in staggered finite volumes methods, but instead additional degrees of freedom are generated ad hoc which are called `ghost points'.

\textbf{Hodge decomposition}. This framework presented more than just a nice analogy between finite volume and finite element methods. It has more structure. Most important is the Hodge decomposition and its finite dimensional and discrete counterparts. The Hodge decomposition decomposes differential form spaces into separate function spaces related to the exterior derivative. Mimicking the Hodge decomposition is essential for numerical stability. A key ingredient here is the commutation between projection operator and the exterior derivative.

\textbf{Extension to curvilinear grids}. Finally it has been shown that all relations hold for both Cartesian, as well as curvilinear grids, thanks to the preservation of the commutation relations of the pullback and cochain map with the various operators discussed in this paper. The topological relations are insensitive to changes in the topology as long as the grid connectivity remains unaltered, whereas the metric-dependent part (inner-product and wedge) do change when one maps the standard element to a curvilinear domain.\\


\bibliographystyle{abbrv}
\bibliography{./bibliography}

\appendix

\end{document}